\documentclass[11pt,a4paper,reqno,svgnames]{amsart}

\usepackage{latexsym}
\usepackage{youngtab}           
\usepackage{graphicx}             
\usepackage{color}
\usepackage{latexsym}
\usepackage{amsfonts}
\usepackage{amsxtra}
\usepackage{amsmath}
\usepackage{amssymb}
\usepackage{mathrsfs}
\usepackage{amsthm}
\usepackage{enumitem}
\usepackage{fullpage}
\usepackage{stmaryrd}
\usepackage{color}
\usepackage{booktabs}
\usepackage{lscape}
\usepackage{tikz}
\usepackage{xcolor}
\usepackage[plainpages=false, pdfpagelabels,  colorlinks=true, pdfstartview=FitV, linkcolor=DarkBlue, citecolor=DarkRed, urlcolor=blue]{hyperref}
\usepackage[color]{changebar}
\cbcolor{red}

\theoremstyle{plain}
\newtheorem{theorem}{Theorem}[section]
\newtheorem{proposition}[theorem]{Proposition}
\newtheorem{lemma}[theorem]{Lemma}
\newtheorem{corollary}[theorem]{Corollary}

\theoremstyle{definition}
\newtheorem{definition}[theorem]{Definition}
\newtheorem{example}[theorem]{Example}
\newtheorem{remark}[theorem]{Remark}
\numberwithin{equation}{section}

\DeclareMathOperator{\Shape}{Shape}
\DeclareMathOperator{\Res}{Res}

\newcommand{\half}{\frac{1}{2}}

\renewcommand{\ge}{\geqslant}

\renewcommand{\ge}{\geqslant}

\renewcommand{\le}{\leqslant}

\renewcommand{\unrhd}{\trianglerighteqslant}

\input xy
\xyoption{all}

\begin{document}
\title[Partition Algebras]
{A Seminormal form for Partition Algebras}

\author[J. Enyang]{John Enyang}
\address{Department of Mathematics and Statistics, University of Melbourne, Parkville VIC 3010, Australia}
\curraddr{School of Mathematics and Statistics F07, University of Sydney, NSW 2006, Australia}
\email{John.Enyang@sydney.edu.au}



\begin{abstract}
Using a new presentation for partition algebras (J.~Algebraic Combin.~37(3):401--454, 2013), we derive explicit combinatorial formulae for the seminormal representations of the partition algebras. These results generalise to the partition algebras the classical formulae given by Young for the symmetric group. 
\end{abstract}

\thanks{The author is grateful to Arun Ram for several helpful conversations related to this research, and to Fred Goodman for his collaboration in~\cite{EG:2012}. The author is also indebted to Andrew Mathas, Chris Bowman and Ge Li for comments on a previous version of this article, and to the two anonymous referees for their suggestions. This research was supported by the Australian Research Council (grant ARC DP--0986774) at the University of Melbourne and (grants ARC DP--0986349, ARC DP--110103451) at the University of Sydney.}

\maketitle


\section{Introduction}
The partition algebras $A_k(n)$, for $k,n\in\mathbb{Z}_{\ge0},$ are a family of algebras defined in the work of Martin~\cite{MR1103994, MR1265453, MR1768036} and Jones~\cite{MR1317365} in connection with the Potts model and higher dimensional statistical mechanics. Jones~\cite{MR1317365} showed that the even partition algebra $A_{2k}(n)$ is in Schur--Weyl duality with the symmetric group $\mathfrak{S}_n$ acting diagonally on the $k$--fold tensor product $V^{\otimes k}$ of its $n$--dimensional permutation representation $V$. The odd partition algebra $A_{2k-1}(n)$ was defined by Martin~\cite{MR1768036} as the centraliser of the subgroup $\mathfrak{S}_{n-1}\subseteq \mathfrak{S}_n$ acting on $V^{\otimes k}$. Including the  algebras $A_{2k-1}(n)$ in the tower 
\begin{align}\label{tower}
A_0(n)\subseteq A_1(n)\subseteq A_{2}(n)\subseteq \cdots 
\end{align}
allowed for the simultaneous analysis of the whole tower of algebras~\eqref{tower} using the Jones Basic construction by Martin~\cite{MR1768036} and Halverson and Ram~\cite{MR2143201}. The partition algebras have connections to Deligne's category $\underline{\rm{Re}}{\rm{p}}(S_t)$~\cite{MR2737787}, and are important examples of cellular algebras~\cite{MR1711582, MR1779601, MR2794027}. 

Halverson and Ram~\cite{MR2143201} used Schur--Weyl duality to show that a certain family of diagrammatically defined elements in the partition algebras commute and play an analogous role to the Jucys--Murphy elements in the group algebra of the symmetric group. Following Mathas~\cite[Sect.~3]{MR2414949}, a seminormal form for $A_k(n)$ may be defined as an irreducible matrix representation of $A_k(n)$ over $\mathbb{Q}$, relative to a basis of eigenvectors for the Jucys--Murphy family in $A_k(n)$. Seminormal forms appeared classically in Young's construction~\cite{Yo:1932} of irreducible representations for the symmetric group $\mathfrak{S}_k$ over $\mathbb{Q}$; a modern treatment of Young's construction realises a seminormal form for $\mathfrak{S}_k$ as an irreducible representation of $\mathfrak{S}_k$ by matrices over $\mathbb{Q}$, relative to a basis of eigenvectors for the Jucys--Murphy elements in $\mathbb{Q}\mathfrak{S}_k$ (see, for example~\cite{MR1444315, MR1711316, MR2165457, MR2050688}). 

Using the presentation for the partition algebras in~\cite{MR2143201, MR2811310}, Kosuda~\cite{MR2583260} has constructed seminormal representations for the partition algebra $A_{6}(n)$. The representations of the subalgebra of the partition algebra $A_{2k}(n)$ that acts as centraliser of $G(r,1,n)$ on the tensor space $V^{\otimes k}$, for $n\ge k$ and $r>k$, have been constructed by  Kosuda~\cite{MR2262344}. 

This paper provides explicit combinatorial formulae for the seminormal representations of each algebra $A_k(n)$ in the tower~\eqref{tower}, thereby solving a problem which was highlighted in~\cite{MR2143201}. For the representations of $A_k(n)$ which factor through the symmetric group, the formulae we derive coincide with those given by Young~\cite{Yo:1932}. The approach to the seminormal representations of the partition algebras in this paper exploits Graham and Lehrer's machinery of cellular algebras~\cite{MR1376244} and the interaction between cellularity and  Jucys--Murphy families in the work of  Mathas~\cite{MR2414949} and Goodman and Graber~\cite{MR2774622}. The main innovation in this paper therefore, is the application of the presentation of partition algebras in~\cite{MR3035512} to computing the seminormal representations. We remark that although it is possible to construct seminormal representations of the diagram algebras without appeal to the machinery of cellular algebras (see, for example, the work of Nazarov~\cite{MR1398116} and Leduc and Ram~\cite{MR1427801} on Brauer algebras), by using a Murphy type cellular basis~\cite{EG:2012}, and exploiting the inner product on cell modules, it is possible to derive explicit rational expressions for the off-diagonal structure constants associated with the contraction elements in the partition algebras (see Theorems~\ref{contractioneven} and~\ref{contractionodd}).   A Young orthogonal form has been given for the Brauer algebras by Nazarov~\cite{MR1398116}, for the BMW and Brauer algebras by Leduc and Ram~\cite{MR1427801}, and for cyclotomic Nazarov--Wenzl algebras by  Ariki, Mathas and Rui~\cite{MR2235339}. Thus, up to a normalisation of seminormal basis vectors, the main construction in this paper (see Sect.~\ref{m-r}) provides a partition algebra analogue to the results of~\cite{MR1398116, MR1427801, MR2235339}.

Nazarov introduced a remarkable recursion for special central elements in the Brauer algebras and established the relation between these central elements and seminormal representations ~\cite[Corollary~3.10, Proposition~4.2]{MR1398116}. A similar relation was established by Beliakova and Blanchet for the BMW algebras~\cite[Lemma~7.4]{MR1866492}. In Sect.~\ref{c-e-r-s}, we obtain analogous recursions for central elements in the partition algebras and explain the relation between these central elements and the seminormal representations of the partition algebras. 

A new recursion formula for the Gram determinant of the symmetric bilinear form~\cite{MR1376244} on cell modules of the partition algebras is derived in Sect.~\ref{det-proof}. This recursion for Gram determinants, which is analogous to the branching rule for discriminants associated with cell modules of Brauer algebras given by Rui and Si~\cite{MR2369064}, is used to derive the off-diagonal structure constants for the contraction elements in the partition algebras (see Propositions~\ref{offdiag:1} and~\ref{offdiag:2}). 

This paper concludes with tables of representing matrices for generators of the partition algebra $A_k(n)$ relative to seminormal bases for cell modules for $A_k(n)$, up to $k=6$.

\section{Preliminaries}\label{p-r-s}
\subsection{Combinatorics}\label{c-t-s}
Let $k$ denote a non--negative integer and $\mathfrak{S}_k$ be the symmetric group acting on ${\lbrace}1,\dots,k{\rbrace}$ on the right. For $i$ an integer, $1\le i<k$, let $s_i$ denote the transposition $(i,i+1)$. Then $\mathfrak{S}_k$ is presented as a Coxeter group by generators $s_1,s_2,\dots,s_{k-1}$, with the relations 
\begin{align*}
&s_i^2=1,&&\text{for $i=1,\ldots,k-1,$}\\
&s_is_j=s_js_i,&&\text{for $|i-j|\ge 2$.}\\
&s_is_{i+1}s_i=s_{i+1}s_is_{i+1},&&\text{for $i=1,\ldots,k-2$.}
\end{align*}
For $1\le i\le j\le k$, let
\begin{align*}
w_{i,j}=s_is_{i+1}\cdots s_{j-1}=(j,j-1,\ldots,i),
\end{align*}
and let $w_{j,i}=w_{i,j}^{-1}$. If $i=0$ or $j=0$, we define $w_{i,j}=1$. 

If $k\ge 0$, a \emph{partition} of $k$ is a weakly decreasing sequence $\lambda=(\lambda_1,\lambda_2,\dots)$ of non-negative integers such that $\sum_{r\ge 1}\lambda_r=k$. We denote by $\emptyset$ the unique partition of $0$. The notation $\lambda\vdash k$ indicates that $\lambda$ is a partition of $k$. If $\lambda$ is a partition, we will also write $|\lambda|=\sum_{r\ge 1}\lambda_r$. 

The \emph{diagram} of a partition $\lambda$ is the set
\begin{align*}
[\lambda]={\lbrace}(i,j)\,|\,\text{$\lambda_i\ge j\ge1$ and $i\ge
1$}\,{\rbrace}\subseteq \mathbb{N}\times\mathbb{N}.
\end{align*}
The elements of $[\lambda]$ are the \emph{nodes} of $\lambda$ and more generally a node is a pair $(i,j)\in\mathbb{N}\times\mathbb{N}$. The diagram $[\lambda]$ is traditionally represented as an array of boxes with $\lambda_i$ boxes on the $i$--th row. For example, if $\lambda=(3,2)$, then $[\lambda]=\text{\tiny\Yvcentermath1$\yng(3,2)$}$\,. We will usually identify the partition $\lambda$ with its diagram and write $\lambda$ in place of $[\lambda]$.

Let $\lambda$ be  a partition. A node $(i,j)$ is an \emph{addable} node of $\lambda$ if $(i,j)\not\in\lambda$ and 
$\mu=\lambda \cup{\lbrace}(i,j){\rbrace}$ is a partition; in this case $(i,j)$ is also referred to as a \emph{removable} node of $\mu$. We let $A(\lambda)$ and $R(\lambda)$ respectively denote the set of addable nodes and removable nodes of $\lambda$.

The dominance $\unrhd$ on partitions of $k$ is defined as follows: if $\lambda\vdash k$ and $\mu\vdash k$, then $\lambda\unrhd\mu$ if
\begin{align*}
\textstyle\sum_{r=1}^j\lambda_r\ge\sum_{r=1}^j\mu_r \qquad\text{for all $j\ge 1$.}
\end{align*}
We write $\lambda\rhd\mu$ to mean that $\lambda\unrhd\mu$ and $\lambda\ne\mu$. 

Let $\lambda\vdash k$. A $\lambda$--tableau $\mathfrak{t}$ is a bijective map between the nodes of  $\lambda$ and the integers $\{1,2,\dots,k\}$.
A  $\lambda$--tableau  can be represented by labelling the nodes of $\lambda$ with the integers $1,2,\dots,k$. For example, if $k=6$ and $\lambda=(3,2,1)$,
\begin{align}\label{tabex0.0}
\mathfrak{t}=\text{\tiny\Yvcentermath1$\young(146,23,5)$}
\end{align}
represents a $\lambda$--tableau. 
If $\lambda\vdash k$, let $\mathfrak{t}^\lambda$ denote the
$\lambda$--tableau in which $1,2,\dots,k$ are
entered in increasing order from left to right along the rows of
$[\lambda]$. Thus in the previous example where $k=6$ and
$\lambda=(3,2,1)$,
\begin{align}\label{tabex1}
\mathfrak{t}^\lambda=\text{\tiny\Yvcentermath1$\young(123,45,6)$}\,.
\end{align}
The tableau $\mathfrak{t}^\lambda$ is the \emph{row reading tableau} of shape $\lambda$. 
The symmetric group $\mathfrak{S}_k$ acts on the set of
$\lambda$--tableaux on the right by permuting
the integer labels of the nodes of $\lambda$. For example,
\begin{align*}
\text{\tiny\Yvcentermath1$\young(123,45,6)$}\,\text{\small$(2,4)(3,6,5)$}\,
=\text{\tiny\Yvcentermath1$\young(146,23,5)$} \,.
\end{align*}
If $\lambda\vdash k$, the \emph{Young subgroup} $\mathfrak{S}_\lambda$ is defined to be the
row stabiliser of $\mathfrak{t}^\lambda$ in $\mathfrak{S}_{k}$.
For instance, when $k=6$ and $\lambda=(3,2,1)$, as
in~\eqref{tabex1}, then $\mathfrak{S}_\lambda=\langle
s_1,s_2,s_4\rangle$. 
\subsection{Partition algebras}\label{p-a-s}
For $k=1,2,\ldots,$ define 
\begin{align*}
{P}_{2k}&={\lbrace}\text{set partitions of ${\lbrace}1,2,\dots,k,1',2',\dots,k'{\rbrace}$}{\rbrace},\qquad\text{and,}\\
{P}_{2k-1}&={\lbrace}d\in P_{2k}\mid\text{$k$ and $k'$ are in the same block of $d$}{\rbrace}.
\end{align*}
Any element $\rho\in{P}_{2k}$ may be represented as a graph with $k$ vertices in the top row, labelled from left to right, by $1,2,\dots,k$ and $k$ vertices in the bottom row, labelled, from left to right by $1',2',\dots,k'$, with vertex $i$ joined to vertex $j$ if $i$ and $j$ belong to the same block of $\rho$. The representation of a partition by a diagram is not unique; for example the partition 
\begin{align*}
\rho={\lbrace}{\lbrace}1,1',3,4',5',6{\rbrace},{\lbrace}2,2',3',4,5,6'{\rbrace}{\rbrace}\in P_{12}
\end{align*}
can be represented by the diagrams:
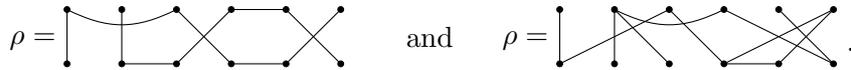
\begin{align*}
\begin{matrix}
\begin{tikzpicture}[scale=0.9]
\node at (-7.1,-0.45) {$\displaystyle \rho= \color{black}$};
\filldraw [black] (-6.6,0) circle (1.2pt);
\filldraw [black] (-6.6,-0.8) circle (1.2pt);
\filldraw [black] (-5.8,0) circle (1.2pt);
\filldraw [black] (-5.8,-0.8) circle (1.2pt);
\filldraw [black] (-5.0,0) circle (1.2pt);
\filldraw [black] (-5.0,-0.8) circle (1.2pt);
\filldraw [black] (-4.2,0) circle (1.2pt);
\filldraw [black] (-4.2,-0.8) circle (1.2pt);
\filldraw [black] (-3.4,0) circle (1.2pt);
\filldraw [black] (-3.4,-0.8) circle (1.2pt);
\filldraw [black] (-2.6,-0.0) circle (1.2pt);
\filldraw [black] (-2.6,-0.8) circle (1.2pt);
\draw (-6.6,0) -- (-6.6,-0.8);
\draw (-5.8,0) -- (-5.8,-0.8);
\draw (-5.8,-0.8) -- (-5.0,-0.8);
\draw (-6.6,0) .. controls (-6.0,-0.3) and (-5.6,-0.3) .. (-5.0,0.0);
\draw (-5.0,-0.0) -- (-4.2,-0.8);
\draw (-5.0,-0.8) -- (-4.2,-0.0);
\draw (-3.4,-0.0) -- (-4.2,-0.0);
\draw (-3.4,-0.8) -- (-4.2,-0.8);
\draw (-2.6,-0.8) -- (-3.4,-0.0);
\draw (-2.6,-0.0) -- (-3.4,-0.8);
\node at (-1.3,-0.40) {$\displaystyle \text{and} \color{black}$};
\node at (0.1,-0.45) {$\displaystyle \rho= \color{black}$};
\filldraw [black] (0.6,0) circle (1.2pt);
\filldraw [black] (0.6,-0.8) circle (1.2pt);
\filldraw [black] (1.4,0) circle (1.2pt);
\filldraw [black] (1.4,-0.8) circle (1.2pt);
\filldraw [black] (2.2,0) circle (1.2pt);
\filldraw [black] (2.2,-0.8) circle (1.2pt);
\filldraw [black] (3.0,0.0) circle (1.2pt);
\filldraw [black] (3.0,-0.8) circle (1.2pt);
\filldraw [black] (3.8,0.0) circle (1.2pt);
\filldraw [black] (3.8,-0.8) circle (1.2pt);
\filldraw [black] (4.6,0) circle (1.2pt);
\filldraw [black] (4.6,-0.8) circle (1.2pt);
\draw (0.6,0) -- (0.6,-0.8);
\draw (0.6,-0.8) -- (2.2,-0.0);
\draw (1.4,0) -- (1.4,-0.8);
\draw (1.4,0) -- (2.2,-0.8);
\draw (1.4,0) .. controls (2.0,-0.3) and (2.4,-0.3) .. (3.0,-0.0);
\draw (2.2,-0.0) -- (3.0,-0.8);
\draw (3.0,-0.8) -- (3.8,-0.8);
\draw (3.0,-0.8) -- (4.6,-0.0);
\draw (3.8,-0.8) -- (4.6,-0.0);
\draw (3.0,-0.0) -- (4.6,-0.8);
\draw (3.8,-0.0) -- (4.6,-0.8);
\end{tikzpicture}
\end{matrix}\ .
\end{align*}
If $\rho_1,\rho_2\in{P}_{2k}$, then the composition $\rho_1\circ\rho_2$ is the partition obtained by placing $\rho_1$ above $\rho_2$ and identifying each vertex in the bottom row of $\rho_1$ with the corresponding vertex in the top row of $\rho_2$ and deleting any components of the resulting diagram which contains only elements from the middle row. The composition product makes $P_{2k}$ into an associative monoid with identity
\begin{align*}
\begin{matrix}
\begin{tikzpicture}
\node at (-7.4,-0.45) {$\displaystyle 1= \color{black}$};
\node at (-5.35,-0.45) {$\displaystyle \cdots \color{black}$};
\draw (-6.8,0) -- (-6.8,-0.8);
\draw (-6,0) -- (-6,-0.8);
\draw (-4.7,0) -- (-4.7,-0.8);
\draw (-3.9,0) -- (-3.9,-0.8);
\filldraw [black] (-6.8,0) circle (1.2pt);
\filldraw [black] (-6.8,-0.8) circle (1.2pt);
\filldraw [black] (-6.0,0) circle (1.2pt);
\filldraw [black] (-6.0,-0.8) circle (1.2pt);
\filldraw [black] (-4.7,0) circle (1.2pt);
\filldraw [black] (-4.7,-0.8) circle (1.2pt);
\filldraw [black] (-3.9,0) circle (1.2pt);
\filldraw [black] (-3.9,-0.8) circle (1.2pt);
\end{tikzpicture}
\end{matrix}\ .
\end{align*}
Let $z$ be an indeterminate and $R=\mathbb{Z}[z]$. The partition algebra $\mathcal{A}_{2k}=\mathcal{A}_{2k}(z)$ is the $R$--module freely generated by $P_{2k}$, equipped with the product 
\begin{align*}
\rho_1\rho_2=z^{c}\rho_1\circ\rho_2,&&\text{for $\rho_1,\rho_2\in{P}_{2k}$,}
\end{align*}
where $c$ is the number of blocks removed from the middle row in constructing the composition $\rho_1\circ\rho_2$. Let $\mathcal{A}_{2k-1}=\mathcal{A}_{2k-1}(z)$ denote the subalgebra of $\mathcal{A}_{2k}$ generated by $P_{2k-1}$.  Halverson and Ram~\cite[Theorem~1.11]{MR2143201} and East~\cite[Theorem~36]{MR2811310} give a presentation for $\mathcal{A}_{2k}$ in terms of generators $e_1,e_2, \ldots,e_{2k-1},s_1, s_2,\ldots, s_{k-1}$, where 
\begin{align*}
\begin{matrix}
\begin{tikzpicture}
\node at (-7.6,-0.45) {$\displaystyle e_{2i-1}= \color{black}$};
\node at (-6.3,-0.45) {$\displaystyle \cdots \color{black}$};
\draw (-6.8,0) -- (-6.8,-0.8);
\draw (-5.8,0) -- (-5.8,-0.8);
\draw (-4.2,0) -- (-4.2,-0.8);
\draw (-3.2,0) -- (-3.2,-0.8);
\filldraw [black] (-6.8,0) circle (1.2pt);
\filldraw [black] (-6.8,-0.8) circle (1.2pt);
\filldraw [black] (-5.8,0) circle (1.2pt);
\filldraw [black] (-5.8,-0.8) circle (1.2pt);
\filldraw [black] (-5.0,0) circle (1.2pt);
\filldraw [black] (-5.0,-0.8) circle (1.2pt);
\filldraw [black] (-4.2,0) circle (1.2pt);
\filldraw [black] (-4.2,-0.8) circle (1.2pt);
\filldraw [black] (-3.2,0) circle (1.2pt);
\filldraw [black] (-3.2,-0.8) circle (1.2pt);
\node at (-3.7,-0.45) {$\displaystyle \cdots \color{black}$};
\node at (-6.8,-1.1) {$\displaystyle 1 \color{black}$};
\node at (-5.0,-1.1) {$\displaystyle i \color{black}$};
\node at (-3.2,-1.1) {$\displaystyle k \color{black}$};
\node at (-1.9,-0.45) {$\displaystyle \text{and} \color{black}$};
\node at (-0.0,-0.45) {$\displaystyle e_{2i}= \color{black}$};
\node at (1.1,-0.45) {$\displaystyle \cdots \color{black}$};
\draw (0.6,0) -- (0.6,-0.8);
\draw (1.6,0) -- (1.6,-0.8);
\draw (2.4,0) -- (3.2,0);
\draw (2.4,0) -- (2.4,-0.8);
\draw (2.4,-0.8) -- (3.2,-0.8);
\draw (3.2,0) -- (3.2,-0.8);
\draw (4.0,0) -- (4.0,-0.8);
\draw (5.0,0) -- (5.0,-0.8);
\filldraw [black] (0.6,0) circle (1.2pt);
\filldraw [black] (0.6,-0.8) circle (1.2pt);
\filldraw [black] (1.6,0) circle (1.2pt);
\filldraw [black] (1.6,-0.8) circle (1.2pt);
\filldraw [black] (2.4,0) circle (1.2pt);
\filldraw [black] (2.4,-0.8) circle (1.2pt);
\filldraw [black] (3.2,0) circle (1.2pt);
\filldraw [black] (3.2,-0.8) circle (1.2pt);
\filldraw [black] (4,0) circle (1.2pt);
\filldraw [black] (4,-0.8) circle (1.2pt);
\filldraw [black] (5,0) circle (1.2pt);
\filldraw [black] (5,-0.8) circle (1.2pt);
\node at (4.5,-0.45) {$\displaystyle \cdots \color{black}$};
\node at (0.6,-1.1) {$\displaystyle 1 \color{black}$};
\node at (2.4,-1.1) {$\displaystyle i \color{black}$};
\node at (3.3,-1.1) {$\displaystyle i+1 \color{black}$};
\node at (5,-1.1) {$\displaystyle k \color{black}$};
\end{tikzpicture}
\end{matrix}\ ,
\end{align*}
and
\begin{align*}
\begin{matrix}
\begin{tikzpicture}
\node at (-7.4,-0.5) {$\displaystyle {s_i}= \color{black}$};
\node at (-6.3,-0.45) {$\displaystyle \cdots \color{black}$};
\draw (-6.8,0) -- (-6.8,-0.8);
\draw (-5.8,0) -- (-5.8,-0.8);
\draw (-5,0) -- (-4.2,-0.8);
\draw (-5,-0.8) -- (-4.2,-0.0);
\draw (-3.4,0) -- (-3.4,-0.8);
\draw (-2.4,0) -- (-2.4,-0.8);
\filldraw [black] (-6.8,0) circle (1.2pt);
\filldraw [black] (-6.8,-0.8) circle (1.2pt);
\filldraw [black] (-5.8,0) circle (1.2pt);
\filldraw [black] (-5.8,-0.8) circle (1.2pt);
\filldraw [black] (-5.0,0) circle (1.2pt);
\filldraw [black] (-5.0,-0.8) circle (1.2pt);
\filldraw [black] (-4.2,0) circle (1.2pt);
\filldraw [black] (-4.2,-0.8) circle (1.2pt);
\filldraw [black] (-3.4,0) circle (1.2pt);
\filldraw [black] (-3.4,-0.8) circle (1.2pt);
\filldraw [black] (-2.4,0) circle (1.2pt);
\filldraw [black] (-2.4,-0.8) circle (1.2pt);
\node at (-2.9,-0.45) {$\displaystyle \cdots \color{black}$};
\node at (-6.8,-1.1) {$\displaystyle 1 \color{black}$};
\node at (-5.0,-1.1) {$\displaystyle i \color{black}$};
\node at (-4.1,-1.1) {$\displaystyle i+1 \color{black}$};
\node at (-2.4,-1.1) {$\displaystyle k \color{black}$};
\end{tikzpicture}
\end{matrix}\ .
\end{align*}
The elements $s_i=(i,i+1)$ are the Coxeter generators in $\mathfrak{S}_k$ and ${\lbrace}s_i\mid i=1,\ldots,k-1{\rbrace}$ generates the group ring $R\mathfrak{S}_k\subseteq \mathcal{A}_{2k}$.

Let $*:\mathcal{A}_{2k}\to \mathcal{A}_{2k}$ denote the algebra anti--involution which reflects each element of the diagram basis for $\mathcal{A}_{2k}$ in the horizontal axis. Then $e_{i}^*=e_{i}$ for $i=1,\ldots,2k-1$ and $w^*=w^{-1}$ for all $w\in\mathfrak{S}_k\subseteq \mathcal{A}_{2k}$. Restricting the map $*$ from $\mathcal{A}_{2k}$ to $\mathcal{A}_{2k-1}$, gives an algebra anti-involution of $\mathcal{A}_{2k-1}$ which we also denote by $*$.  
\subsection{Jucys--Murphy elements}\label{j-m-s}
Jucys--Murphy (or JM) elements for the partition algebras were defined by Halverson and Ram~\cite[Sect.~3]{MR2143201}. An equivalent recursive definition for partition algebra JM elements, given in~\cite[Sect.~3]{MR3035512}, yields a set of involutions in $\mathcal{A}_{2k}$ that are local in the sense of~\cite{MR1109059}. 
Let 
\begin{align*}
L_1=0,&& L_2=e_1,&& \sigma_{2}=1,&&\text{and}&&\sigma_3=s_1.
\end{align*}
For $i=1,2,\dots,$ define 
\begin{align}\label{jm-i-1}
L_{2i+2}=-s_iL_{2i}e_{2i}-e_{2i}L_{2i}s_i+ e_{2i}L_{2i}e_{2i+1}e_{2i} +s_{i}L_{2i}s_{i}+\sigma_{2i+1},
\end{align}
where, for $i=2,3,\dots,$
\begin{equation}
\begin{split}\label{sigma-i-1}
\sigma_{2i+1}=s_{i-1}s_i\sigma_{2i-1}s_i&s_{i-1} +s_ie_{2i-2}L_{i-1}s_ie_{2i-2}s_i+e_{2i-2}L_{2i-2}s_ie_{2i-2}\\
&-s_ie_{2i-2}L_{2i-2}s_{i-1}e_{2i}e_{2i-1}e_{2i-2}-e_{2i-2}e_{2i-1}e_{2i}s_{i-1}L_{2i-2}e_{2i-2}s_i,
\end{split}
\end{equation}
and,
\begin{align}\label{jm-i-2}
L_{2i+1}=-L_{2i}e_{2i}-e_{2i}L_{2i}+(z-L_{2i-1})e_{2i}+s_iL_{2i-1}s_i+\sigma_{2i},
\end{align}
where, for $i=2,3,\dots,$
\begin{equation}
\begin{split}\label{sigma-i-2}
\sigma_{2i}=s_{i-1}s_i\sigma_{2i-2}s_is_{i-1}+&e_{2i-2}L_{i-1}s_ie_{2i-2}s_i+s_ie_{2i-2}L_{i-1}s_ie_{2i-2}\\
-&e_{2i-2}L_{i-1}s_{i-1}e_{2i}e_{2i-1}e_{2i-2}-s_ie_{2i-2}e_{2i-1}e_{2i}s_{i-1}L_{i-1}e_{2i-2}s_i.
\end{split}
\end{equation}

It is shown in~\cite[Theorem~5.5]{MR3035512} that the family $\{L_i\mid i\ge1\}$ determined by the recursions above coincides with the family of JM elements defined by Halverson and Ram~\cite{MR2143201}. For $k\in\mathbb{Z}_{>0}$, we let $\mathscr{L}_k$ denote the JM subalgebra of $\mathcal{A}_k$ generated by $L_1,\ldots,L_k$. 

\begin{remark}
In~\cite{MR3035512}, the elements $\sigma_{2i}$ and $ \sigma_{2i+1}$  are respectively denoted by $\sigma_{i+\half}$ and $\sigma_{i+1}$, while $e_{2i}$ and $e_{2i+1}$ are denoted by $p_{i+\half}$ and $p_{i+1}$ respectively. 
\end{remark}
\begin{example}
In terms of the diagram basis,
\begin{align*}
&\begin{matrix}
\begin{tikzpicture}
\foreach \x in {0,0.8,1.6,...,11.6} 
\filldraw [black] (\x,0) circle (1.1pt);
\foreach \x in {0,0.8,1.6,...,11.6} 
\filldraw [black] (\x,-0.8) circle (1.1pt);
\foreach \x in {0.0, 1.6, 4.0,...,11.6} 
\draw (\x,0) -- (\x,-0.8);
\draw (0,-0.8) -- (0.8,-0.8);
\draw (2.4,-0.8) -- (2.4,-0.0) -- (3.2,-0.0); 
\draw (4.8,0.0) -- (5.6,0.0); 
\draw (4.8,-0.8) -- (5.6,-0.8); 
\draw (7.2,0) -- (7.2,-0.8);
\draw (9.6,0.0) -- (10.4,-0.8);
\draw (9.6,-0.8) -- (10.4,0.0);
\node at (-0.8,-0.4) {$\displaystyle L_4=- \color{black}$};
\node at (2.0,-0.4) {$\displaystyle - \color{black}$};
\node at (4.4,-0.4) {$\displaystyle + \color{black}$};
\node at (6.8,-0.4) {$\displaystyle + \color{black}$};
\node at (9.2,-0.4) {$\displaystyle + \color{black}$};
\end{tikzpicture}
\end{matrix}
\intertext{and}
&\begin{matrix}
\begin{tikzpicture}
\foreach \x in {0,0.8,1.6,...,11.6} 
\filldraw [black] (\x,0) circle (1.1pt);
\foreach \x in {0,0.8,1.6,...,11.6} 
\filldraw [black] (\x,-0.8) circle (1.1pt);
\foreach \x in {1.6,4.0,...,11.6} 
\draw (-0.8+\x,0) -- (\x,-0.8);
\foreach \x in {1.6,4.0,...,11.6} 
\draw (-0.8+\x,-0.8) -- (\x,-0.0);
\draw (0,-0.8) -- (0.8,-0.8);
\draw (2.4,-0.8) -- (2.4,-0.0) -- (3.2,-0.0); 
\draw (4.8,0.0) -- (5.6,0.0); 
\draw (4.8,-0.8) -- (5.6,-0.8); 
\draw (7.2,0) -- (7.2,-0.8);
\draw (9.6,0.0) -- (10.4,-0.8);
\draw (9.6,-0.8) -- (10.4,0.0);
\draw (0,0) -- (0,-0.8);
\node at (-0.8,-0.4) {$\displaystyle \sigma_5=- \color{black}$};
\node at (2.0,-0.4) {$\displaystyle - \color{black}$};
\node at (2.0,-0.4) {$\displaystyle - \color{black}$};
\node at (4.4,-0.4) {$\displaystyle + \color{black}$};
\node at (6.8,-0.4) {$\displaystyle + \color{black}$};
\node at (9.2,-0.4) {$\displaystyle + \color{black}$};
\end{tikzpicture}
\end{matrix}\,\,.
\end{align*}
\end{example}
\begin{example}
The elements $L_6$ and $\sigma_7$ may be represented respectively by the diagrams 
\begin{align*}
&\begin{matrix}
\begin{tikzpicture}
\foreach \x in {0,0.8,1.6,...,15.6} 
	{
	\foreach \y in {0,-1.2,...,-4.8}
		{
		\filldraw [black] (\x,\y) circle (1.1pt);
		}
	}
\foreach \x in {0,0.8,1.6,...,15.6} 
	{
	\foreach \y in {-0.8,-2.0,...,-5.6}
		{
		\filldraw [black] (\x,\y) circle (1.1pt);
		}
	}	
\foreach \x in {1.6,4.8,...,14.4} 
	{
	\foreach \y in {-0.8,-2.0,...,-5.6}
		{
		\draw (0.8+\x,\y) -- (0.8+\x,0.8+\y);
		}
	}	
\foreach \x in {0.8,4.0,...,13.6}
	{
	\draw (\x,-4.8) -- (0.8+\x,-5.6);
	}
\foreach \x in {0.8,4.0,...,13.6}
	{
	\draw (\x,-5.6) -- (0.8+\x,-4.8);
	}
\foreach \x in {0.8,4.0,...,13.6}
	{
	\draw (\x,-2.4) -- (0.8+\x,-2.4);
	}
\foreach \x in {0.8,4.0,...,13.6}
	{
	\draw (\x,-3.2) -- (0.8+\x,-3.2);
	}
\foreach \x in {0.8,4.0,...,13.6}
	{
	\draw (\x,-1.2) -- (0.8+\x,-1.2);
	}
\foreach \x in {0.8,4.0,...,13.6}
	{
	\draw (\x,-0.8) -- (0.8+\x,-0.8);
	}
\foreach \y in {-0,-1.2,...,-4.8}
	{
	\draw (0,\y) -- (0,-0.8+\y);
	}
\foreach \x in {0.8,4.0,...,13.6}
	{
	\draw (\x,0) -- (\x,-0.8);
	}
\foreach \x in {0.8,4.0,...,13.6}
	{
	\draw (\x,-1.2) -- (\x,-2.0);
	}
\foreach \x in {0.8,4.0,...,13.6}
	{
	\draw (\x,-3.6) -- (\x,-4.4);
	}
\foreach \y in {-0,-1.2,...,-4.8}
	{
	\draw (9.6,\y) -- (9.6,-0.8+\y);
	}
\node at (-0.3,-1.6) {$+\color{black}$};
\node at (-0.3,-2.8) {$-\color{black}$};
\node at (-0.3,-4.0) {$-\color{black}$};
\node at (-0.3,-5.2) {$+\color{black}$};
\node at (2.8,-0.4) {$+ \color{black}$};
\node at (2.8,-1.6) {$+ \color{black}$};
\node at (2.8,-2.8) {$- \color{black}$};
\node at (2.8,-4.0) {$- \color{black}$};
\node at (2.8,-5.2) {$+ \color{black}$};
\node at (6,-0.4) {$- \color{black}$};
\node at (6,-1.6) {$- \color{black}$};
\node at (6,-2.8) {$+ \color{black}$};
\node at (6,-4.0) {$+ \color{black}$};
\node at (6,-5.2) {$+ \color{black}$};
\node at (9.2,-0.4) {$- \color{black}$};
\node at (9.2,-1.6) {$+ \color{black}$};
\node at (9.2,-2.8) {$+ \color{black}$};
\node at (9.2,-4.0) {$+ \color{black}$};
\node at (9.2,-5.2) {$- \color{black}$};
\node at (12.4,-0.4) {$- \color{black}$};
\node at (12.4,-1.6) {$- \color{black}$};
\node at (12.4,-2.8) {$+ \color{black}$};
\node at (12.4,-4.0) {$+ \color{black}$};
\node at (12.4,-5.2) {$- \color{black}$};
\draw (0.0,0.0) -- (0.8,0.0);
\draw (3.2,0.0) -- (3.2,-0.8);
\draw (0.0,-0.8) -- (0.8,-0.8);
\draw (6.4,-0.8) -- (7.2,-0.8);
\draw (3.2,-1.2) -- (3.2,-2.0);
\draw (3.2,-1.2) -- (4.0,-1.2);
\draw (3.2,-2.0) -- (4.0,-2.0);
\draw (6.4,-1.2) -- (7.2,-1.2);
\draw (12.8,-0.0) -- (13.6,-0.0);
\draw (12.8,-2.0) -- (13.6,-2.0);
\draw (12.8,-0.8) -- (14.4,0);
\draw (12.8,-1.2) -- (14.4,-2.0);
\draw (12.8,-2.4) --(13.6,-3.2);
\draw (12.8,-3.2) --(13.6,-2.4);
\draw (0,-3.2) -- (0.8,-3.2);
\draw (3.2,-3.2) -- (3.2,-2.4) -- (4.0,-2.4);
\draw (6.4,-2.4) -- (7.2,-2.4);
\draw (6.4,-3.2) -- (7.2,-3.2);
\draw (3.2,-3.6) -- (3.2,-4.4);
\draw (3.2,-4.8) -- (4.0,-5.6);
\draw (3.2,-5.6) -- (4.0,-4.8);
\draw (6.4,-4.8) -- (7.2,-4.8);
\draw (6.4,-5.6) -- (7.2,-5.6);
\draw (12.8,-5.6) -- (12.8,-4.8) -- (13.6,-4.8);
\draw (9.6,-5.6) -- (10.4,-5.6);
\draw (3.2,0) .. controls (3.8,-0.3) and (4.2,-0.3) .. (4.8,-0.0);
\draw (6.4,0) .. controls (7.0,-0.3) and (7.4,-0.3) .. (8.0,-0.0);
\draw (0.0,-2.0) .. controls (0.6,-1.7)  and (1.0,-1.7).. (1.6,-2.0);
\draw (6.4,-2.0) .. controls (7.0,-1.7) and (7.4,-1.7) .. (8.0,-2.0);
\draw (0.0,-4.4) .. controls (0.6,-4.1) and (1.0,-4.1) .. (1.6,-4.4);
\draw (3.2,-3.6) .. controls (3.8,-3.9) and (4.2,-3.9) .. (4.8,-3.6);
\draw (6.4,-3.6) .. controls (7.0,-3.9) and (7.4,-3.9) .. (8.0,-3.6);
\draw (6.4,-4.4) .. controls (7.0,-4.1)  and (7.4,-4.1) .. (8.0,-4.4);
\draw (12.8,-3.6) .. controls (13.8,-3.95) .. (14.4,-4.4);
\draw (12.8,-4.4) .. controls (13.8,-4.05) .. (14.4,-3.6);
\end{tikzpicture}
\end{matrix}
\intertext{and}
&\begin{matrix}
\begin{tikzpicture}
\foreach \x in {0,0.8,1.6,...,15.6} 
	{
	\foreach \y in {0,-1.2,...,-4.8}
		{
		\filldraw [black] (\x,\y) circle (1.1pt);
		}
	}
\foreach \x in {0,0.8,1.6,...,15.6} 
	{
	\foreach \y in {-0.8,-2.0,...,-5.6}
		{
		\filldraw [black] (\x,\y) circle (1.1pt);
		}
	}	
\foreach \x in {1.6,4.8,...,14.4} 
	{
	\foreach \y in {-0.0,-1.2,...,-4.8}
		{
		\draw (\x,\y) -- (0.8+\x,-0.8+\y);
		}
	}	
\foreach \x in {1.6,4.8,...,14.4} 
	{
	\foreach \y in {-0.8,-2.0,...,-5.6}
		{
		\draw (\x,\y) -- (0.8+\x,0.8+\y);
		}
	}	
\foreach \x in {0.8,4.0,...,13.6}
	{
	\draw (\x,-4.8) -- (0.8+\x,-5.6);
	}
\foreach \x in {0.8,4.0,...,13.6}
	{
	\draw (\x,-5.6) -- (0.8+\x,-4.8);
	}
\foreach \x in {0.8,4.0,...,13.6}
	{
	\draw (\x,-2.4) -- (0.8+\x,-2.4);
	}
\foreach \x in {0.8,4.0,...,13.6}
	{
	\draw (\x,-3.2) -- (0.8+\x,-3.2);
	}
\foreach \x in {0.8,4.0,...,13.6}
	{
	\draw (\x,-1.2) -- (0.8+\x,-1.2);
	}
\foreach \x in {0.8,4.0,...,13.6}
	{
	\draw (\x,-0.8) -- (0.8+\x,-0.8);
	}
\foreach \y in {-0,-1.2,...,-4.8}
	{
	\draw (0,\y) -- (0,-0.8+\y);
	}
\foreach \x in {0.8,4.0,...,13.6}
	{
	\draw (\x,0) -- (\x,-0.8);
	}
\foreach \x in {0.8,4.0,...,13.6}
	{
	\draw (\x,-1.2) -- (\x,-2.0);
	}
\foreach \x in {0.8,4.0,...,13.6}
	{
	\draw (\x,-3.6) -- (\x,-4.4);
	}
\foreach \y in {-0,-1.2,...,-4.8}
	{
	\draw (9.6,\y) -- (9.6,-0.8+\y);
	}
\node at (-0.3,-1.6) {$+\color{black}$};
\node at (-0.3,-2.8) {$-\color{black}$};
\node at (-0.3,-4.0) {$-\color{black}$};
\node at (-0.3,-5.2) {$+\color{black}$};
\node at (2.8,-0.4) {$+ \color{black}$};
\node at (2.8,-1.6) {$+ \color{black}$};
\node at (2.8,-2.8) {$- \color{black}$};
\node at (2.8,-4.0) {$- \color{black}$};
\node at (2.8,-5.2) {$+ \color{black}$};
\node at (6,-0.4) {$- \color{black}$};
\node at (6,-1.6) {$- \color{black}$};
\node at (6,-2.8) {$+ \color{black}$};
\node at (6,-4.0) {$+ \color{black}$};
\node at (6,-5.2) {$+ \color{black}$};
\node at (9.2,-0.4) {$- \color{black}$};
\node at (9.2,-1.6) {$+ \color{black}$};
\node at (9.2,-2.8) {$+ \color{black}$};
\node at (9.2,-4.0) {$+ \color{black}$};
\node at (9.2,-5.2) {$- \color{black}$};
\node at (12.4,-0.4) {$- \color{black}$};
\node at (12.4,-1.6) {$- \color{black}$};
\node at (12.4,-2.8) {$+ \color{black}$};
\node at (12.4,-4.0) {$+ \color{black}$};
\node at (12.4,-5.2) {$- \color{black}$};
\draw (0.0,0.0) -- (0.8,0.0);
\draw (3.2,0.0) -- (3.2,-0.8);
\draw (0.0,-0.8) -- (0.8,-0.8);
\draw (6.4,-0.8) -- (7.2,-0.8);
\draw (3.2,-1.2) -- (3.2,-2.0);
\draw (3.2,-1.2) -- (4.0,-1.2);
\draw (3.2,-2.0) -- (4.0,-2.0);
\draw (6.4,-1.2) -- (7.2,-1.2);
\draw (12.8,-0.0) -- (13.6,-0.0);
\draw (12.8,-2.0) -- (13.6,-2.0);
\draw (12.8,-0.8) -- (14.4,0);
\draw (12.8,-1.2) -- (14.4,-2.0);
\draw (12.8,-2.4) --(13.6,-3.2);
\draw (12.8,-3.2) --(13.6,-2.4);
\draw (0,-3.2) -- (0.8,-3.2);
\draw (3.2,-3.2) -- (3.2,-2.4) -- (4.0,-2.4);
\draw (6.4,-2.4) -- (7.2,-2.4);
\draw (6.4,-3.2) -- (7.2,-3.2);
\draw (3.2,-3.6) -- (3.2,-4.4);
\draw (3.2,-4.8) -- (4.0,-5.6);
\draw (3.2,-5.6) -- (4.0,-4.8);
\draw (6.4,-4.8) -- (7.2,-4.8);
\draw (6.4,-5.6) -- (7.2,-5.6);
\draw (12.8,-5.6) -- (12.8,-4.8) -- (13.6,-4.8);
\draw (9.6,-5.6) -- (10.4,-5.6);
\draw (3.2,0) .. controls (3.8,-0.3) and (4.2,-0.3) .. (4.8,-0.0);
\draw (6.4,0) .. controls (7.0,-0.3) and (7.4,-0.3) .. (8.0,-0.0);
\draw (0.0,-2.0) .. controls (0.6,-1.7)  and (1.0,-1.7).. (1.6,-2.0);
\draw (6.4,-2.0) .. controls (7.0,-1.7) and (7.4,-1.7) .. (8.0,-2.0);
\draw (0.0,-4.4) .. controls (0.6,-4.1) and (1.0,-4.1) .. (1.6,-4.4);
\draw (3.2,-3.6) .. controls (3.8,-3.9) and (4.2,-3.9) .. (4.8,-3.6);
\draw (6.4,-3.6) .. controls (7.0,-3.9) and (7.4,-3.9) .. (8.0,-3.6);
\draw (6.4,-4.4) .. controls (7.0,-4.1)  and (7.4,-4.1) .. (8.0,-4.4);
\draw (12.8,-3.6) .. controls (13.8,-3.95) .. (14.4,-4.4);
\draw (12.8,-4.4) .. controls (13.8,-4.05) .. (14.4,-3.6);
\end{tikzpicture}
\end{matrix}
\end{align*}
\end{example}
\begin{theorem}[{\cite[Theorem~4.1]{MR3035512}}]\label{n-pres}
If $k=1,2,\ldots,$ then $\mathcal{A}_{2k}$ is the unital associative algebra presented by the generators $e_1,e_2,\ldots,e_{2k-1},\sigma_3,\sigma_4,\ldots,\sigma_{2k-1}$ 
and the relations:
\begin{enumerate}
\item (Involutions)\label{inv} $\sigma_i^2=1$ for $i=1,\ldots,2k-1$. 
\item (Braid relations) \label{brd}
\begin{enumerate}
\item $\sigma_{i}\sigma_{j}=\sigma_{j}\sigma_{i}$ if $|i-j|\ge 2$. \label{brd-1}
\item $s_is_{i+1}s_i=s_{i+1}s_is_{i+1}$, for $i=1,\ldots,k-2$, where \label{brd-4}
\begin{align*}
s_j=
\begin{cases}
\sigma_{2j+1},&\text{if $j=1$,}\\
\sigma_{2j}\sigma_{2j+1},&\text{if $1<j<k$,}
\end{cases}
\end{align*}
are the Coxeter generators for $\mathfrak{S}_k$.
\end{enumerate}
\item (Idempotent relations)\label{ide}
\begin{enumerate}
\item $e_{2i-1}^2=ze_{2i-1}$, for $i=1,\ldots,k$.\label{ide-1}
\item $e_{2i}^2=e_{2i}$, for $i=1,\ldots,k-1$.\label{ide-2}
\item $\sigma_{2i+1}e_{2i}=e_{2i}\sigma_{2i+1}=e_{2i}$, for $i=1,\ldots,k-1$.\label{ide-3}
\item $\sigma_{2i}e_{2i}=e_{2i}\sigma_{2i}=e_{2i}$, for $i=1,\ldots,k-1$. \label{ide-4}
\item $\sigma_{2i}e_{2i-1}e_{2i+1}=\sigma_{2i+1}e_{2i-1}e_{2i+1}$, for $i=1,\ldots,k-1$.\label{ide-5}
\item $e_{2i-1}e_{2i+1}\sigma_{2i}=e_{2i-1}e_{2i+1}\sigma_{2i+1}$, for $i=1,\ldots,k-1$.\label{ide-6}
\end{enumerate}
\item (Commutation relations)\label{com}
\begin{enumerate}
\item $e_ie_j=e_je_i$, if $|i-j|\ge 2$.\label{com-1}
\item $\sigma_ip_j=p_j\sigma_i$ if $j\ne i-1,i$. \label{com-4}
\item $\sigma_{2i-1}e_{2j}=e_{2j}\sigma_{2i-1}$, if $j\ne i$. \label{com-5}
\item $\sigma_{2i}e_{2j-1}=e_{2j-1}\sigma_{2i}$, if $j\ne i,i+1$. \label{com-6}
\item $\sigma_{2i}e_{2j+1}=e_{2j+1}\sigma_{2i}$, if $j\ne i-1$. \label{com-7}
\item $\sigma_{2i}e_{2i-1}\sigma_{2i}=\sigma_{2i+1}e_{2i+1}\sigma_{2i+1}$, for $i=1,\ldots,k-1$.\label{com-8}
\item $\sigma_{2i}e_{2i-2}\sigma_{2i}=\sigma_{2i-1}e_{2i} \sigma_{2i-1}$, for $i=2,\ldots,k-1$. \label{com-9}
\end{enumerate}
\item (Contractions) $e_ie_{i+1}e_i=e_i$ and $e_{i+1}e_ie_{i+1}=e_{i+1}$, for $i=1,\ldots,k-2$. 
\end{enumerate}
\end{theorem}
The involution $\sigma_{2i+1}$, which maps to $s_i$ under the quotient map $\mathcal{A}_{2i+2}\to R\mathfrak{S}_{i+1}$, may be regarded as a partition algebra analogue of the Coxeter generator $s_i\in\mathfrak{S}_{i+1}$. 
\begin{proposition}[{\cite[Sect.~3]{MR3035512}}]\label{a-c-r}
For $i=1,2,\ldots,$ the following statements hold:
\begin{enumerate}[label=(\arabic{*}), ref=\arabic{*},leftmargin=0pt,itemindent=1.5em]
\item $L_i\in\mathcal{A}_i$ and $\sigma_{i}\in\mathcal{A}_{i+1}$.\label{a-c-r-1}
 \item $L_{i}^*=L_i$ and $\sigma_{i}^*=\sigma_i$. \label{a-c-r-2}
\item $L_{i}$ and $\sigma_{i+1}$ commute with $\mathcal{A}_{i-1}(z)$.\label{a-c-r-10}
\item $s_{i+1}\sigma_{2i+1}e_{2i+2}=e_{2i}s_{i+1}\sigma_{2i+1}$.\label{a-c-r-4}
\item $\sigma_{2i}e_{2i-1}e_{2i}=L_{2i}e_{2i}$.\label{a-c-r-5}
\item $\sigma_{2i+1}e_{2i+1}e_{2i}=L_{2i}e_{2i}$.\label{a-c-r-6}
\label{a-c-r-9}
\item $(L_{i}+L_{i+1})e_{i}=e_{i}(L_{i}+L_{i+1})=z e_{i}$.\label{a-c-r-12}
\item $e_{2i+1}\sigma_{2i+1}e_{2i+1}=L_ie_{2i+1}$.\label{a-c-r-15}
\item $e_{2i+1}\sigma_{2i}e_{2i+1}=(z-L_{2i-1})e_{2i+1}$.\label{a-c-r-16}
\item $e_{2i+2}\sigma_{2i+1}e_{2i+2}=e_{2i}e_{2i+2}$.\label{a-c-r-17}
\end{enumerate}
\end{proposition}
The next statement translates the recursions~\eqref{jm-i-1} and~\eqref{jm-i-2} into the language of Theorem~\ref{n-pres}.
\begin{proposition}\label{t-0}
For $i=1,2,\ldots,$ the following statements hold:
\begin{enumerate}[label=(\arabic{*}), ref=\arabic{*},leftmargin=0pt,itemindent=1.5em]
\item $\sigma_{2i+1}L_{2i+2}-L_{2i}\sigma_{2i+1}=-L_{2i}e_{2i}-e_{2i}e_{2i+1}+e_{2i}L_{2i}e_{2i+1}e_{2i}+1$. \label{t-0-1}
\item $\sigma_{2i}L_{2i+1}-L_{2i-1}\sigma_{2i}=-e_{2i-1}e_{2i}-e_{2i}L_{2i}+(z-L_{2i-1})e_{2i}+1$.\label{t-0-2}
\item\label{j-n-1} $L_{2i+2}=-e_{2i+1}e_{2i}-e_{2i}e_{2i+1}+e_{2i}L_{2i}e_{2i+1}e_{2i}+\sigma_{2i+1}L_{2i}\sigma_{2i+1}+\sigma_{2i+1}$. 
\item\label{j-n-2} $L_{2i+1}=-e_{2i-1}e_{2i}-e_{2i}e_{2i-1}
+(z-L_{2i-1})e_{2i} +\sigma_{2i}L_{2i-1}\sigma_{2i}+\sigma_{2i}$.
\end{enumerate}
\end{proposition}
\begin{proof}
\eqref{t-0-1} The definition~\eqref{jm-i-1} and the relations 
$\sigma_{2i+1}s_i=s_i\sigma_{2i+1}=\sigma_{2i}$ and $\sigma_{2i+1}^2=1$ give 
\begin{align*}
s_iL_{2i+2}s_i&=-L_{2i}e_{2i}-e_{2i}L_{2i}+e_{2i}L_{2i}e_{2i+1}e_{2i}+L_{2i}+\sigma_{2i+1}
\end{align*}
and
\begin{align*}
s_iL_{2i+2}\sigma_{2i}&=-L_{2i}e_{2i}-e_{2i}L_{2i}\sigma_{2i+1}+e_{2i}L_{2i}e_{2i+1}e_{2i} +L_{2i}\sigma_{2i+1}+1.
\end{align*}
Since $L_{2i+2}$ commutes with $\sigma_{2i}\in\mathcal{A}_{2i+1}$, 
\begin{align*}
\sigma_{2i+1}L_{2i+2}&=-L_{2i}e_{2i}-e_{2i}L_i\sigma_{2i+1}+e_{2i}L_{2i}e_{2i+1}e_{2i} +L_{2i}\sigma_{2i+1}+1,
\end{align*}
and
\begin{align*}
\sigma_{2i+1}L_{2i+2}-L_{2i}\sigma_{2i+1}&=-L_{2i}e_{2i}-e_{2i}L_{2i}\sigma_{2i+1}+e_{2i}L_{2i}e_{2i+1}e_{2i}+1.
\end{align*}
Making the substitution $e_{2i}e_{2i+1}\sigma_{2i+1}=e_{2i}L_{2i}$, or $e_{2i}e_{2i+1}=e_{2i}L_{2i}\sigma_{2i+1}$ in the last expression gives the required statement.

\eqref{t-0-2} The definition~\eqref{jm-i-2} and the relations $\sigma_{2i+1}s_i=s_i\sigma_{2i+1}=\sigma_{2i}$ and $\sigma_{2i+1}^2=1$ give
\begin{align*}
\sigma_{2i}L_{2i+1}&=-\sigma_{2i}L_{2i}e_{2i}-e_{2i}L_{2i}
+(z-L_{2i-1})e_{2i}+\sigma_{2i+1}L_{2i-1}s_i+1.
\end{align*}
Noting that $\sigma_{2i+1}$ commutes with $L_{2i-1}\in\mathcal{A}_{2i-1}$, 
\begin{align*}
\sigma_{2i}L_{2i+1}-L_{2i-1}\sigma_{2i}&=-\sigma_{2i}L_{2i}e_{2i}-e_{2i}L_{2i}+(z-L_{2i-1})e_{2i}+1.
\end{align*}
Applying the relation $e_{2i}L_{2i}\sigma_{2i}=e_{2i}e_{2i-1}$  from Proposition~\ref{a-c-r} to the right hand side of the last equality gives the required statement.

\eqref{j-n-1} From Proposition~\ref{a-c-r}\eqref{a-c-r-6} and the relation $\sigma_{2i+1}^2=1$, we obtain $\sigma_{2i+1}L_{2i}e_{2i}=e_{2i+1}e_{2i}$. Since $\sigma_{2i+1}e_{2i}=e_{2i}$, the statement follows from item~\eqref{t-0-1}.

\eqref{j-n-2} We proceed as in the proof of~\eqref{j-n-1}, using the relation $e_{2i}L_{2i}\sigma_{2i}=e_{2i}e_{2i-1}$.
\end{proof}

\section{A Cellular Basis}\label{m-b-s}
Introduced by Graham and Lehrer as a device for studying non--semisimple representations of a class of algebras that includes Hecke algebras, Schur algebras and Brauer algebras~\cite{MR1376244}, cellular algebras are defined by the existence of a \emph{cellular basis} and a \emph{cell datum}, which have combinatorial properties analogous to the ``Robinson--Schensted correspondence" in the group algebra of the symmetric group. Cellularity of partition algebras was established by Xi using a tangle type basis~\cite[Theorem~4.1]{MR1711582}. The construction of seminormal representations will rely on the Murphy type cellular bases for partition algebras given in~\cite{EG:2012} and we will adhere to the conventions and notation established in that paper. For further details on the subject of cellular algebras in general, we refer the reader to~\cite{MR1376244, MR1711316, MR2414949, MR2794027,MR2774622}.
   
For $i=0,1,\ldots,$ let $\widehat{H}_{2i}=\widehat{H}_{2i+1}=\lbrace \lambda\mid \lambda\vdash i\rbrace$, and 
\begin{align*}
\hat{A}_{2i}=\hat{A}_{2i+1}={\lbrace}(\lambda,l)\mid\lambda\in \widehat{H}_{i-2l}, \text{ for }l=0,1,\ldots,\lfloor i/2\rfloor{\rbrace}.
\end{align*}
Let $\widehat{H}$ denote the graph with 
\begin{enumerate}
\item vertices on level $i$: $\widehat{H}_i$, 
\item if $i$ is even and $\lambda\in\widehat{H}_i$, an edge $\lambda\to \mu$ in $\widehat{H}$, for $\mu\in\widehat{H}_{i+1}$, if $\lambda=\mu$, and 
\item if $i$ is odd and $\lambda\in\widehat{H}_i$, an edge $\lambda\to \mu$ in $\widehat{H}$, for $\mu\in\widehat{H}_{i+1}$, if $\mu$ is obtained from $\lambda$ by adding a node. 
\end{enumerate}
Similarly, build a graph $\hat{A}$, with
\begin{enumerate}
\item vertices on level $i$: $\hat{A}_i$,
\item if $i$ is even and $(\lambda,l)\in\hat{A}_i$, an edge $(\lambda,{l})\to(\mu,m)$ in $\hat{A}$, for $(\mu,m)\in\hat{A}_{i+1}$, if 
\begin{enumerate}
\item $(\lambda,l)=(\mu,m)$, or 
\item $m=l+1$ and $\mu$ is obtained from $\lambda$ by removing a node, and 
\end{enumerate}
\item if $i$ is odd and $(\lambda,l)\in\hat{A}_i$, an edge $(\lambda,{l})\to(\mu,m)$ in $\hat{A}$, for $(\mu,m)\in\hat{A}_{i+1}$, if 
\begin{enumerate}
\item $(\lambda,l+1)=(\mu,m)$, or 
\item $l=m$ and $\mu$ is obtained from $\lambda$ by adding a node. 
\end{enumerate}
\end{enumerate}

The first few levels of $\hat{A}$ are given in~\eqref{d-i-a}.
\begin{align}\label{d-i-a}
\begin{matrix}
\xymatrix{
0&\emptyset\ar[d]      & & & & & &\\
&\emptyset \ar[d]\ar[dr] & & & & & &\\
2&\emptyset\ar[d]    &\text{\tiny\Yvcentermath1$\yng(1)$}\ar[d]\ar[dl] & & & & &\\
&\emptyset\ar[d]\ar[dr]&\text{\tiny\Yvcentermath1$\yng(1)$}\ar[d]\ar[dr]\ar[drr] & & & & &\\
4&\emptyset\ar[d] &\text{\tiny\Yvcentermath1$\yng(1)$}\ar[d]\ar[dl] &\text{\tiny\Yvcentermath1$\yng(2)$}\ar[d]\ar[dl] & \text{\tiny\Yvcentermath1$\yng(1,1)$}\ar[d]\ar[dll] & & & \\
&\emptyset \ar[d]\ar[dr] &\text{\tiny\Yvcentermath1$\yng(1)$} \ar[d]\ar[dr]\ar[drr]&\text{\tiny\Yvcentermath1$\yng(2)$} \ar[d]\ar[drr]\ar[drrr]& \text{\tiny\Yvcentermath1$\yng(1,1)$}\ar[d]\ar[drr]\ar[drrr] & & &\\
6& \emptyset & \text{\tiny\Yvcentermath1$\yng(1)$} &\text{\tiny\Yvcentermath1$\yng(2)$} & \text{\tiny\Yvcentermath1$\yng(1,1)$} &\text{\tiny\Yvcentermath1$\yng(3)$} & \text{\tiny\Yvcentermath1$\yng(2,1)$} &\text{\tiny\Yvcentermath1$\yng(1,1,1)$}
}
\end{matrix}
\end{align}
The graph $\hat{A}$ is the semisimple branching diagram~\cite[Theorem~2.24]{MR2143201} for the tower of algebras~\eqref{tower}. The next definition~\cite[Sect.~2]{MR2143201} extracts from $\hat{A}$ a set of paths that are partition algebra analogues to the standard tableaux in the representation theory of the symmetric group.

\begin{definition}
Let $k\in\mathbb{Z}_{\ge0}$ and $(\lambda,0)\in\hat{A}_k$. A path to level $k$ of \emph{shape} $(\lambda,l)$ in $\hat{A}$ is a sequence
\begin{align}\label{p-d-e}
\mathfrak{t}=\big((\lambda^{(0)},l_0),(\lambda^{(1)},l_1),\ldots,(\lambda^{(k)},l_k)),
\end{align}
where $(\lambda^{(0)},{l}_0)=(\emptyset,0)$, $(\lambda^{(k)},l_k)=(\lambda,l)$ and $(\lambda^{(i-1)},{l}_{i-1})\to (\lambda^{(i)},{l}_{i})$ is an edge in $\hat{A}$, for $i=1,\ldots,k$. We write $\Shape(\mathfrak{t})=(\lambda,l)$ and denote 
\begin{align*}
\hat{A}_k^{(\lambda,{l})}={\big\lbrace}
\mathfrak{s}\mid \text{$\mathfrak{s}$ is a path to level $k$ and $\Shape(\mathfrak{s})=(\lambda,l)$}{\big\rbrace}.
\end{align*}  
For $\mathfrak{t}\in\hat{A}_k^{(\lambda,l)}$, a path of the form~\eqref{p-d-e}, let $\mathfrak{t}^{(i)}=(\lambda^{(i)},l_i)$ for $i=0,\ldots,k$.
\end{definition}

\begin{definition}
Assume that $(\lambda,l)\in\hat{A}_k$ and $\mathfrak{t}=\big( (\lambda^{(0)},l_0),(\lambda^{(1)},l_1)\ldots,(\lambda^{(k)},l_k))\in\hat{A}_k^{(\lambda,l)}$. If $r=0,1,\ldots,k$, we denote $\mathfrak{t}^{(r)}=(\lambda^{(r)},l_r)$ and define the truncation of $\mathfrak{t}$ to level $r$ to be the path
\begin{align*}
\mathfrak{t}\downarrow _r=\big( (\lambda^{(0)},l_0),(\lambda^{(1)},l_1)\ldots,(\lambda^{(r)},l_r)).
\end{align*}
\end{definition}

For $i=1,2,\ldots,$ let
\begin{align*}
e_{i-1}^{(l)}=
\begin{cases}
1, &\text{if $l=0$,}\\
\underbrace{e_{i-{2l}+1}e_{i-2l+3}\cdots e_{i-1}}_{\text{${l}$ factors}},&\text{if $l=1,\ldots,\lfloor i/2\rfloor$,}\\
0,&\text{if $l>\lfloor i/2\rfloor$.}
\end{cases}
\end{align*}
If $(\lambda,0)\in\hat{A}_{k-2l}$ and $(\lambda,l)\in\hat{A}_k$, define
\begin{align*}
c_{(\lambda,0)}= \sum_{w\in\mathfrak{S}_\lambda}w &&\text{and}&& c_{(\lambda,l)}= c_{(\lambda,0)}e_{k-1}^{(l)}. 
\end{align*}
Occasionally, to avoid ambiguity, it will be necessary to emphasise the fact that $(\lambda,l)\in\hat{A}_k$ and $c_{(\lambda,l)}\in \mathcal{A}_k$. In these circumstances, we will write $c_{(\lambda,l)}^{(k)}$ in place of $c_{(\lambda,l)}$. 

If $\lambda\in\widehat{H}_{2i-1}$ and $\mu\in\widehat{H}_{2i}$, where $\lambda\to\mu$ in $\widehat{H}$ and $\mu=\lambda\cup\lbrace(j,\mu_j)\rbrace$, let $a=\sum_{r=1}^{j}\mu_j$, and define
\begin{align*}
\bar{u}_{\lambda\to\mu}^{(2i)}=w_{i,a}\sum_{r=0}^{\lambda_j} w_{a,a-r}, && \text{and} && \bar{d}_{\lambda\to\mu}^{\,(2i)}=w_{a,i}. 
\end{align*}
For each $\mu\to\mu$, from level $2i$ to level $2i+1$ in $\widehat{H}$, we define the elements $\bar{d}^{\,(2i+1)}_{\mu\to\mu}=\bar{u}^{\,(2i+1)}_{\mu\to\mu}=1$. 
If $(\lambda,l)\in\hat{A}_{2i-1}$, $(\mu,m)\in\hat{A}_{2i}$, and $(\lambda,l)\to(\mu,m)$ in $\hat{A}$, define 
\begin{align*}
d_{(\lambda,l)\to(\mu,m)}^{(2i)}=
\begin{cases}
\bar{d}_{\lambda\to\mu}^{\,(2i-2l)}e_{2i-2}^{(l)},&\text{if $\lambda\subsetneq\mu$ and $l=m$, }\vspace{3pt}\\
e_{2i-2}^{(l)},&\text{if $\lambda=\mu$ and $l=m-1$.}
\end{cases}
\end{align*}
Similarly, if $(\lambda,l)\in\hat{A}_{2i}$, $(\mu,m)\in\hat{A}_{2i+1}$ and $(\lambda,l)\to(\mu,m)$ in $\hat{A}$, then 
\begin{align*}
d_{(\lambda,l)\to(\mu,m)}^{(2i+1)}=
\begin{cases}
\bar{u}_{\lambda\to\mu}^{(2i-2l)}e_{2i-1}^{(l)},&\text{if $\mu\subsetneq\lambda$ and $l=m-1$, }\vspace{3pt}\\
e_{2i-1}^{(l)},&\text{if $\lambda=\mu$ and $l=m$.}
\end{cases}
\end{align*}
For $\mathfrak{t}=\big( (\lambda^{(0)},l_0),(\lambda^{(1)},l_1)\ldots,(\lambda^{(i)},l_i))\in\hat{A}_{i}^{(\lambda,l)}$, let 
\begin{align*}
d_\mathfrak{t}= d_{(\lambda^{(i-1)},l_{i-1})\to (\lambda^{(i)},l_i)}^{(i)}d_{(\lambda^{(i-2)},l_{2-1})\to (\lambda^{(i-1)},l_{i-1})}^{(i-1)}\cdots d_{(\lambda^{(1)},l_{1})\to (\lambda^{(0)},l_{0})}^{(1)}.
\end{align*}
Occasionally, when it is necessary to emphasise that  $\mathfrak{t}\in\hat{A}_k^{(\lambda,l)}$ and $d_\mathfrak{t}\in \mathcal{A}_k$, we will write $d^{(k)}_\mathfrak{t}$ in place of $d_\mathfrak{t}$. 

\begin{definition}
If $(\lambda,l),(\mu,m)\in\hat{A}_i$, write $(\lambda,{l})\unrhd(\mu,m)$ if either ${l}>m$, or $l=m$ and $\lambda\unrhd\mu$. 
\end{definition}
\begin{theorem}[{\cite[Theorem~5.29]{EG:2012}}]\label{murphybasisthm}
If $i=1,2,\ldots,$ the set
\begin{align}\label{murphybasis}
\mathscr{A}_i=
\left\lbrace 
m_\mathfrak{st}=d_\mathfrak{s}^*c_{(\lambda,l)}d_\mathfrak{t}\mid \mathfrak{s},\mathfrak{t}\in\hat{A}_i^{(\lambda,l)}\text{ and } (\lambda,l)\in\hat{A}_i
\right\rbrace
\end{align}
is an $R$--basis for $\mathcal{A}_i$, and $(\mathcal{A}_i,*,\hat{A},\unrhd,\mathscr{A}_i)$ is a cell datum for $\mathcal{A}_i$. 
\end{theorem}
We recall for later reference some basic properties of the bases~\eqref{murphybasis} which flow from the Jones basic construction and the theory of cellular algebras (for details, the reader is referred to~\cite{MR1376244,  MR2774622, MR2794027, EG:2012}).

For $(\lambda,l)\in\hat{A}_k$, define the two sided ideal
\begin{align*}
\mathcal{A}_{k}^{\rhd(\lambda,{l})}=
\sum_{(\mu,m)\rhd(\lambda,{l})} \mathcal{A}_k c_{(\mu,m)}\mathcal{A}_k
\end{align*}
where the sum is taken over $(\mu,m)\in\hat{A}_k$ such that $(\mu,m)\rhd(\lambda,{l})$. 
\begin{lemma}\label{cellaction}
Let $(\lambda,l)\in\hat{A}_k$ and $\mathfrak{s,t}\in\hat{A}_k$. If $p\in\hat{A}_k$, then 
\begin{align}\label{rightaction}
m_\mathfrak{st}p\equiv \sum_{\mathfrak{u}}r_\mathfrak{u}m_\mathfrak{su}\mod \mathcal{A}_k^{\rhd(\lambda,l)},
\end{align}
where $r_\mathfrak{u}\in R$, for $\mathfrak{u}\in\hat{A}_k^{(\lambda,l)}$, depend only on $p$ and $\mathfrak{t}$. 
\end{lemma}

\begin{definition}\label{cellmodule}
Let $(\lambda,l)\in\hat{A}_k$. The \emph{cell module} $\Delta^{(\lambda,l)}_{k}$ is the right $\mathcal{A}_k$--module  with $R$--basis $\lbrace m_\mathfrak{t}\mid\mathfrak{t}\in\hat{A}_k^{(\lambda,l)}\rbrace$ and right $\mathcal{A}_k$ action
\begin{align*}
m_\mathfrak{t}p\equiv \sum_{\mathfrak{u}}r_\mathfrak{u}m_\mathfrak{u}, \quad\text{for $p\in\mathcal{A}_k$,}
\end{align*}
where $r_\mathfrak{u}\in R$, for $\mathfrak{u}\in\hat{A}_k^{(\lambda,l)}$, are determined by~\eqref{rightaction}. 
\end{definition}
 Let $(\lambda,l)\in\hat{A}_k$ and $\mathfrak{s,t}\in\hat{A}_k^{(\lambda,l)}$. The map $\langle\,,\rangle:\Delta^{(\lambda,l)}_{k}\times \Delta^{(\lambda,{l})}_{k}\to R$, defined by 
\begin{align}\label{f-d-1}
\langle m_\mathfrak{u}, m_\mathfrak{v}\rangle
m_\mathfrak{st}\equiv 
m_\mathfrak{su} m_\mathfrak{vt} \mod A_{k}^{\rhd(\lambda,l)},\qquad\text{for $\mathfrak{u},\mathfrak{v}\in\hat{A}_{k}^{(\lambda,l)}$,}
\end{align}
is a symmetric bilinear form on $\Delta^{(\lambda,l)}_{k}$. It is shown in ~\cite[Sect.~2]{MR1376244} that the form defined by~\eqref{f-d-1} is independent of the choice of $\mathfrak{s}$ and $\mathfrak{t}$. 
The bilinear form~\eqref{f-d-1} may be used in principle to classify the irreducible representations of partition algebras over a field~\cite[Theorem~(3.4)]{MR1376244}. In Sect~\ref{det-proof} we give an explicit recursion for the form~\eqref{f-d-1} and use these results to write down expressions for the off-diagonal structure constants for each contraction $e_i$ relative to seminormal bases for cell modules. 

The salient feature of the construction~\cite{EG:2012} of the basis for $\Delta^{(\lambda,l)}_k$ in Definition~\ref{cellmodule} is that it realises an explicit  filtration of the restriction of  $\Delta^{(\lambda,l)}_k$ to $\mathcal{A}_{k-1}$  by cell modules for $\mathcal{A}_{k-1}$. 
\begin{proposition}[{\cite[Lemma~3.13]{EG:2012}}]\label{filtration}
Assume that $(\lambda,l)\in\hat{A}_k$ and $(\rho,r)\in\hat{A}_{k-1}$, where $(\rho,r)\to(\lambda,l)$ in $\hat{A}$. Let $N^{\unrhd(\rho,r)}\subseteq \Delta_{k}^{(\lambda,l)}$ and $N^{\rhd(\rho,r)}\subseteq \Delta_{k}^{(\lambda,l)}$ denote the $\mathcal{A}_{k-1}$-submodules respectively generated by 
\begin{align*}
&\lbrace m_\mathfrak{t}\in \Delta_{k}^{(\lambda,l)}\mid \Shape(\mathfrak{t}\downarrow_{k-1})\unrhd(\rho,r)\rbrace
\qquad\text{and}\qquad
\lbrace m_\mathfrak{t}\in \Delta_{k}^{(\lambda,l)}\mid \Shape(\mathfrak{t}\downarrow_{k-1})\rhd(\rho,r)\rbrace.
\end{align*}
Then the linear map $N^{\unrhd(\rho,r)}/N^{\rhd(\rho,r)} \to \Delta^{(\rho,r)}_{k-1}$ given by 
\begin{align*}
m_\mathfrak{s}+N^{\rhd(\rho,r)}&\mapsto m_\mathfrak{t}, \qquad\text{if $\mathfrak{s}\in\hat{A}_k^{(\lambda,l)},$ $\mathfrak{t}\in\hat{A}^{(\rho,r)}$ and  $\mathfrak{s}\downarrow_{k-1}=\mathfrak{t},$}
\end{align*}
is an isomorphism of $\mathcal{A}_{k-1}$-modules.
\end{proposition}
Goodman and Graber~\cite[Definition~2.16]{MR2774622} introduced the following order on paths in the general context of JM elements in coherent towers of cellular algebras. 
\begin{definition}
Let $\mathfrak{s}=((\lambda^{(0)},l_0),\ldots,(\lambda^{(k)},l_k))$ and $\mathfrak{t}=((\mu^{(0)},m_0),\ldots,(\mu^{(k)},m_k))$ be paths in $\hat{A}$. We say that $\mathfrak{s}$ \emph{precedes} $\mathfrak{t}$ in \emph{reverse lexicographic order} (denoted $\mathfrak{t}\succcurlyeq\mathfrak{s}$) if $\mathfrak{s}=\mathfrak{t}$, or if for the last index $i$ such that $(\lambda^{(i)},l_i)\ne (\mu^{(i)},m_i)$ we have $(\lambda^{(i)},l_i)\unrhd (\mu^{(i)},m_i)$. Let $\mathfrak{t}\succ\mathfrak{s}$ denote the fact that $\mathfrak{t}\succcurlyeq\mathfrak{s}$ and $\mathfrak{s}\ne \mathfrak{t}$. 
\end{definition}
\begin{lemma}
If $(\lambda,l)\in\hat{A}_k$, then there exists a unique element $\mathfrak{t}^{(\lambda,l)}\in \hat{A}_k^{(\lambda,l)}$ which is maximal with respect to the reverse lexicographic order $\succcurlyeq$ on on $\hat{A}_k^{(\lambda,l)}$. 
\end{lemma}
\begin{proof}
Define a sequence $\mathfrak{t}^{(\lambda,l)}=((\lambda^{(0)},l_0),\ldots,(\lambda^{(k)},l_k))$ by: 
\begin{enumerate}
\item $(\lambda^{(2i)},l_{2i})=(\emptyset,i)$, for $i=0,1,\ldots,l,$
\item  $(\lambda^{(2i+2)},l_{2i+2})= (\lambda^{(2i)}\cup \lbrace\alpha_{i}\rbrace,l_{2i}) $ for $l\le i$ and $i<\lfloor k/2\rfloor$, where the node $\alpha_i$ has minimal row index (or maximal column index) in $A(\lambda^{(2i)})\cap\lambda$, and 
\item $(\lambda^{(2i+1)},l_{2i+1})=(\lambda^{(2i)},l_{2i})$, for $0\le i\le \lfloor k/2\rfloor$.
\end{enumerate}
Then $\mathfrak{t}^{(\lambda,l)}$ is manifestly an element of $\hat{A}_k^{(\lambda,l)}$, which is unique with respect to the property of being maximal under the reverse lexicographic order on $\hat{A}_k^{(\lambda,l)}$. 
\end{proof}

It will be convenient to write the basis $\lbrace m_\mathfrak{t}\mid \mathfrak{t}\in\hat{A}_k^{(\lambda,l)}\}$ for $\Delta^{(\lambda,l)}_{k}$ relative to the maximal element $m_{\mathfrak{t}^{(\lambda,l)}}$.  For this purpose we define the following branching coefficients. 

If $(\lambda,l)\in\hat{A}_{2i-1}$, $(\mu,m)\in\hat{A}_{2i}$, and $(\lambda,l)\to(\mu,m)$ in $\hat{A}$, let 
\begin{align*}
p^{(2i)}_{(\lambda,l)\to(\mu,m)}=
\begin{cases}
w_{m,i}, &\text{if $\lambda=\mu$ and $l=m-1$,}\\
w_{a,i}, &\text{if $\mu=\lambda\cup\lbrace (j,\mu_j)\rbrace$ and $l=m$,}
\end{cases}
\end{align*}
where, in the second case, we have written $a=l+\sum_{r=1}^{j}\mu_r$.  Similarly, if $(\lambda,l)\in\hat{A}_{2i}$, $(\mu,m)\in\hat{A}_{2i+1}$, and $(\lambda,l)\to(\mu,m)\in\hat{A}$, let 
\begin{align*}
p^{(2i+1)}_{(\lambda,l)\to(\mu,m)}=
\begin{cases}
1,&\text{if $\lambda=\mu$ and $l=m$,}\\
w_{m,i}e_{2i}w_{i,a}\sum_{r=0}^{\mu_j}w_{a,a-r},& \text{if $\lambda=\mu\cup\lbrace(j,\lambda_j)\rbrace$ and $l=m-1$,}
\end{cases}
\end{align*}
where, in the second case, we have written $a=l+\sum_{r=1}^{j}\lambda_r$. 

For $\mathfrak{t}=\big( (\lambda^{(0)},l_0),(\lambda^{(1)},l_1)\ldots,(\lambda^{(i)},l_i))\in\hat{A}_{i}^{(\lambda,l)}$, let 
\begin{align*}
p_\mathfrak{t}= p_{(\lambda^{(i-1)},l_{i-1})\to (\lambda^{(i)},l_i)}^{(i)}p_{(\lambda^{(i-2)},l_{2-1})\to (\lambda^{(i-1)},l_{i-1})}^{(i-1)}\cdots p_{(\lambda^{(1)},l_{1})\to (\lambda^{(0)},l_{0})}^{(1)}.
\end{align*}
When when it is necessary to emphasise that $p_\mathfrak{t}\in \mathcal{A}_k$, for $\mathfrak{t}\in\hat{A}_k^{(\lambda,l)}$, we will write $p^{(k)}_\mathfrak{t}$ in place of $p_\mathfrak{t}$. We require the following distinguished element of $\hat{A}^{(\lambda,l)}$. 
\begin{definition}
Assume that $(\lambda,l)\in\hat{A}_k$ and 
\[
(\rho,l) =\min\lbrace(\mu,m)\in \hat{A}_{k-1}\mid (\mu,m)\to (\lambda,l)\text{ in }\hat{A}\rbrace,
\]
where the $\min$ is taken with respect to $\unrhd$.  Denote by $\mathfrak{u}=\mathfrak{u}^{(\lambda,l)}$ the element in $\hat{A}^{(\lambda,l)}$ defined by the condition that   $\mathfrak{u}\downarrow_{k-2l}=\mathfrak{t}^{(\lambda,0)}$ and 
\begin{align*}
\mathfrak{u}^{(k-2l+i)}=
\begin{cases}
(\rho,r),&\text{if $i=2r+1$,}\\
(\lambda,r),&\text{if $i=2r$,}
\end{cases}
\end{align*}
for $r=1,2,\ldots,l$. 
\end{definition}
\begin{lemma}\label{rewrite}
Assume that $(\lambda,l)\in\hat{A}_k$. Then the following statements hold:
\begin{enumerate}
\item\label{rewrite:1} $c_{(\lambda,l)}d_{\mathfrak{u}^{(\lambda,l)}} =c_{(\lambda,l)}$. 
\item\label{rewrite:1.1}  $d_{\mathfrak{t}^{(\lambda,l)}}=e_{k-2}^{(l)} e_{k-3}^{(l)}\cdots e_{2l-1}^{(l)}$. 
\item\label{rewrite:2} $p_{\mathfrak{t}^{(\lambda,l)}}=1$ and $m_\mathfrak{t}=m_{\mathfrak{t}^{(\lambda,l)}}p_\mathfrak{t}$ for all $\mathfrak{t}\in\hat{A}_k^{(\lambda,l)}$.
\end{enumerate}
\end{lemma}
\begin{proof}
\eqref{rewrite:1} 
Let $\mathfrak{u}=\mathfrak{u}^{(\lambda,l)}$ and $\rho=\min\lbrace\mu\in \widehat{H}_{k-2l-1}\mid \mu\to \lambda\text{ in }\widehat{H}\rbrace$. If $l=1$, the definitions give  $d^{(k)}_{\mathfrak{u}^{(k-1)}\to\mathfrak{u}^{(k)}} =d^{(k)}_{(\rho,0)\to(\lambda,1)}=1$. Since $\mathfrak{u}\downarrow_{k-1}=\mathfrak{t}^{(\rho,0)}$, we have $d_{\mathfrak{u}^{(i-1)}\to\mathfrak{u}^{(i)}}=1$ for $i=1,\ldots,k-1$, and  $d_\mathfrak{u}^{(k)}=1$ when $l=1$. If $l>1$, then 
\begin{align*}
d_{\mathfrak{u}^{(k-1)}\to \mathfrak{u}^{(k)}}^{(k)}= d^{(k)}_{(\rho,l-1)\to(\lambda,l)}= e_{k-2}^{(l-1)} \qquad\text{and}\qquad  d_{\mathfrak{u}^{(k-2)}\to \mathfrak{u}^{(k-1)}}^{(k-1)}= d^{(k)}_{(\lambda,l-1)\to(\rho,l-1)}= e_{k-3}^{(l-1)}.
\end{align*}
Thus, using the contraction relations $e_{i+1}e_ie_{i+1}=e_{i+1}$, 
\begin{align*}
c_{(\lambda,l)}d_{\mathfrak{u}^{(k-1)}\to \mathfrak{u}^{(k)}}^{(k)} d_{\mathfrak{u}^{(k-2)}\to \mathfrak{u}^{(k-1)}}^{(k-1)} =c_{(\lambda,0)}e_{k-1}^{(l)} e_{k-2}^{(l-1)}e_{k-3}^{(l-1)}=c_{(\lambda,0)}e_{k-1}^{(l)} =c_{(\lambda,l)}. 
\end{align*}
Proceeding by induction on $l$, we obtain $c_{(\lambda,l)}d_{\mathfrak{u}^{(\lambda,l)}} =c_{(\lambda,l)}$. 

\eqref{rewrite:1.1} Let us write $\mathfrak{u}^{(\lambda,l)}=((\lambda^{(0)},l_0),\ldots,(\lambda^{(k)},l_k))$. Then, $d_{(\lambda^{(k-2l-1)},l)\to(\lambda^{(k-2l)},l)}^{(k)}=e_{k-2}^{(l)}$, and if $k$ is odd,
\begin{align*}
c_{(\lambda,l)}^{(k)}d_{(\lambda^{(k-2l-1)},l)\to(\lambda^{(k-2l)},l)}^{(k)}=e_{k-1}^{(l)}c_{(\lambda,l)}^{(k-1)}, 
\end{align*}
while, if $k=2i$ and $\lambda=\lambda^{(k-2l-1)}\cup\lbrace (j,\lambda_j)\rbrace$, 
\begin{align*}
c_{(\lambda,l)}^{(k)}d_{(\lambda^{(k-2l-1)},l)\to(\lambda^{(k-2l)},l)}^{(k)}={\displaystyle\sum_{r=0}^{\lambda_{j}-1} w_{i-l-r,i-l}e_{k-1}^{(l)} c_{(\lambda^{(k-2l-1)},l)}^{(k-1)}}.
\end{align*}
Therefore, using part~\eqref{rewrite:1}, we may define a path $\mathfrak{s}\in\hat{A}_{k}^{(\lambda,l)}$ by $m_\mathfrak{s}=m_{\mathfrak{u}^{(\lambda,l)}}e_{k-2}^{(l)}$, with the property that $\mathfrak{s}\downarrow_{k-1}=\mathfrak{u}^{(\lambda^{(k-2l-1)},l)}$. By induction on $k$, we have $m_\mathfrak{s}e_{k-3}^{(l)}e_{k-4}^{(l)}\cdots e_{2l-1}^{(l)}=m_{\mathfrak{t}^{(\lambda,l)}}$. 

\eqref{rewrite:2} If  $\mathfrak{t}=\mathfrak{t}^{(\lambda,l)}$, the definitions show that $p^{(i)}_{\mathfrak{t}^{(i-1)}\to\mathfrak{t}^{(i)}}=1$ for $i=1,\ldots,k$. Thus $p_{\mathfrak{t}^{(\lambda,l)}}=1$. We now prove that, for all $\mathfrak{t}\in\hat{A}_k^{(\lambda,l)}$, 
\begin{align}\label{cbasis} c_{(\lambda,l)}d_\mathfrak{t} =c_{(\lambda,l)}d_{\mathfrak{t}^{(\lambda,l)}}p_\mathfrak{t}.
\end{align}
If $l=0$, then $d_{\mathfrak{t}^{(\lambda,l)}}=1$. Thus, if $(\lambda^{(i-1)},0)\in\hat{A}_{i-1}$ and $(\lambda^{(i)},0)\in\hat{A}_i$, where $(\lambda^{(i-1)},0)\to (\lambda^{(i)},0)$ in $\hat{A}$, the definitions give $d_{(\lambda^{(i-1)},0)\to \lambda^{(i)},0)}^{(i)}= p_{(\lambda^{(i-1)},0)\to (\lambda^{(i)},0)}^{(i)}$. Hence~\eqref{cbasis} is true for all $\mathfrak{t}\in\hat{A}_k^{(\lambda,0)}$. In particular, we have shown that~\eqref{cbasis} holds for all $\mathfrak{t}\in\hat{A}_1^{(\emptyset,0)}$. 

Now assume that $l>0$ and let $\mathfrak{t}=((\lambda^{(0)},l_0),\ldots,(\lambda^{(k)},l_k)) \in\hat{A}_k^{(\lambda,l)}$. Write  $\mathfrak{s} =\mathfrak{t}\downarrow_{k-1}$. We use induction on $k$ to show that~\eqref{cbasis} holds for $\mathfrak{t}$. The cases $k=2i$ and $k=2i+1$ are treated separately below. 

\textsc{Case 1.} Assume that $k=2i$. Let  $l_{k-1}=l$ and  $\lambda=\lambda^{(k-1)}\cup\lbrace(j,\lambda_j)\rbrace$.  Then 
\begin{align*}
d_{(\lambda^{(k-1)},l)\to(\lambda,l)}^{(k)}=w_{a,i-l}e_{k-1}^{(l)}&&\text{and}&& p_{(\lambda^{(k-1)},l)\to(\lambda,l)}^{(k)}=w_{l+a,i},
\end{align*}
where $a=\sum_{r=0}^{j}\lambda_r$. Since 
\begin{align*}
c_{(\lambda,0)}e_{k-1}^{(l)}w_{a,i-l}e_{k-2}^{(l)}= e_{k-1}^{(l)}\sum_{r=0}^{\lambda_j-1}w_{a-r,a}w_{a,i-l} c_{(\lambda^{(k-1)},l)},
\end{align*}
by induction on $k$, we have 
\begin{align*}
c_{(\lambda,l)}d_\mathfrak{t}&= e_{k-1}^{(l)}\sum_{r=0}^{\lambda_j-1}w_{a-r,a}w_{a,i-l} c_{(\lambda^{(k-1)},l)} d_{\mathfrak{s}}\\
&=e_{k-1}^{(l)}\sum_{r=0}^{\lambda_j-1}w_{a-r,a}w_{a,i-l} c_{(\lambda^{(k-1)},l)} d_{\mathfrak{t}^{(\lambda^{(k-1)},l)}} p_{\mathfrak{s}}\\
&= c_{(\lambda,l)}w_{a,i-l}e_{k-2}^{(l)} e_{k-3}^{(l)}\cdots e_{2l-1}^{(l)}p_\mathfrak{s}. 
\end{align*}
It therefore suffices to show that 
\begin{align}\label{firstcase}
w_{a,i-l}e_{k-1}^{(l)} e_{k-2}^{(l)}\cdots e_{2l-1}^{(l)}=e_{k-1}^{(l)} e_{k-2}^{(l)}\cdots e_{2l-1}^{(l)}w_{a+l,i}. 
\end{align}
To this end,  note that in the diagram presentation of $A_k=A_{2i}$, 
\begin{align*}
\begin{matrix}
\begin{tikzpicture}
\foreach \x in {1,1.8,2.6}
\filldraw [black] (\x,1) circle (1.2pt);
\foreach \x in {1,1.8,2.6}
\filldraw [black] (\x,0) circle (1.2pt);
\node at (3.4,-0.05) {$\displaystyle\ldots \color{black}$};
\foreach \x in {4.2,5.0,5.8,6.6,7.4}
\filldraw [black] (\x,0) circle (1.2pt);
\foreach \x in {5.0,5.8,6.6,7.4,8.2}
\filldraw [black] (\x,1) circle (1.2pt);
\node at (9.0,0.95) {$\displaystyle\ldots \color{black}$};
\foreach \x in {9.8,10.6}
\filldraw [black] (\x,1) circle (1.2pt);
\foreach \x in {9.8,10.6}
\filldraw [black] (\x,0) circle (1.2pt);
\draw (1,1) -- (5.8,0);
\draw (1.8,1) -- (6.6,0);
\draw (2.6,1) -- (7.4,0);
\draw (5.0,1) -- (9.8,0);
\draw (5.8,1) -- (10.6,0);
\node at (6.45,0.46) {$\displaystyle\ldots \color{black}$};
\node at (1,-0.3) {$1$};
\node at (1.8,-0.3) {$2$};
\node at (5.0,-0.3) {$l$};
\node at (10.6,-0.3) {$i$};
\node at (-0.87,0.55) {$e_{k-1}^{(l)}e_{k-2}^{(l)}\cdots e_{2l-1}^{(l)}=$};
\end{tikzpicture}
\end{matrix}\ .
\end{align*}
Hence 
\begin{align}\label{abraid}
s_je_{k-1}^{(l)}e_{k-2}^{(l)}\cdots e_{2l-1}^{(l)}=e_{k-1}^{(l)}e_{k-2}^{(l)}\cdots e_{2l-1}^{(l)}s_{j+l},&&\text{for $j=1,\ldots,i-l-1$}
\end{align}
and the relation~\eqref{firstcase} follows.

{\textsc{Case~2.}} Assume that $k=2i$ and $(\lambda^{(k-1)},l_{k-1})=(\lambda,l-1)$. Then 
\begin{align*}
d_{(\lambda,l-1)\to(\lambda,l)}^{(k)}=e_{k-2}^{(l-1)} &&\text{and}&& p_{(\lambda,l-1)\to(\lambda,l)}^{(k)}=w_{l,i}. 
\end{align*}
By induction on $k$, 
\begin{align*}
c_{(\lambda,l)}d_\mathfrak{t} &=c_{(\lambda,0)}e_{k-1}^{(l)}e_{k-2}^{(l-1)} d_\mathfrak{s}=e_{k-1}^{(l)}c_{(\lambda,l-1)} d_\mathfrak{s} = e_{k-1}^{(l)}c_{(\lambda,l-1)} d_{\mathfrak{t}^{(\lambda,l-1)}} p_\mathfrak{s} =c_{(\lambda,l)}e_{k-2}^{(l-1)} d_{\mathfrak{t}^{(\lambda,l-1)}} p_\mathfrak{s}.
\end{align*}
The statement will follow from the relation
\begin{align}\label{secondcase}
e_{k-1}^{(l)}e_{k-2}^{(l-1)}e_{k-3}^{(l-1)}\cdots e_{2l-3}^{(l-1)} =e_{k-1}^{(l)}e_{k-2}^{(l)}e_{k-3}^{(l)}\cdots e_{2l-1}^{(l)}w_{l,i}
\end{align}
which can be verified in the diagram presentation for $A_k$ by simply concatenating the diagram given above for $e_{k-1}^{(l)}e_{k-2}^{(l)}e_{k-3}^{(l)}\cdots e_{2l-1}^{(l)}$ with the permutation $w_{l,i}$. 

{\textsc{Case~3.}} Assume that $k=2i+1$. Let $l_{k-1}=l-1$ and $\lambda^{(k-1)}=\lambda \cup\lbrace(j,\lambda_j^{(k-1)})\rbrace$. Then
\begin{align*}
d_{(\lambda^{(k-1)},l-1)\to(\lambda,l)}^{(k)}= w_{i-l+1,a}\sum_{r=0}^{\lambda_j}w_{a,a-r}e_{k-2}^{(l-1)}
\end{align*}
and
\begin{align*} p_{(\lambda^{(k-1)},l-1)\to(\lambda,l)}^{(k)}= w_{l,i}e_{2i}w_{i,a+l-1}\sum_{r=0}^{\lambda_j} w_{a+l-1,a+l-1-r},
\end{align*} 
where $a=1+\sum_{r=1}^j\lambda_r$. Now,
\begin{align*}
c_{(\lambda,l)}^{(k)} d_{\mathfrak{t}^{(\lambda,l)}} d_\mathfrak{t}&= c_{(\lambda,l)}^{(k)} d_{\mathfrak{t}^{(\lambda,l)}}p_{(\lambda^{(k-1)},l-1) \to(\lambda,l)}^{(k)}d_\mathfrak{s}\\ &=c_{(\lambda,l)}^{(k)} e_{k-2}^{(l)}e_{k-3}^{(l)}\cdots e_{2l-1}^{(l)}w_{l,i}e_{2i}w_{i,a+l-1} \sum_{r=0}^{\lambda_j} w_{a+l-1,a+l-1-r}d_\mathfrak{s},
\end{align*}
while induction on $k$ gives
\begin{align*}
c_{(\lambda,l)}^{(k)}d_\mathfrak{t}&=c_{(\lambda,l)}^{(k)} d_{(\lambda^{(k-1)},l-1)\to(\lambda,l)}^{(k)} d_\mathfrak{s} =e_{k-1}^{(l)}c_{(\lambda,0)} w_{i-l+1,a}\sum_{r=0}^{\lambda_j}w_{a,a-r}e_{k-2}^{(l-1)}d_\mathfrak{s}\\
&=e_{k-1}^{(l)}w_{i-l+1,a}c_{(\lambda^{(k-1)},l-1)}^{(k-1)} d_\mathfrak{s} = e_{k-1}^{(l)}w_{i-l+1,a}c_{(\lambda^{(k-1)},l-1)}^{(k-1)} d_{\mathfrak{t}^{(\lambda^{(k-1)},l-1)}}p_\mathfrak{s}\\
&= c_{(\lambda,l)}^{(k)}w_{i-l+1,a}\sum_{r=0}^{\lambda_j}w_{a,a-r}e_{k-2}^{(l-1)} e_{k-3}^{(l-1)}\cdots e_{2l-3}^{(l-1)}   p_\mathfrak{s}. 
\end{align*}
It therefore suffices to show that
\begin{equation}\label{casetwo}
\begin{split}
&e_{k-1}^{(l)}w_{i-l-1,a}\sum_{r=0}^{\lambda_j} w_{a,a-r} e_{k-2}^{(l-1)}e_{k-3}^{(l-1)}\cdots e_{2l-3}^{(l-1)}\\&=  e_{k-1}^{(l)}e_{k-2}^{(l)}\cdots e_{2l-1}^{(l)} w_{l,i}e_{2i}w_{i,a+l-1}\sum_{r=0}^{\lambda_j} w_{a+l-1,a+l-1-r}.
\end{split}
\end{equation}
To establish the above relation, we first use~\eqref{abraid} to observe that 
\begin{align*}
w_{i-l-1,a}\sum_{r=0}^{\lambda_j} w_{a,a-r} e_{k-2}^{(l-1)}e_{k-3}^{(l-1)}\cdots e_{2l-3}^{(l-1)}= e_{k-2}^{(l-1)}e_{k-3}^{(l-1)}\cdots e_{2l-3}^{(l-1)} w_{i,a+l-1}\sum_{r=0}^{\lambda_j} w_{a+l-1,a+l-1-r}.
\end{align*}
The relation~\eqref{casetwo} will therefore follow from 
\begin{align*}
e_{k-1}^{(l)}e_{k-2}^{(l-1)}e_{k-3}^{(l-1)}\cdots e_{2l-3}^{(l-1)}= e_{k-1}^{(l)}e_{k-2}^{(l)}\cdots e_{2l-1}^{(l)}w_{l,i}e_{2i}, 
\end{align*}
which can easily be established by induction or in the diagram presentation for $A_k$. 

{\textsc{Case~4.}} Assume that $(\lambda^{(k-1)},l_{k-1})=(\lambda,l)$. Then 
\begin{align*}
d^{(k)}_{(\lambda,l)\to(\lambda,l)}=e_{k-2}^{(l)}&& \text{and} && p_{(\lambda,l)\to(\lambda,l)}^{k)}=1. 
\end{align*}
By induction on $k$, 
\begin{align*}
c_{(\lambda,l)}^{(k)}d_\mathfrak{t}& =c_{(\lambda,l)}^{(k)}e_{k-2}^{(l)}d_\mathfrak{s} = e_{k-1}^{(l)}c_{(\lambda,l)}^{(k-1)}d_\mathfrak{s} = e_{k-1}^{(l)}c_{(\lambda,l)}^{(k-1)} d_{\mathfrak{t}^{(\lambda,l)}}^{(k-1)}p_\mathfrak{s} \\ 
&= c^{(k)}_{(\lambda,l)}e_{k-2}^{(l)}e_{k-3}^{(l)}\cdots e_{2l-1}^{(l)}p_\mathfrak{s} = c_{(\lambda,l)}^{(k)}d_{\mathfrak{t}^{(\lambda,l)}}^{(k)} p_{(\lambda,l)\to(\lambda,l)}^{(k)}p_\mathfrak{s} =c^{(k)}_{(\lambda,l)}d_{\mathfrak{t}^{(\lambda,l)}}^{(k)} p_\mathfrak{t},
\end{align*}
as required. 
\end{proof}
For $(\lambda,l)\in\hat{A}_i$, we define the element
\begin{align*}
a_{(\lambda,l)}=a_{(\lambda,l)}^{(i)}=d_{\mathfrak{t}^{(\lambda,l)}}^*c_{(\lambda,l)}d_{\mathfrak{t}^{(\lambda,l)}}\in A_i. 
\end{align*}
\begin{corollary}
If $i=1,2,\ldots,$ the set 
\begin{align*}
\mathscr{A}_i=
\left\lbrace 
m_\mathfrak{st}=p_\mathfrak{s}^*a_{(\lambda,l)}^{(i)}p_\mathfrak{t}\mid \mathfrak{s},\mathfrak{t}\in\hat{A}_i^{(\lambda,l)}\text{ and } (\lambda,l)\in\hat{A}_i\right\rbrace
\end{align*} 
is an $R$--basis for $\mathcal{A}_i$, and $(\mathcal{A}_i,*,\hat{A},\unrhd,\mathscr{A}_i)$ is a cell datum for $\mathcal{A}_i$. 
\end{corollary}
Recall that if $a=(i,j)$ is a node, $c(a)=j-i$ is the content of $a$.
\begin{definition}
Assume that $(\lambda,l)\in\hat{A}_k$ and $\mathfrak{t}=((\lambda^{(0)},l_0),\ldots,(\lambda^{(k)},l_k))\in\hat{A}_k^{(\lambda,l)}$. Let $i\in \mathbb{Z}$, where $0< i\le k$.  If $i$ is even, define  
\begin{align*}
c_\mathfrak{t}(i)=
\begin{cases}
z-|\lambda^{(i)}|,&\text{if $\lambda^{(i)}=\lambda^{(i-1)}$,}\\
c(a),&\text{if $\lambda^{(i)}=\lambda^{(i-1)}\cup\lbrace a\rbrace$,}
\end{cases}
\end{align*}
and, if $i$ is odd, let
\begin{align*}
c_\mathfrak{t}(i)=
\begin{cases}
|\lambda^{(i)}|, &\text{if $\lambda^{(i)}=\lambda^{(i-1)}$,}\\
z-c(a),&\text{if $\lambda^{(i)}=\lambda^{(i-1)}\setminus \lbrace a\rbrace$.}
\end{cases}
\end{align*}
\end{definition}
Halverson and Ram \cite[Theorem~3.37]{MR2143201} prove that if $\mathfrak{t} \in\hat{A}_{k}^{(\lambda,l)}$, then  $\{c_\mathfrak{t}(i)\mid i=1,\ldots,k\}$ are precisely the eigenvalues for the action of the JM subalgebra $\mathscr{L}_k$ on the basis element corresponding to the path $\mathfrak{t}$ in the irreducible representation of $A_k(n)$ indexed by $(\lambda,l)$. The next statement is proved in~\cite[Theorem~3.35]{MR2143201}. 
\begin{lemma}\label{centralaction} (1) If $(\lambda,l)\in\hat{A}_{2k}$, the central element $z_{2k}=L_1+L_2+\cdots+L_{2k}$ acts on $\Delta_{2k}^{(\lambda,l)}$ as multiplication by the scalar
\begin{align}\label{centraleven}
z_{2k}(\lambda,l)=lz+\binom{k-l}{2}+\sum_{a\in\lambda} c(a).
\end{align}
(2) If $(\lambda,l)\in\hat{A}_{2k+1}$, the central element $z_{2k+1}=L_1+L_2+\cdots+L_{2k+1}$ acts on $\Delta_{2k+1}^{(\lambda,l)}$ as multiplication by the scalar
\begin{align}\label{centralodd}
z_{2k+1}(\lambda,l)=lz+\binom{k-l+1}{2}+\sum_{a\in\lambda} c(a).
\end{align}
\end{lemma}
\begin{proof}
We prove~\eqref{centraleven}. Observing that $\sigma_{2i}\equiv 1\mod \mathcal{A}_{2k-2l}^{\rhd(\lambda,0)}$, for $1\le i<k-l$, and using~\cite[Theorem~3.32]{MR1711316}, we may write
\begin{align}\label{congruences}
c_{(\lambda,0)}L_{2i-1}\equiv (i-1)c_{(\lambda,0)} \qquad \text{and}\qquad c_{(\lambda,0)}L_{2i}\equiv c_{\mathfrak{t}^{(\lambda,0)}}(i)c_{(\lambda,0)}
\end{align}
modulo $\mathcal{A}_{2k-2l}^{\rhd (\lambda,0)}$ for $i=1,2,\ldots, k-l$. Using the relation $e_{2i-1}(L_{2i-1}+ L_{2i})=ze_{2i-1}$ gives 
\begin{align*}
c_{(\lambda,l)}z_{2k}=c_{(\lambda,0)}e_{2k-1}^{(l)}z_{2k}& =e_{2k-1}^{(l)}\sum_{i=1}^{2k-2l}c_{(\lambda,0)}  L_{i}+ lz c_{(\lambda,l)}.
\end{align*}
Applying the congruences~\eqref{congruences} to the sum on the right hand side of the last expression yields the relation~\eqref{centraleven}. The proof of~\eqref{centralodd} is similar. 
\end{proof}

By Proposition~\ref{a-c-r} and Lemma~\ref{centralaction}, $\lbrace L_i\mid i\ge 1\rbrace$ is a family of additive JM elements for the tower of algebras $(\mathcal{A}_i)_{i\ge0}$ in the sense of~\cite[Definition~3.4]{MR2774622}. Lemma~\ref{centralaction} and~\cite[Proposition~3.7]{MR2774622} together show that $\mathscr{L}_k$ acts on cell modules as follows.
\begin{proposition}\label{u-t-c}
Assume that $(\lambda,l)\in\hat{A}_k$. If $\mathfrak{t}\in\hat{A}_k$ and $1\le i\le k$, then 
\begin{align*}
m_\mathfrak{t}L_i=c_\mathfrak{t}(i)m_t +\sum_{\mathfrak{s}\succ \mathfrak{t}}r_\mathfrak{s}m_\mathfrak{s}
\end{align*}
for some $r_\mathfrak{s}\in R$, which depend on $i$ and $\mathfrak{t}$. 
\end{proposition}

\section{A Seminormal Form}\label{s-n-s}
The integral basis for the cell module $\Delta_{k}^{(\lambda,l)}$ given in Proposition~\ref{cellmodule} does not in general diagonalise the JM subalgebra $\mathscr{L}_k$. In this section, we extend scalars to the field of fractions of $R$ and, following Mathas~\cite{MR2414949}, use the triangular action of the JM elements on cell modules to define irreducible representations of the partition algebras on cell modules relative to relative to a basis of eigenvectors for $\mathscr{L}_k$ over the field of fractions of $R$.

Let $\mathbb{F}$ denote the field of fractions of $R=\mathbb{Z}[z]$ and define
\begin{align*}
A_{k}(z)=\mathcal{A}_{k}(z)\otimes_R \mathbb{F}.
\end{align*}
We will write $A_k=A_k(z)$. To simplify notation, we will also freely write $L_i,e_{i},m_\mathfrak{st},$ and so on, in place of $L_i\otimes 1_\mathbb{F}$, $e_i\otimes 1_\mathbb{F}$, $m_\mathfrak{st}\otimes 1_\mathbb{F}$, and so on.  If $(\lambda,{l})\in\hat{A}_k$, define  
\begin{align*}
A_{k}^{\rhd(\lambda,{l})}=\mathcal{A}_{k}^{\rhd(\lambda,{l})}\otimes_R \mathbb{F}
&&\text{and}&&\Delta_{k,\mathbb{F}}^{(\lambda,l)}=\Delta_{k}^{(\lambda,l)}\otimes_R \mathbb{F}.
\end{align*}
\begin{definition}\label{basis:1}
Let $(\lambda,{l})\in\hat{A}_k$ and   $\mathfrak{s,t}\in\hat{A}_{k}^{(\lambda,{l})}$. Define
\begin{align*}
F_\mathfrak{t}=
\prod_{
1\le i\le k
}
\prod_{
\substack{
\mathfrak{u}\in\hat{A}_{k}^{(\rho,r)}\\
c_\mathfrak{u}(i)\ne c_\mathfrak{t}(i)}}
\frac{L_i-c_\mathfrak{u}(i)}{c_\mathfrak{t}(i)-c_\mathfrak{u}(i)},
\end{align*}
where the product is taken over $(\rho,r)\in\hat{A}_k$. Let $f_\mathfrak{t}=m_\mathfrak{t}F_\mathfrak{t}$ and $F_{\mathfrak{s}\mathfrak{t}}=F_\mathfrak{s} m_\mathfrak{st} F_\mathfrak{t}$.
\end{definition}
A straightforward induction on $k$ shows the JM elements $L_1,\ldots,L_k$ satisfy the separation condition of Mathas~\cite[Definition~2.8]{MR2414949}. Thus, from~\cite[Sect.~3]{MR2414949}, we obtain the next statement. 
\begin{proposition}\label{s-n-d}
Let $k=1,2,\ldots,$ and $(\lambda,l)\in\hat{A}_{k}$. 
\begin{enumerate}[label=(\arabic{*}), ref=\arabic{*},leftmargin=0pt,itemindent=1.5em]
\item\label{s-n-d:1} If $\mathfrak{t}\in\hat{A}_{k}^{(\lambda,l)}$, then there exist scalars $r_\mathfrak{s}\in \mathbb{F}$, for $\mathfrak{s}\in\hat{A}_{k}^{(\lambda,l)}$, such that 
\begin{align*}
f_\mathfrak{t}=m_\mathfrak{t}
+\sum_{\substack{\mathfrak{s}\succ\mathfrak{t}}} r_\mathfrak{s}m_\mathfrak{s}.
\end{align*}
\item\label{s-n-d:2} ${\lbrace}f_\mathfrak{t}\mid \mathfrak{t}\in\hat{A}_{k}^{(\lambda,l)}{\rbrace}$ is an $\mathbb{F}$-basis for $\Delta^{(\lambda,l)}_{k,\mathbb{F}}$. 
\item\label{s-n-d:3}${\lbrace}F_\mathfrak{st}\mid \text{$\mathfrak{s},\mathfrak{t}\in\hat{A}_{k}^{(\lambda,l)}$ and $(\lambda,l)\in\hat{A}_{k}$}{\rbrace}$ is an $\mathbb{F}$-basis for $A_{k}$. 
\item\label{s-n-d:4} $f_\mathfrak{t}L_i=c_\mathfrak{t}(i)f_\mathfrak{t}$ for all $\mathfrak{t}\in\hat{A}_k^{(\lambda,l)}$ and $i=1,\ldots,k$. 
\item\label{s-n-d:5}  $F_\mathfrak{s}F_\mathfrak{t} =\delta_\mathfrak{st}F_\mathfrak{s}$ and $f_\mathfrak{s}F_{\mathfrak{t}}= \delta_{\mathfrak{st}}f_\mathfrak{s}$ for all $\mathfrak{s},\mathfrak{t}\in\hat{A}_{k}^{(\lambda,l)}$.
\item\label{s-n-d:6}  $\langle f_\mathfrak{s},f_\mathfrak{t}\rangle=\delta_\mathfrak{st} \langle f_\mathfrak{s},f_\mathfrak{s}\rangle$ for all $\mathfrak{s},\mathfrak{t}\in\hat{A}_k^{(\lambda,l)}$.
\item\label{s-n-d:7} $f_\mathfrak{s}F_\mathfrak{tu}=\langle f_\mathfrak{s},f_\mathfrak{t}\rangle f_\mathfrak{u}$ for all $\mathfrak{s,t,u}\in\hat{A}_k^{(\lambda,l)}$.
\end{enumerate}
\end{proposition}
The bases given in~\eqref{s-n-d:3},\eqref{s-n-d:4} above are the \emph{seminormal bases} for for $\Delta^{(\lambda,l)}_{k,\mathbb{F}}$ and $A_k$ respectively. The seminormal bases are for $\Delta^{(\lambda,l)}_{k,\mathbb{F}}$ and $A_k$ are unique up to scaling factors in $\mathbb{F}$.  
\begin{proposition}
Assume that $(\lambda,l)\in\hat{A}_k$ and  $(\rho,r)\in\hat{A}_{k-1}$, where $(\rho,r)\to(\lambda,l)\in\hat{A}$. Let $N^{(\rho,r)}\subseteq \Delta_{k,\mathbb{F}}^{(\lambda,l)}$ denote the subspace with basis $\lbrace f_\mathfrak{s} \in \Delta_{k,\mathbb{F}}^{(\lambda,l)} \mid \Shape(\mathfrak{s}\downarrow_{k-1})=(\rho,r)\rbrace$.
\item[(1)]  The linear map $N^{(\rho,r)}\to \Delta_{k-1,\mathbb{F}}^{(\rho,r)}$ given by
\begin{align}\label{restriction}
f_\mathfrak{s}\mapsto f_\mathfrak{t},\qquad\text{if  $\mathfrak{s}\in\hat{A}_k^{(\lambda,l)}$, $\mathfrak{t}\in\hat{A}_{k-1}^{(\rho,r)}$ and $\Shape(\mathfrak{s}\downarrow_{k-1})=(\rho,r)$,}
\end{align} 
is an isomorphism of $A_{k-1}$-modules. 
\item[(2)] The maps~\eqref{restriction} induce an isomorphism of right $A_{k-1}$--modules
\begin{align*}
\Delta^{(\lambda,l)}_{k,\mathbb{F} }\cong\bigoplus_{ \substack{(\mu,m) \to(\lambda,l)}}\Delta^{(\mu,m)}_{k-1,\mathbb{F}},
\end{align*}
where the sum is over $(\mu,m)\in\hat{A}_{k-1}$ such that $(\mu,m)\to(\lambda,l)$ in $\hat{A}$. 
\end{proposition}
\begin{proof}
Both statements follow from Proposition~\ref{filtration} and the separating property~\cite[Definition~2.8]{MR2414949} of the JM subalgebra in $\hat{A}_k$. 
\end{proof}
Let $(\lambda,{l})\in\hat{A}_{k+1}$ and  $\mathfrak{t}\in\hat{A}_{k+1}^{(\lambda,{l})}$. For $i=1,\ldots,k$, define structure constants $e_i(\mathfrak{s},\mathfrak{t}),\sigma_{i}(\mathfrak{s},\mathfrak{t})\in\mathbb{F}$ by
\begin{align*}
f_\mathfrak{t}e_{i}
=\sum_{\mathfrak{s}} e_{i}(\mathfrak{s},\mathfrak{t})f_\mathfrak{s},\qquad\text{and}\qquad f_\mathfrak{t}\sigma_{i}
&=\sum_{\mathfrak{s}} \sigma_{i}(\mathfrak{s},\mathfrak{t})f_\mathfrak{s}.
\end{align*}
In a seminormal representation of $A_{k}$, the structure constants for the Coxeter generator $s_i$  are obtained by the relation
\begin{align*}
s_i(\mathfrak{s},\mathfrak{t})= \sum_{\mathfrak{u}}
\sigma_{2i}(\mathfrak{s},\mathfrak{u})\sigma_{2i+1}(\mathfrak{u},\mathfrak{t}),\qquad\text{for $\mathfrak{s},\mathfrak{t}\in\hat{A}_{k}^{(\lambda,{l})}$ and $1\le i<\lfloor k/2\rfloor$.}
\end{align*}

The statement of the main results of this paper requires some combinatorial preparation.
\begin{definition}
Assume that $(\lambda,l)\in\hat{A}_k$ and $\mathfrak{s},\mathfrak{t}\in\hat{A}^{(\lambda,l)}$. 
\begin{enumerate}
 \item Write $\mathfrak{s}\stackrel{i}{\sim}\mathfrak{t}$ if $\mathfrak{s}^{(j)}=\mathfrak{t}^{(j)}$ whenever $j\ne i$. 
 \item Write $\mathfrak{s}\stackrel{i}{\approx}\mathfrak{t}$ if $\mathfrak{s}^{(j)}=\mathfrak{t}^{(j)}$ whenever $j\ne i-1$ and $j\ne i$. 
\end{enumerate}
\end{definition}
The next statement follows from the commutativity relations of Proposition~\ref{a-c-r}\eqref{a-c-r-10} and the fact that the seminormal basis for $\Delta_{k,\mathbb{F}}^{(\lambda,l)}$ diagonalises the subalgebra $\mathscr{L}_k$.
\begin{lemma}\label{similar}
If $(\lambda,l)\in\hat{A}$ and $\mathfrak{t}\in\hat{A}_k^{(\lambda,l)}$, then 
\begin{align*}
f_\mathfrak{t}e_{i}
=\sum_{\mathfrak{s}\stackrel{i}{\sim}\mathfrak{t}} e_{i}(\mathfrak{s},\mathfrak{t})f_\mathfrak{s},\qquad\text{and}\qquad f_\mathfrak{t}\sigma_{i}
&=\sum_{\mathfrak{s}\stackrel{i}{\approx}\mathfrak{t}} \sigma_{i}(\mathfrak{s},\mathfrak{t})f_\mathfrak{s},
\end{align*}
for $i=1,\ldots,k-1$. 
\end{lemma}
\begin{lemma}
Assume that $(\lambda,l)\in\hat{A}_{i+1}$ and $\mathfrak{t}= ((\lambda^{(0)},l_0),\ldots,(\lambda^{(i+1)},l_{i+1}))\in\hat{A}_{i+1}^{(\lambda,l)}$. Then $c_\mathfrak{t}(i)+c_\mathfrak{t}(i+1)=z$ if and only if $\lambda^{(i-1)}=\lambda^{(i+1)}$. 
\end{lemma}
\begin{proof}
Assume that $i$ is odd. If $\lambda^{(i)}=\lambda^{(i+1)}$, then $c_\mathfrak{t}(i+1)=z-|\lambda|$. Thus $c_\mathfrak{t}(i)+c_\mathfrak{t}(i+1)=z$ if and only if $\lambda^{(i-1)}=\lambda$. On the other hand, if $\lambda=\lambda^{(i)}\cup\lbrace a\rbrace$, then $c_\mathfrak{t}(i+1)=c(a)$ and $c_\mathfrak{t}(i)+c_\mathfrak{t}(i+1)=z$ if and only if $\lambda^{(i-1)}=\lambda^{(i)}\cup\lbrace a\rbrace=\lambda$.  The case where $i$ is even is similar.
\end{proof}
\begin{proposition}\label{nonzero}
Assume that $(\lambda,l)\in\hat{A}_{i+1}$ and $\mathfrak{t}= ((\lambda^{(0)},l_0),\ldots,(\lambda^{(i+1)},l_k))\in\hat{A}_{i+1}^{(\lambda,l)}$. Then the following statements are equivalent:
\begin{enumerate}
\item\label{firstz} $\lambda^{(i-1)}\ne \lambda^{(i+1)}$.
\item\label{secondz} $f_\mathfrak{t}e_i=0$.
\item\label{thirdz} $e_i(\mathfrak{t},\mathfrak{t})=0$. 
\end{enumerate}
\end{proposition}
\begin{proof} 
\eqref{firstz}$\Rightarrow$\eqref{secondz} If  $\lambda^{(i-1)}\ne \lambda^{(i+1)}$, then  $zf_\mathfrak{t}e_i=f_\mathfrak{t}(L_i+L_{i+1})e_i =(c_\mathfrak{t}(i)+c_\mathfrak{t}(i+1)) f_\mathfrak{t}e_i$. Hence $c_\mathfrak{t}(i)+c_\mathfrak{t}(i+1)\ne z$ forces that $f_\mathfrak{t}e_i=0$. 

\eqref{thirdz}$\Rightarrow$\eqref{firstz} Let $n\in\mathbb{Z}$, where $n\ge 2i+2$. Define $A_{i+1}(n)=A_{i+1}\otimes_\mathbb{F}\mathbb{Q}$ and $\Delta_{i+1,\mathbb{Q}}^{(\lambda,l)}=\Delta_{i+1,\mathbb{F}}^{(\lambda,l)}\otimes_\mathbb{F} \mathbb{Q}$, where $z\in \mathbb{F}$ acts on $\mathbb{Q}$ via $z\cdot 1_\mathbb{Q}=n$. Then $A_{i+1}(n)$ is semisimple~\cite[Theorem~3.27]{MR2143201} and $\lbrace f_\mathfrak{s}\otimes 1_\mathbb{Q} \mid\mathfrak{s}\in \hat{A}_{i+1}^{(\lambda,l)}\rbrace$ is a seminormal basis over $\mathbb{Q}$ for the $A_k(n)$-module $\Delta_{i+1,\mathbb{Q}}^{(\lambda,l)}$. We have two cases to consider. 

{\textsc{Case 1.}} If $i$ is even, let $\rho=(n-|\lambda|-1,\lambda_1,\lambda_2,\ldots)$ and 
\begin{align*}
\mu=\begin{cases}
(n-|\lambda|,\lambda^{(i)}_1,\lambda^{(i)}_2,\ldots), &\text{if $\lambda^{(i)}=\lambda^{(i+1)}$,}\\
(n-|\lambda|-1,\lambda^{(i)}_1,\lambda^{(i)}_2,\ldots), &\text{if $\lambda^{(i+1)}\subsetneq\lambda^{(i)}$.}\\
\end{cases}
\end{align*} 
If $\lambda^{(i-1)}=\lambda^{(i+1)}$,  then Schur--Weyl duality~\cite[Theorem~3.32]{MR2143201} and the argument of~\cite[Lemma~3.5]{MR1398116} show that 
\begin{align*}
e_i(\mathfrak{t,t})\cdot 1_\mathbb{Q}=\frac{\dim(S^\mu)}{n \dim(S^\rho)},
\end{align*}
where $S^\mu$ and $S^\rho$ are irreducible (Specht) modules for the symmetric groups $\mathfrak{S}_{n+1}$ and $\mathfrak{S}_n$ indexed by $\mu$ and $\rho$ respectively. 

{\textsc{Case 2.}} If $i$ is odd, let $\rho=(n-|\lambda|,\lambda_1,\lambda_2, \ldots)$ and 
\begin{align*}
\mu=
\begin{cases}
(n-|\lambda|-1,\lambda^{(i)}_1,\lambda^{(i)}_2,\ldots),&\text{if $\lambda^{(i)}=\lambda^{(i+1)}$,}\\
(n-|\lambda|,\lambda^{(i)}_1,\lambda^{(i)}_2,\ldots), &\text{if $\lambda^{(i)}\subsetneq\lambda^{(i+1)}$.}
\end{cases}
\end{align*} 
If $\lambda^{(i-1)}=\lambda^{(i+1)}$, then
\begin{align*}
e_i(\mathfrak{t,t})\cdot1_\mathbb{Q}=\frac{n\dim(S^\mu)}{\dim(S^\rho)},
\end{align*}
where $S^\mu$ and $S^\rho$ are Specht modules for $\mathfrak{S}_{n-1}$ and $\mathfrak{S}_n$ respectively.

In both cases above, we conclude that $e_i(\mathfrak{t,t})\ne0$ whenever $\lambda^{(i-1)}=\lambda^{(i+1)}$. This gives the required implication. 
\end{proof}
If $\lambda,\mu$ are partitions, let $\lambda\ominus\mu=\lambda\setminus \mu\cup\mu\setminus\lambda$.
\begin{lemma}\label{involutions}
Let $(\lambda,l)\in\hat{A}_{i+1}$ and $\mathfrak{t}=((\lambda^{(0)},l_0),\ldots,(\lambda^{(i+1)},l_{i+1}))\in\hat{A}_{i+1}^{(\lambda,l)}$. Assume that $\lambda^{(i+1)}\ominus\lambda^{(i-2)}=\lbrace \alpha,\beta\rbrace$.  
\begin{enumerate}[label=(\arabic{*}), ref=\arabic{*},leftmargin=0pt,itemindent=1.5em]
\item\label{involutions:1} If $\alpha$ and $\beta$ are neither in the same row nor the same column, then there exists $\mathfrak{s}= \mathfrak{t}\sigma_{i}\in\hat{A}_{i+1}^{(\lambda,l)}$ such that $\lbrace \mathfrak{u}\in\hat{A}_{i+1}^{(\lambda,l)}\mid \mathfrak{u}\stackrel{i}{\approx}\mathfrak{t}\rbrace =\lbrace \mathfrak{s},\mathfrak{t}\}$. Moreover,  
\begin{align*}
c_\mathfrak{s}(i+1)=c_\mathfrak{t}(i-1),\qquad c_\mathfrak{s}(i-1)=c_\mathfrak{t}(i+1),\qquad\text{and}\qquad c_\mathfrak{s}(i)=c_\mathfrak{t}(i).
\end{align*}
\item\label{involutions:2} If $\alpha$ and $\beta$ are in the same row or the same column, then $\lbrace \mathfrak{u}\in\hat{A}_{i+1}^{(\lambda,l)}\mid \mathfrak{u}\stackrel{i}{\approx}\mathfrak{t}\rbrace =\lbrace \mathfrak{t}\}$. 
\end{enumerate}
\end{lemma}
\begin{proof}\eqref{involutions:1}
Assume that $i$ is even. If  $\lambda^{(i+1)}=\lambda^{(i)}$, then $\lambda^{(i-1)}\subsetneq\lambda^{(i)}$ and $\lambda^{(i-1)}\subsetneq\lambda^{(i-2)}$. If $\lambda^{(i-2)}=\lambda^{(i-1)}\cup\lbrace \alpha\rbrace$ and $\lambda^{(i)}=\lambda^{(i-1)}\cup\lbrace \beta\rbrace$, then assumption that $\alpha$ and $\beta$ are neither in the same column nor the same row implies that  $\rho=\lambda^{(i-2)}\cup\lbrace \beta\rbrace$ is a partition and $(\lambda^{(i-2)},l-1)\to (\lambda^{(i-2)},l-1)\to (\rho,l-1)\to (\lambda^{(i+1)},l)$ is a path from level $i-2$ to level $i+1$ of $\hat{A}$. Define $\mathfrak{s}\in \hat{A}_{i+1}^{(\lambda,l)}$ by $\mathfrak{s}\stackrel{i}{\approx}\mathfrak{t}$ and  $\mathfrak{s}^{(i-1)}=(\lambda^{(i-2)},l-1)$ and $\mathfrak{s}^{(i)}=(\rho,l-1)$. 

If $i$ is even and $\lambda^{(i-1)}=\lambda^{(i)}$, then $\lambda^{(i-2)}\supsetneq\lambda^{(i-1)}$ and $\lambda^{(i)}\supsetneq\lambda^{(i+1)}$. Let $\lambda^{(i-2)}=\lambda^{(i-1)}\cup\lbrace \alpha\rbrace$ and $\lambda^{(i)}=\lambda^{(i+1)}\cup \lbrace \beta\rbrace$. Then $\rho=\lambda^{(i+1)}\cup\lbrace \alpha\rbrace$ is a partition and $(\lambda^{(l-2)},l-3)\to (\rho,l-2)\to (\rho,l-1)\to (\lambda,l)$ is a path from level $i-2$ to level $i+1$ of $\hat{A}$. Define $\mathfrak{s}\in \hat{A}_{i+1}^{(\lambda,l)}$ by $\mathfrak{s}\stackrel{i}{\approx}\mathfrak{t}$ and  $\mathfrak{s}^{(i-1)}=(\rho,l-1)$ and $\mathfrak{s}^{(i)}=(\rho,l-1)$.

Assume that $i$ is odd.  If $\lambda^{(i+1)}=\lambda^{(i)}$, then $\lambda^{(i)}\subsetneq\lambda^{(i-1)}$ and $\lambda^{(i-2)}\subsetneq\lambda^{(i-1)}$. If $\lambda^{(i-1)}=\lambda^{(i-2)}\cup\lbrace \alpha\rbrace$ and $\lambda^{(i)}=\lambda^{(i-1)}\cup\lbrace \beta\rbrace$, then $\rho=\lambda^{(i-2)}\setminus \lbrace\beta \rbrace $ is a partition and $(\lambda^{(i-2)},l-1)\to (\lambda^{(i-2)},l-1)\to (\rho,l)\to(\lambda^{(i+1)},l)$ is a path from level $i-2$ to level $i+1$ of $\hat{A}$. Define $\mathfrak{s}\in\hat{A}_{i+1}^{(\lambda,l)}$ by $\mathfrak{s}\stackrel{i}{\approx}\mathfrak{t}$ and $\mathfrak{s}^{(i-1)}=(\lambda^{(i-2)},l-1)$ and $\mathfrak{s}^{(i)}=(\rho,l)$.

If $i$ is odd and $\lambda^{(i-1)}=\lambda^{(i)}$, then $\lambda^{(i-1)}\subsetneq\lambda^{(i-2)}$ and $\lambda^{(i)}\subsetneq\lambda^{(i+1)}$. Let $\lambda^{(i-1)}=\lambda^{(i-2)}\cup\lbrace \alpha\rbrace $ and $\lambda^{(i+1)}=\lambda^{(i)}\cup\lbrace \beta\rbrace$. Then $\rho=\lambda^{(i-2)}\cup\lbrace \beta\rbrace$ is a partition and $(\lambda^{(i-2)},l)\to (\rho,l)\to (\rho,l)\to (\lambda^{(i+1)},l)$ is a path from level $i-2$ to level $i+1$ of $\hat{A}$. Define $\mathfrak{s}\in\hat{A}_{i+1}^{(\lambda,l)}$ by $\mathfrak{s}\stackrel{i}{\approx}\mathfrak{t}$ and $\mathfrak{s}^{(i-1)}=(\rho,l)$ and $\mathfrak{s}^{(i)}=(\rho,l)$. 

In each case above, $\mathfrak{s}$ is the unique element in $\hat{A}_{i+1}^{(\lambda,l)}$ satisfying the conditions $\mathfrak{s}\stackrel{i}{\approx}\mathfrak{t}$ and $\mathfrak{s}\ne \mathfrak{t}$. We therefore denote $\mathfrak{s}=\mathfrak{t}\sigma_i$ and $\mathfrak{t}=\mathfrak{s}\sigma_i$ in each case. 

\eqref{involutions:2} Assume that $i$ is odd and that $\alpha$ and $\beta$ are in the same row. Then $\lambda^{(i-2)}\subsetneq\lambda^{(i-1)}=\lambda^{(i)}\subsetneq \lambda^{(i+1)}$. Without loss of generality, we may write $\alpha=(j,\lambda_j-1)$ and $\beta=(j,\lambda_j)$, thereby making it clear that $(\lambda^{(i-2)},l)\to(\lambda^{(i-1)},l)\to(\lambda^{(i)},l)\to(\lambda^{(i+1)},l)$ is the unique path from the vertex $(\lambda^{(i-2)},l)$ in level  $i-2$ to the vertex $(\lambda,l)$ in level $i+1$ of $\hat{A}$. The remaining cases are similar. 
\end{proof}
\begin{lemma}\label{onebox}
Let $(\lambda,l)\in\hat{A}_{i+1}$ and $\mathfrak{t}=((\lambda^{(0)},l_0),\ldots,(\lambda^{(i+1)},l_{i+1}))\in\hat{A}_{i+1}^{(\lambda,l)}$, where $\lambda^{(i+1)}\ominus\lambda^{(i-2)}$ consists of a single node. If $\lambda^{(i+1)}\ne\lambda^{(i-1)}$ and $\lambda^{(i)}\ne\lambda^{(i-2)}$, then $\lbrace \mathfrak{s}\in\hat{A}_{i+1}^{(\lambda,l)}\mid \mathfrak{s}\stackrel{i}{\approx}\mathfrak{t}\rbrace=\lbrace \mathfrak{t} \rbrace$. 
\end{lemma}
\begin{proof}
If $i$ is even, then $\lambda^{(i-2)}=\lambda^{(i-1)}\subsetneq\lambda^{(i)}=\lambda^{(i+1)}$ and $(\lambda^{(i-2)},l)\to(\lambda^{(i-1)},l)\to(\lambda^{(i)},l)\to(\lambda^{(i+1)},l)$ is the unique path from the vertex $(\lambda^{(i-2)},l)$ in level  $i-2$ to the vertex $(\lambda,l)$ in level $i+1$ of $\hat{A}$.  If $i$ is odd,  then $\lambda^{(i-2)}=\lambda^{(i-1)}\supsetneq\lambda^{(i)}=\lambda^{(i+1)}$ and $(\lambda^{(i-2)},l-3)\to(\lambda^{(i-1)},l-2)\to(\lambda^{(i)},l-1)\to(\lambda^{(i+1)},l)$ is the unique path from the vertex $(\lambda^{(i-2)},l-3)$ in level $i-2$ to the vertex $(\lambda,l)$ in level $i+1$ of $\hat{A}$. 
\end{proof}
\begin{definition}
Let $\lambda$ be a partition.
\begin{enumerate}[label=(\arabic{*}), ref=\arabic{*},leftmargin=0pt,itemindent=1.5em]
\item If $a=(i,\lambda_i)\in R(\lambda)$, let $R(\lambda)^{<a}=\lbrace (i,\lambda_i)\in R(\lambda)\mid i>j\rbrace$.
\item If $a=(i,\lambda_i+1)\in A(\lambda)$, let $A(\lambda)^{<a}=\lbrace (i,\lambda_i+1)\in A(\lambda)\mid i>j\rbrace$.
\end{enumerate}
\end{definition}
The next definition gives a set branching factors for the inner product~\eqref{f-d-1} on cell modules. 
\begin{definition}
Let $\lambda,\mu$ be partitions.
\begin{enumerate}[label=(\arabic{*}), ref=\arabic{*},leftmargin=0pt,itemindent=1.5em]
\item If $\lambda\in\widehat{H}_{2i}$, $\mu\in\widehat{H}_{2i+1}$, where $\lambda\to\mu$ in $\widehat{H}$, then $\lambda=\mu$ and, in this case, define $\gamma_{\lambda\to\mu}=1$. 
\item If $\lambda\in\widehat{H}_{2i+1}$, $\mu\in\widehat{H}_{2i+2}$, where $\lambda\to\mu$ in $\widehat{H}$, then $\lambda\subsetneq\mu$ and, in this case, define 
\begin{align*}
\gamma_{\lambda\to\mu}= 
{\displaystyle\frac{\prod_{\beta\in A(\lambda)^{<\alpha}}(c(\alpha)-c(\beta))}{\prod_{\beta\in R(\lambda)^{<\alpha}}(c(\alpha)-c(\beta))}}, \qquad\text{where $\mu=\lambda\cup \lbrace \alpha\rbrace$.}
\end{align*}
\item If $(\lambda,l)\in\hat{A}_i$, $(\mu,m)\in\hat{A}_{i+1}$, where $(\lambda,l)\to(\mu,m)$ in $\hat{A}$, let 

\begin{align}\label{gammadef}
\gamma^{(i+1)}_{(\lambda,l)\to(\mu,m)}=
\begin{cases}
\gamma_{\lambda\to\mu}, &\text{if $l=m$,}\vspace{0.5em}\\
e_{i}(\mathfrak{t,t})\gamma_{\mu\to\lambda}, &\text{if $l=m-1$,}
\end{cases}
\end{align} 
where, in the second case in~\eqref{gammadef}, the path $\mathfrak{t}\in\hat{A}_{i+1}^{(\mu,m)}$ satisfies $\mathfrak{t}^{(i-1)}=(\mu,l)$ and $\mathfrak{t}^{(i)}=(\lambda,l)$. 
\end{enumerate}
\end{definition}
Using Proposition~\ref{nonzero} and Lemma~\ref{r-e-s-i}, we note that the structure constant $e_{i}(\mathfrak{t},\mathfrak{t})$  in~\eqref{gammadef} is non-zero and completely determined by $\lambda$ and $\mu$. We defer the proof of the following statement on the branching factors for the inner product on cell modules to Sect.~\ref{det-proof}.
\begin{proposition}\label{branching}
Let $(\lambda,l)\in\hat{A}_{k}$ and $(\mu,m)\in\hat{A}_{k+1}$, where $(\lambda,l)\to(\mu,m)\in\hat{A}$. If  $\mathfrak{s}\in\hat{A}_{k+1}^{(\mu,m)}$ and  $\mathfrak{s}'\in\hat{A}_{k}^{(\lambda,l)}$, where $\mathfrak{s}'=\mathfrak{s}\downarrow_{k}$, then  
\begin{align}\label{quotient}
\frac{\langle f_\mathfrak{s},f_\mathfrak{s}\rangle}{ \langle f_\mathfrak{s'},f_\mathfrak{s'}\rangle} =\gamma_{(\lambda,l)\to(\mu,m)}^{(k+1)}.
\end{align}
\end{proposition}
\section{Main Results}\label{m-r}
Theorems~\ref{contractioneven} to \ref{sigmaodd} below give explicit combinatorial formulae for representations of generators of the partition algebras relative to the seminormal basis for $\Delta_{k+1,\mathbb{F}}^{(\lambda,l)}$. The statement of results is split into four only to simplify the presentation and avoid clashes between indices. 
\begin{theorem}\label{contractioneven}
Let $k$ be even and $(\lambda,l)\in\hat{A}_{k+1}$. Write $\mathfrak{t}=((\lambda^{(0)},l_0),\ldots,(\lambda^{(k+1)},l_{k+1}))\in\hat{A}_{k+1}^{(\lambda,l)}$ and let $\mathfrak{s}\in\hat{A}_{k+1}^{(\lambda,l)}$, where $\mathfrak{s}\stackrel{k}{\sim}\mathfrak{t}$. 
~ Then the following statements hold: 
\begin{enumerate}[label=(\arabic{*}), ref=\arabic{*},leftmargin=0pt,itemindent=1.5em]
\item\label{contractioneven.1} If $\mathfrak{s}^{(k-1)}\ne(\lambda,l-1)$, then $e_k(\mathfrak{s,t}) =0$. 
\item\label{contractioneven.2} If $\mathfrak{s}^{(k-1)}=(\lambda,l-1)$, then
\begin{align}
e_{k}(\mathfrak{s,t})&=
\begin{cases}
\gamma_{\lambda\to\mu}^{-1}, &\text{if $\mathfrak{s}^{(k)}=(\mu,l-1)$ and $\mathfrak{t}^{(k)}=(\lambda,l)$,}\\
e_k(\mathfrak{s,s})e_k(\mathfrak{t,t})\gamma_{\lambda\to\mu}, &\text{if $\mathfrak{s}^{(k)}=(\lambda,l)$ and $\mathfrak{t}^{(k)}=(\mu,l-1)$,}\\ 
{\displaystyle\frac{\gamma_{\lambda\to\mu}}{\gamma_{\lambda\to\rho}}}e_k(\mathfrak{t,t}),&\text{if $\mathfrak{s}^{(k)}=(\rho,l-1)$ and $\mathfrak{t}^{(k)}=(\mu,l-1)$,}
\end{cases}\label{contractioneven.2.1}
\intertext{where} 
e_{k}(\mathfrak{t},\mathfrak{t})&=
\begin{cases}
\displaystyle{\frac{\prod_{\beta\in R(\lambda)}(z-c(\beta)-|\lambda|)}
{\prod_{\beta\in A(\lambda)}(z-c(\beta)-|\lambda|)}},
&\text{if $\lambda^{(k-1)}=\lambda^{(k)} =\lambda^{(k+1)}=\lambda$,}\medskip \\
\displaystyle{
\frac{(z-c(\alpha)-|\lambda|-1)}{(z-c(\alpha)-|\lambda|)}}
\frac{\prod_{\beta\in R(\lambda)}(c(\alpha)-c(\beta))}
{\prod_{\substack{\beta\in A(\lambda),\\ \beta\ne\alpha}}(c(\alpha)-c(\beta))},
&
\begin{matrix}
\text{if $\lambda^{(k-1)} =\lambda^{(k+1)}=\lambda$ and} \\
\text{$\lambda^{(k)} =\lambda\cup{\lbrace}\alpha{\rbrace}$.}
\end{matrix}
\end{cases}\notag
\end{align}
\end{enumerate}
\end{theorem}
\begin{theorem}\label{contractionodd}
Let $k$ be odd and $(\lambda,l)\in\hat{A}_{k+1}$. Write $\mathfrak{t}=((\lambda^{(0)},l_0),\ldots,(\lambda^{(k+1)},l_{k+1}))\in\hat{A}_{k+1}^{(\lambda,l)}$ and let $\mathfrak{s}\in\hat{A}_{k+1}^{(\lambda,l)}$, where $\mathfrak{s}\stackrel{k}{\sim}\mathfrak{t}$. 
~ Then the following statements hold: 
\begin{enumerate}[label=(\arabic{*}), ref=\arabic{*},leftmargin=0pt,itemindent=1.5em]
\item\label{contractionodd.1} If $\mathfrak{s}^{(k-1)}\ne(\lambda,l-1)$, then $e_k(\mathfrak{s,t}) =0$. 
\item\label{contractionodd.2}  If $\mathfrak{s}^{(k-1)}=(\lambda,l-1)$, then
\begin{align}
e_{k}(\mathfrak{s,t})&=
\begin{cases}
\gamma_{\mu\to\lambda} , &\text{if  $\mathfrak{s}^{(k)}=(\lambda,l-1)$ and  $\mathfrak{t}^{(k)}=(\mu,l),$}\\
{\displaystyle\frac{e_k(\mathfrak{s,s}) e_k(\mathfrak{t,t})}{\gamma_{\mu\to\lambda}}}, &\text{if $\mathfrak{s}^{(k)}=(\mu,l)$ and $\mathfrak{t}^{(k)}=(\lambda,l-1)$,}\\
{\displaystyle{\frac{\gamma_{\mu\to\lambda}}{\gamma_{\rho\to\lambda}}}}e_k(\mathfrak{s,s}), &\text{if $\mathfrak{s}^{(k)}=(\rho,l)$ and $\mathfrak{t}^{(k)}=(\mu,l)$,}
\end{cases}\label{contractionodd.2.1} 
\intertext{where}
e_{k}(\mathfrak{t},\mathfrak{t})&=
\begin{cases}
\displaystyle{
\frac{\prod_{\beta\in A(\lambda)}\big(z-c(\beta)-|\lambda|\big)}
{\prod_{\beta\in R(\lambda)}\big(z-c(\beta)-|\lambda|\big)}
},
&\text{if $\lambda^{(k-1)} =\lambda^{(k)} =\lambda^{(k+1)}=\lambda$,} \medskip\\
-\displaystyle{
\frac{(z-c(\alpha)-|\lambda|+1)}{(z-c(\alpha)-|\lambda|)}
\frac{\prod_{\beta\in A(\lambda)}(c(\beta)-c(\alpha))}{\prod_{\substack{\beta\in R(\lambda),\\ \beta\ne\alpha}}(c(\beta)-c(\alpha))}
},
&
\begin{matrix}
\text{if $\lambda^{(k-1)}= \lambda^{(k+1)}=\lambda$ and}\\
\text{$\lambda^{(k)}=\lambda\setminus{\lbrace} \alpha{\rbrace}.$}
\end{matrix}
\end{cases}\notag
\end{align}
\end{enumerate}
\end{theorem}
\begin{theorem}\label{sigmaeven} Let $k$ be even and $(\lambda,{l})\in\hat{A}_{k+1}$. If  $\mathfrak{t}=((\lambda^{(0)},l_0),\ldots,(\lambda^{(k+1)},l_{k+1}))\in\hat{A}_{k+1}^{(\lambda,{l})}$ and  $\mathfrak{s}=((\rho^{(0)},r_0),\ldots,(\rho^{(k+1)},r_{k+1}))\in\hat{A}_{k+1}^{(\lambda,l)}$, where $\mathfrak{s}\stackrel{k}{\approx}\mathfrak{t}$, then the following statements hold:
\begin{enumerate}[label=(\arabic{*}), ref=\arabic{*},leftmargin=0pt,itemindent=1.5em]
\item\label{sigmaeven.1} If $\lambda^{(k-1)}=\lambda^{(k+1)}$  and  $\lambda^{(k-2)}=\lambda^{(k)}$, then   
\begin{align*}
\sigma_{k}(\mathfrak{t,t})=
\frac{c_\mathfrak{t}(k)}{e_{k-1}(\mathfrak{t,t})}
\end{align*}
and, when $\mathfrak{s}\ne\mathfrak{t}$,
\begin{align}\label{sigmaeven.1.1}
\sigma_{k}(\mathfrak{s},\mathfrak{t})=
\frac{z-c_\mathfrak{s}(k)-c_\mathfrak{t}(k-1)-e_{k-1}(\mathfrak{t,t})}{c_\mathfrak{s}(k+1)-c_\mathfrak{t}(k-1)}e_k(\mathfrak{s,t})
\end{align}
\item\label{sigmaeven.2} If $\lambda^{(k-1)}\ne \lambda^{(k+1)}$ and $\lambda^{(k)}=\lambda^{(k-2)}$, let $\mathfrak{v}\in\hat{A}_{k+1}^{(\lambda,l)}$, where $\mathfrak{v}\stackrel{k-1}{\sim}\mathfrak{t}$ and $\mathfrak{v}^{(k-1)}=(\lambda,l-1)$. Then, 
\begin{align}\label{sigmaeven.2.2}
\sigma_k(\mathfrak{v,t})=c_\mathfrak{t}(k)e_k(\mathfrak{v,t})
\end{align}
and, 
\begin{align}\label{sigmaeven.2.1}
\sigma_{k}(\mathfrak{s,t})=\frac{\delta_\mathfrak{st}-e_{k-1}(\mathfrak{v,t}) e_k(\mathfrak{s,v})}{c_\mathfrak{s}(k+1)-c_\mathfrak{t}(k-1)}.
\end{align}
\item\label{sigmaeven.3} If $\lambda^{(k+1)}=\lambda^{(k-1)}$ and $\lambda^{(k)}\ne\lambda^{(k-2)}$, then   
\begin{align}\label{sigmaeven.3.1}
\sigma_{k}(\mathfrak{s},\mathfrak{t})&=
\frac{\delta_\mathfrak{st}+(z-c_\mathfrak{t}(k-1)-c_\mathfrak{s}(k))e_k(\mathfrak{s,t})}{c_\mathfrak{s}(k+1)-c_\mathfrak{t}(k-1)}, \quad\text{if $\rho^{(k)}\ne \lambda^{(k-2)}$,}
\end{align}
and 
\begin{align}\label{sigmaeven.3.2}
\sigma_k(\mathfrak{s,t})=\frac{\sigma_k(\mathfrak{t,s})\langle f_\mathfrak{t},f_\mathfrak{t}\rangle}{\langle f_\mathfrak{s},f_\mathfrak{s}\rangle},\quad\text{if $\rho^{(k)}=\lambda^{(k-2)}$.}
\end{align}
\item\label{sigmaeven.4} If If $\lambda^{(k+1)}\ne\lambda^{(k-1)}$ and $\lambda^{(k)}\ne\lambda^{(k-2)}$, and $\mathfrak{t}\sigma_{k}$ does not exist, then 
\begin{align}\label{sigmaeven.4.1}
\sigma_{k}(\mathfrak{s},\mathfrak{t})=
\frac{\delta_\mathfrak{st}}{c_\mathfrak{t}(k+1)-c_\mathfrak{t}(k-1)}.
\end{align}
\item\label{sigmaeven.5} If $\lambda^{(k+1)}\ne\lambda^{(k-1)}$ and $\lambda^{(k)}\ne\lambda^{(k)}$, and $\mathfrak{t}\sigma_{i}$ exists, then 
\begin{align}\label{sigmaeven.5.1}
\sigma_{k}(\mathfrak{s},\mathfrak{t})=
\begin{cases}
{\displaystyle\frac{1}{c_\mathfrak{t}(k+1)-c_\mathfrak{t}(k-1)},}&
\text{if $\mathfrak{s}=\mathfrak{t}$,}\medskip\\1-{\displaystyle\frac{1}{(c_\mathfrak{t}(k+1)-c_\mathfrak{t}(k-1))^2}},
&\text{if $\mathfrak{s}=\mathfrak{t}\sigma_{k}$ and $\mathfrak{s}\succ\mathfrak{t}$,}\medskip\\
1, &\text{if $\mathfrak{s}=\mathfrak{t}\sigma_{k}$ and $\mathfrak{t}\succ\mathfrak{s}$.}
\end{cases}
\end{align}
\end{enumerate}
\end{theorem}
\begin{theorem}\label{sigmaodd}
Let $k$ be odd and $(\lambda,{l})\in\hat{A}_{k+1}$. If  $\mathfrak{t}=((\lambda^{(0)},l_0),\ldots,(\lambda^{(k+1)},l_{k+1}))\in\hat{A}_{k+1}^{(\lambda,{l})}$ and  $\mathfrak{s}=((\rho^{(0)},r_0),\ldots,(\rho^{(k+1)},r_{k+1}))\in\hat{A}_{k+1}^{(\lambda,l)}$, where $\mathfrak{s}\stackrel{k}{\approx}\mathfrak{t}$, then the following statements hold:
\begin{enumerate}[label=(\arabic{*}), ref=\arabic{*},leftmargin=0pt,itemindent=1.5em]
\item\label{sigmaodd.1} If $\lambda^{(k+1)}=\lambda^{(k-1)}$ and $\lambda^{(k)}=\lambda^{(k-2)}$, then 
\begin{align}\label{sigmaodd.1.2}
\sigma_{k}(\mathfrak{t,t})=\frac{c_\mathfrak{t}(k-1)}{e_{k-1}(\mathfrak{t,t})}
\end{align}
and, when $\mathfrak{s}\ne\mathfrak{t}$,
\begin{align}\label{sigmaodd.1.1}
\sigma_{k}(\mathfrak{s},\mathfrak{t})= -\frac{e_{k-1}(\mathfrak{t,t})e_{k-1}(\mathfrak{s,t})}{c_\mathfrak{s}(k+1)-c_\mathfrak{t}(k-1)}.
\end{align}
\item\label{sigmaodd.2} If $\lambda^{(k+1)}\ne\lambda^{(k-1)}$ and $\lambda^{(k)}=\lambda^{(k-2)}$,  let $\mathfrak{v}\in\hat{A}_{k+1}^{(\lambda,l)}$, where $\mathfrak{v}\stackrel{k}{\sim}\mathfrak{t}$ and $\mathfrak{v}^{(k-1)}=(\lambda,l-1)$. Then, 
\begin{align}\label{sigmaodd.2.1} 
\sigma_k(\mathfrak{v,t})=c_\mathfrak{t}(k-1)e_k(\mathfrak{v,t})
\end{align}
and
\begin{align}\label{sigmaodd.2.2} 
\sigma_k(\mathfrak{s,t}) =\frac{\delta_\mathfrak{st} -e_{k-1}(\mathfrak{v,t})e_k(\mathfrak{s,v}) +(c_\mathfrak{v}(k-1) -c_\mathfrak{t}(k-1))e_{k-1}(\mathfrak{s,t})}{c_\mathfrak{s}(k+1)-c_\mathfrak{t}(k-1)}.
\end{align}
\item\label{sigmaodd.3} If $\lambda^{(k+1)}=\lambda^{(k-1)}$ and $\lambda^{(k)}\ne\lambda^{(k-2)}$, then 
\begin{align}\label{sigmaodd.3.1} 
\sigma_{k}(\mathfrak{s},\mathfrak{t})=
\frac{\delta_\mathfrak{st}}{c_\mathfrak{t}(k+1)-c_\mathfrak{t}(k-1)}\quad\text{if $\rho^{(k)}\ne\lambda^{(k-2)}$,}
\end{align}
and 
\begin{align}\label{sigmaodd.3.2}
\sigma_k(\mathfrak{s,t})=\frac{\sigma_k(\mathfrak{t,s})\langle f_\mathfrak{t},f_\mathfrak{t}\rangle}{\langle f_\mathfrak{s},f_\mathfrak{s}\rangle}\quad\text{if $\rho^{(k)}=\lambda^{(k-2)}$.}
\end{align}
\item\label{sigmaodd.4}  If $\lambda^{(k+1)}\neq\lambda^{(k-1)}$ and $\lambda^{(k)}\neq\lambda^{(k-2)}$, and $\mathfrak{t}\sigma_{k}$ does not exist, then 
\begin{align}\label{sigmaodd.4.1} 
\sigma_{k}(\mathfrak{s},\mathfrak{t})=
\frac{\delta_\mathfrak{st}}{c_\mathfrak{t}(k+1)-c_\mathfrak{t}(k-1)}.
\end{align}
\item\label{sigmaodd.5} If $\lambda^{(k+1)}\neq\lambda^{(k-1)}$ and $\lambda^{(k)}\neq\lambda^{(k-2)}$, and $\mathfrak{t}\sigma_{k}$ exists, then 
\begin{align}\label{sigmaodd.5.1}
\sigma_{k}(\mathfrak{s},\mathfrak{t})=
\begin{cases}
{\displaystyle\frac{1}{c_\mathfrak{t}(k+1)-c_\mathfrak{t}(k-1)},}&\text{if $\mathfrak{s}=\mathfrak{t}$,}\medskip\\
1-{\displaystyle\frac{1}{(c_\mathfrak{t}(k+1)-c_\mathfrak{t}(k-1))^2}},
&\text{if $\mathfrak{s}=\mathfrak{t}\sigma_{k}$ and $\mathfrak{s}\succ\mathfrak{t}$,}\medskip\\
1, &\text{if $\mathfrak{s}=\mathfrak{t}\sigma_{k}$ and $\mathfrak{t}\succ\mathfrak{s}$.}
\end{cases}
\end{align}
\end{enumerate}
\end{theorem}
\subsection{Proof of Main Results}
Before embarking on the proof of the statements above, we establish some basic properties of the contraction element $e_i\in A_{i+1}$.
\begin{lemma}\label{contract}
Assume that $(\lambda,l)\in\hat{A}_{i+1}$ and $\mathfrak{s}=((\lambda^{(0)},l_0),\ldots,(\lambda^{(i+1)},l_{i+1}))\in\hat{A}_{i+1}^{(\lambda,l)}$, where $\lambda^{(i-1)}=\lambda^{(i+1)}$. If $\mathfrak{u}=\mathfrak{s}\downarrow_{i-1}$, then $e_iF_\mathfrak{s}e_i=e_i(\mathfrak{s,s})e_i F_\mathfrak{u}$. 
\end{lemma}
\begin{proof}
Assuming that $\lambda^{(i-1)}=\lambda^{(i+1)}$, Proposition~\ref{nonzero} implies that
\begin{align}\label{r-e-1}
F_\mathfrak{s}e_{i}F_\mathfrak{s}=e_{i}(\mathfrak{s},\mathfrak{s})F_\mathfrak{s}\ne0.
\end{align}
If $\Phi(L_{i},L_{i+1})$ is a polynomial in $L_{i}$ and $L_{i+1}$ over $\mathbb{F}$, the Jones basic construction shows that 
\begin{align*}
e_{i}\Phi(L_{i},L_{i+1})e_{i}=\zeta e_{i},
\end{align*}
where $\zeta$ is a central element in $A_{i-1}$. In particular, taking $\Phi_\mathfrak{s}(L_{i},L_{i+1})$ to be a polynomial in $L_{i},L_{i+1}$ over $\mathbb{F}$ such that $F_\mathfrak{s}=\Phi_\mathfrak{s}F_\mathfrak{u}$, we obtain 
\begin{align}\label{r-e-2}
e_{i}F_\mathfrak{s}e_{i}=e_{i}\Phi_\mathfrak{s}(L_{i},L_{i+1})e_{k}F_\mathfrak{u}
=F_\mathfrak{u}\zeta_\mathfrak{s}e_{i},
\end{align}
where $\zeta_\mathfrak{s}$ is central in $A_{i-1}$ and $F_\mathfrak{u}\in A_{i-1}$. Using~\eqref{r-e-1} and~\eqref{r-e-2}, we obtain $f_\mathfrak{s}e_{i}F_\mathfrak{s}e_{i}=e_{i}(\mathfrak{s},\mathfrak{s})f_\mathfrak{s}$ and $f_\mathfrak{s}e_{i}F_\mathfrak{s}e_{i}
=f_\mathfrak{s}F_\mathfrak{u} \zeta_\mathfrak{s}e_{i}= f_\mathfrak{s} \zeta_\mathfrak{s}e_{i}$. It follows that $f_\mathfrak{s}\zeta_\mathfrak{s}=e_i(\mathfrak{s},\mathfrak{s})f_\mathfrak{s}$ and $\zeta_\mathfrak{s}$ acts on  $\Delta_{i-1,\mathbb{F}}^{(\lambda,{l}-1)}$ as scalar multiplication by $e_{i}(\mathfrak{s},\mathfrak{s})$. Since $F_\mathfrak{u}=\langle f_\mathfrak{u},f_\mathfrak{u}\rangle^{-1} F_\mathfrak{uu}$, we obtain $F_\mathfrak{u}\zeta_\mathfrak{s}=e_i(\mathfrak{s,s})F_\mathfrak{u}$ which, substituted into~\eqref{r-e-2}, gives the required result. 
\end{proof}
\begin{lemma}\label{s-q}
Let $(\lambda,{l})\in\hat{A}_{i+1}$ and $\mathfrak{s},\mathfrak{t}\in\hat{A}_{i+1}^{(\lambda,l)}$, where $\mathfrak{s}\stackrel{i}{\sim}\mathfrak{t}$.
\begin{enumerate}[label=(\arabic{*}), ref=\arabic{*},leftmargin=0pt,itemindent=1.5em]
\item\label{s-q-0} $e_{i}(\mathfrak{s},\mathfrak{t}) f_\mathfrak{s}e_{i}=e_i(\mathfrak{s},\mathfrak{s}) f_\mathfrak{t}e_i$
\item\label{s-q-1}  $e_{i}(\mathfrak{s},\mathfrak{t})e_{i}(\mathfrak{t},\mathfrak{u})
=e_{i}(\mathfrak{t},\mathfrak{t})e_{i}(\mathfrak{s},\mathfrak{u})$ for all $\mathfrak{u}\in\hat{A}_{k}^{(\lambda,{l})}$.
\end{enumerate}
\end{lemma}
\begin{proof}
Let $\mathfrak{s}=((\lambda^{(0)},l_0),\ldots,(\lambda^{(i+1)},l_{i+1}))\in\hat{A}_{i+1}^{(\lambda,{l})}$. If $\lambda^{(i-1)}\ne\lambda^{(i+1)}$, then $f_\mathfrak{s}e_i=f_\mathfrak{t}e_i=0$ and 
 the statements~\eqref{s-q-0} and~\eqref{s-q-1} hold. We therefore assume that $\lambda^{(i-1)}=\lambda^{(i+1)}$.  By Lemma~\ref{contract}, we have $e_i(\mathfrak{s},\mathfrak{s})f_\mathfrak{t}e_i=f_\mathfrak{t}e_i F_\mathfrak{s}e_i=e_i(\mathfrak{s},\mathfrak{t})f_\mathfrak{s}e_i$, which proves the first item. For the proof of the second item, let $E_\mathfrak{uv}:\Delta_{i+1,\mathbb{F}}^{(\lambda,l)}\to \Delta_{i+1,\mathbb{F}}^{(\lambda,l)}$, for $\mathfrak{u,v}\in\hat{A}_{i+1}^{(\lambda,l)}$ denote the linear map given by
\begin{align*}
f_\mathfrak{r}E_\mathfrak{uv}= \delta_\mathfrak{ru}f_\mathfrak{v},\qquad\text{for all $r\in\hat{A}_{i+1}^{(\lambda,l)}$.}
\end{align*}
Then, as endomorphisms of $\Delta_{i+1,\mathbb{F}}^{(\lambda,l)}$, 
\begin{align*}
e_iF_\mathfrak{t}e_i= e_i(\mathfrak{t},\mathfrak{t})\sum_{\mathfrak{s}\stackrel{i}{\sim}\mathfrak{u}\stackrel{i}{\sim}\mathfrak{t}} e_i(\mathfrak{s},\mathfrak{u})E_\mathfrak{su},
\end{align*}
and 
\begin{align*}
e_iF_\mathfrak{t}e_i= e_iF_\mathfrak{t}^2e_i= \bigg( \sum_{\mathfrak{s}\stackrel{i}{\sim}\mathfrak{t}} e_i(\mathfrak{s,t})E_\mathfrak{st}\bigg) \bigg( \sum_{\mathfrak{u}\stackrel{i}{\sim}\mathfrak{t}} e_i(\mathfrak{t,u})E_\mathfrak{tu}\bigg) = \sum_{\mathfrak{s}\stackrel{i}{\sim}\mathfrak{u}\stackrel{i}{\sim}\mathfrak{t}} e_i(\mathfrak{s},\mathfrak{t})e_i(\mathfrak{t,u})E_\mathfrak{su}.
\end{align*}
Comparing coefficients in the two expressions above gives the required formula.
\end{proof}
\begin{corollary}
Let $(\lambda,l)\in\hat{A}_{i+1}$ and $\mathfrak{t}=((\lambda^{(0)},l_0),\ldots,(\lambda^{(i+1)},l_{i+1}) \in\hat{A}_{i+1}^{(\lambda,l))}$. If  $\lambda^{(i-1)}=\lambda^{(i+1)}$,  then the following statements hold:
\begin{enumerate}[label=(\arabic{*}), ref=\arabic{*},leftmargin=0pt,itemindent=1.5em]
\item  If $i$ is even, then the eigenspace corresponding to the eigenvalue $1$ for the action of $e_i$ on the subspace $\langle f_\mathfrak{s}\in\Delta_{i+1,\mathbb{F}}^{(\lambda,l)}\mid \mathfrak{s}\stackrel{i}{\sim}\mathfrak{t}\rangle$  is one--dimensional. 
\item If $i$ is odd, then the eigenspace corresponding to the  eigenvalue $z$ for the action of $e_i$ on the subspace $\langle f_\mathfrak{s}\in\Delta_{i+1,\mathbb{F}}^{(\lambda,l)}\mid \mathfrak{s}\stackrel{i}{\sim}\mathfrak{t}\rangle$  is one--dimensional. 
\end{enumerate}
\end{corollary}

Lemma~\ref{s-q}\eqref{s-q-0} shows that, granted the diagonal structure constants for the contraction $e_k$, the off--diagonal structure constants of $e_k$ are completely determined by Proposition~\ref{offdiag:1} and Proposition~\ref{offdiag:2} below. The proofs of Proposition~\ref{offdiag:1} and Proposition~\ref{offdiag:2} rely on the branching factors for inner products~\eqref{quotient} which are verified in Sect.~\ref{det-proof}.
\begin{proposition}\label{offdiag:1}
Let $k=2i$ and $(\lambda,l)\in\hat{A}_{k+1}$. Assume that $\mathfrak{s}\in\hat{A}_{k+1}^{(\lambda,l)}$, where $\mathfrak{s}\downarrow_{k-1}=\mathfrak{t}^{(\lambda,l-1)}$ and $\mathfrak{s}^{(k)}=(\lambda,l)$. Let 
\begin{align*}
(\nu,l-1)=\min\lbrace(\rho,r)\in\hat{A}_k\mid(\rho,r)\to(\lambda,l)\text{ in }\hat{A}\rbrace
\end{align*}
and $\mathfrak{u}\in\hat{A}_{k+1}^{(\lambda,l)}$, where $\mathfrak{u}\stackrel{k}{\sim}\mathfrak{s}$ and $\mathfrak{u}^{(k)}=(\nu,l-1)$. Then the following statements hold:
\begin{enumerate}[label=(\arabic{*}), ref=\arabic{*},leftmargin=0pt,itemindent=1.5em]
\item\label{offdiag:1.1} $m_\mathfrak{u}=f_\mathfrak{s}e_k$ and $f_\mathfrak{u}e_k=e_k(\mathfrak{u,u})m_\mathfrak{u}$. 
\item\label{offdiag:1.2} If $\mathfrak{t}\stackrel{k}{\sim}\mathfrak{s}$ and $\mathfrak{t}^{(k)}=(\mu,l-1)$, then  
\[
e_k(\mathfrak{s,t})=e_k( \mathfrak{s,s})e_k(\mathfrak{t,t}) \gamma_{\lambda\to\mu}\qquad\text{and}\qquad e_k(\mathfrak{t,s})=\frac{1}{\gamma_{\lambda\to\mu}}. 
\]
\end{enumerate}
\end{proposition}
\begin{proof}
\eqref{offdiag:1.1} By definition,
\begin{align*}
p^{(k+1)}_{(\lambda,l)\to(\lambda,l)}&=1, &p^{(k)}_{(\lambda,l-1)\to(\lambda,l)}&=w_{l,i}, & \text{and}& &p_\mathfrak{s}&=w_{l,i},
\intertext{while}
p^{(k+1)}_{(\nu,l-1)\to(\lambda,l)}&=w_{l,i}e_k, & p^{(k)}_{(\lambda,l-1)\to(\nu,l-1)}&=1, &\text{and}& & p_\mathfrak{u}&=w_{l,i}e_k. 
\end{align*}
It follows that $m_\mathfrak{s}e_k=m_\mathfrak{u}$. If we write $m_\mathfrak{s}=f_\mathfrak{s}+\sum_{\mathfrak{v}\succ\mathfrak{s}}r_\mathfrak{v}f_\mathfrak{v}$, then Proposition~\ref{nonzero} and the maximality of $\mathfrak{s}$ in $\lbrace \mathfrak{v}\in\hat{A}_{k+1}^{(\lambda,l)}\mid \mathfrak{v}\stackrel{k}{\sim}\mathfrak{s}\rbrace$ give 
\begin{align*}
m_\mathfrak{u}=m_\mathfrak{s}e_k=f_\mathfrak{s}e_k+\sum_{\mathfrak{v}\succ\mathfrak{s}}r_\mathfrak{v}f_\mathfrak{v}e_k=f_\mathfrak{s}e_k. 
\end{align*}
For the second equality, we have $
e_k(\mathfrak{u,u})m_\mathfrak{u}=e_k(\mathfrak{u,u})f_\mathfrak{s}e_k=f_\mathfrak{s}e_kF_\mathfrak{u}e_k=m_\mathfrak{u}F_\mathfrak{u}e_k=f_\mathfrak{u}e_k$. 

\eqref{offdiag:1.2}  Assume that $\mu=\lambda\cup\lbrace(j,\mu_j)\rbrace$ and let $a=l+\sum_{r=0}^j\lambda_r$. Then
\begin{align*}
p_{(\mu,l-1)\to(\lambda,l)}^{(k+1)}=w_{l,i}e_kw_{i,a}\sum_{r=0}^{\lambda_j}w_{a,a-r}.
\end{align*}
Let $\mathfrak{v}\in\hat{A}_{k+1}^{(\lambda,l)}$, where $\mathfrak{v'}=\mathfrak{v}\downarrow_{k}=\mathfrak{t}^{(\mu,l-1)}\in\hat{A}_k^{(\mu,l-1)}$. Let $\Psi_{\mathfrak{v},k}$ be a polynomial in $L_{k+1}$ where $F_\mathfrak{v}=F_\mathfrak{v'}\Psi_{\mathfrak{v},k}$. Computing the inner product $\langle f_\mathfrak{v},f_\mathfrak{v}\rangle$ yields 
\begin{align}
a_{(\lambda,l)}^{(k+1)}p_\mathfrak{v}F_\mathfrak{v} p_\mathfrak{v}^*a_{(\lambda,l)}^{(k+1)}&= e_{2l-1}w_{l,i}e_kw_{i,a}a_{(\mu,l-1)}^{(k)} F_\mathfrak{v'} a_{(\mu,l-1)}^{(k)} \Psi_{\mathfrak{v},k} w_{a,i}e_kw_{i,l}e_{2l-1}\notag\\
&\equiv \langle f_\mathfrak{v'},f_\mathfrak{v'}\rangle e_{2l-1}w_{l,i}e_kw_{i,a}a_{(\mu,l-1)}^{(k)}F_\mathfrak{v}w_{a,i}e_kw_{i,l}e_{2l-1}\notag\\
&\equiv \langle f_\mathfrak{v'},f_\mathfrak{v'}\rangle a_{(\lambda,l)}^{(k+1)}p_\mathfrak{v}F_\mathfrak{v}w_{a,i}e_kw_{i,l}e_{2l-1}\mod A_{k+1}^{\rhd(\lambda,l)}.\label{char}
\end{align}
By Lemma~\ref{obs:5}, we may write
\begin{align*}
f_\mathfrak{v}w_{a,i}=f_\mathfrak{t}+\sum_\mathfrak{a}r_\mathfrak{a}f_\mathfrak{a},
\end{align*}
where the sum is over $\mathfrak{a}\in\hat{A}_{k+1}^{(\lambda,l)}$ such that $\Shape(\mathfrak{a}\downarrow_{k-1})\ne(\lambda,l-1)$.  The congruence~\eqref{char} therefore gives 
\begin{align*}
\gamma_{(\mu,l-1)\to(\lambda,l)}^{(k+1)}a_{(\lambda,l)}^{(k+1)}&\equiv a_{(\lambda,l)}^{(k+1)}p_\mathfrak{t}F_\mathfrak{t}e_kw_{i,l}e_{2l-1}\equiv \frac{e_k(\mathfrak{s,t})}{e_k(\mathfrak{s,s})}a_{(\lambda,l)}^{(k+1)}p_\mathfrak{s}F_\mathfrak{s}e_kw_{i,l}e_{2l-1}\\ 
&\equiv \frac{e_k(\mathfrak{s,t})}{e_k(\mathfrak{s,s})}a_{(\lambda,l)}^{(k+1)}p_\mathfrak{u}w_{i,l}e_{2l-1} \equiv\frac{e_k(\mathfrak{s,t})}{e_k(\mathfrak{s,s})}a_{(\lambda,l)}^{(k+1)} \mod A_{k+1}^{\rhd(\lambda,l)}. 
\end{align*}
Since $\gamma_{(\mu,l-1)\to(\lambda,l)}^{(k+1)}=e_k(\mathfrak{t,t})\gamma_{\lambda\to\mu}$ the stated formula for $e_k(\mathfrak{s,t})$ follows. The formula for $e_k(\mathfrak{t,s})$ is  obtained from the relation $e_k(\mathfrak{s,t})e_k(\mathfrak{t,s})= e_k(\mathfrak{s,s})e_k(\mathfrak{t,t})$. 
\end{proof}
\begin{proposition}\label{offdiag:2}
Let $k=2i+1$ and $(\lambda,l)\in\hat{A}_{k+1}$. Assume that 
\begin{align*}
(\nu,l)=\max\lbrace(\rho,r)\in\hat{A}_k\mid(\rho,r)\to(\lambda,l)\text{ in }\hat{A}\rbrace
\end{align*}
and  $\mathfrak{s}\in\hat{A}_{k+1}^{(\lambda,l)}$, where $\mathfrak{s}\downarrow_{k}=\mathfrak{t}^{(\lambda,l-1)}$ and $\mathfrak{s}^{(k)}=(\nu,l)$. Let $\mathfrak{u}\in\hat{A}_{k+1}^{(\lambda,l)}$, where $\mathfrak{u}\stackrel{k}{\sim}\mathfrak{s}$ and $\mathfrak{u}^{(k)}=(\lambda,l-1)$. Then the following statements hold:
\begin{enumerate}[label=(\arabic{*}), ref=\arabic{*},leftmargin=0pt,itemindent=1.5em]
\item\label{offdiag:2.1} $f_\mathfrak{s}e_k=\gamma_{\nu\to\lambda}m_\mathfrak{u}$ and $f_\mathfrak{u}e_k=\gamma_{\nu\to\lambda}e_k(\mathfrak{u,u})m_\mathfrak{u}$. 
\item\label{offdiag:2.2} If $\mathfrak{t}\stackrel{k}{\sim}\mathfrak{s}$ and $\mathfrak{t}^{(k)}=(\mu,l)$, then 
\[
e_k(\mathfrak{s,t})=\frac{\gamma_{\mu\to\lambda}}{\gamma_{\nu\to\lambda}}e_k( \mathfrak{s,s})\qquad \text{and}\qquad e_k(\mathfrak{t,s})=\frac{\gamma_{\nu\to\lambda}}{\gamma_{\mu\to\lambda}}e_k(\mathfrak{t,t}). 
\]
\end{enumerate}
\end{proposition}
\begin{proof}
\eqref{offdiag:2.1} Assume that $\lambda=(\lambda_1,\ldots,\lambda_t)$, where $\lambda_t\ne0$; then $\nu=(\lambda_1,\ldots,\lambda_t-1)$. By definition,
\begin{align*}
p_{(\nu,l)\to(\lambda,l)}^{(k+1)}&=1, & \text{and}& &p_{(\lambda,l-1)\to(\nu,l)}^{(k)}&= w_{l,i}e_{k-1}\sum_{r=0}^{\nu_t}w_{i,i-r},
\intertext{and} 
p_{(\lambda,l-1)\to(\lambda,l)}^{(k+1)}&=w_{l,i+1} &\text{and}& &p_{(\lambda,l-1)\to(\lambda,l)}^{(k)}&=1. 
\end{align*}
By direct calculation, 
\begin{align*}
a_{(\lambda,l)}^{(k+1)}p_\mathfrak{s}e_k&= a_{(\lambda,l)}^{(k+1)}w_{l,i}e_ke_{k-1}\sum_{r=0}^{\nu_t}w_{i,i-r}=a_{(\lambda,l)}^{(k+1)}w_{l,i+1}\sum_{r=0}^{\nu_t}w_{i,i-r}\\ 
&=\lambda_ta_{(\lambda,l)}^{(k+1)}w_{l,i+1}=\gamma_{\nu\to\lambda}a_{(\lambda,l)}^{(k+1)}p_\mathfrak{u}.
\end{align*}
Thus we have $m_\mathfrak{s}e_k=\gamma_{\nu\to\lambda}m_\mathfrak{u}$. If we write $m_\mathfrak{s}=f_\mathfrak{s}+\sum_{\mathfrak{v}\succ\mathfrak{s}}r_\mathfrak{v}f_\mathfrak{v}$, then Proposition~\ref{nonzero} and the maximality of $\mathfrak{s}$ in $\lbrace \mathfrak{v}\in\hat{A}_{k+1}^{(\lambda,l)}\mid \mathfrak{v}\stackrel{k}{\sim}\mathfrak{s}\rbrace$ give 
\begin{align*}
\gamma_{\nu\to\lambda}m_\mathfrak{u}=m_\mathfrak{s}e_k=f_\mathfrak{s}e_k+\sum_{\mathfrak{v}\succ\mathfrak{s}}r_\mathfrak{v}f_\mathfrak{v}e_k=f_\mathfrak{s}e_k. 
\end{align*}
For the second equality, we have $
\gamma_{\nu\to\lambda} e_k(\mathfrak{u,u})m_\mathfrak{u}=e_k(\mathfrak{u,u})f_\mathfrak{s}e_k=f_\mathfrak{s}e_kF_\mathfrak{u}e_k=m_\mathfrak{u}F_\mathfrak{u}e_k=f_\mathfrak{u}e_k$. 

\eqref{offdiag:2.2} Let $\mathfrak{t}\in\hat{A}_{k+1}^{(\lambda,l)}$, where $\mathfrak{t}\stackrel{k}{\sim}\mathfrak{s}$ and $\mathfrak{t}^{(k)}=(\mu,l)$. By part~\eqref{offdiag:2.1}, we have 
\begin{align*}
\gamma_{\nu\to\lambda}m_\mathfrak{u}=\gamma_{\nu\to\lambda}f_\mathfrak{u}+\sum_{\mathfrak{v}\succ\mathfrak{u}}e_k(\mathfrak{v,s})f_\mathfrak{v},
\end{align*}
where the sum is over $\mathfrak{v}\in\hat{A}_{k+1}^{(\lambda,l)}$ such that $\mathfrak{v}\stackrel{k}{\sim}\mathfrak{s}$. Hence $\gamma_{\nu\to\lambda}\langle m_\mathfrak{u},f_\mathfrak{t}\rangle = \langle f_\mathfrak{t},f_\mathfrak{t}\rangle e_k(\mathfrak{t,s})$. We determine $\langle m_\mathfrak{u},f_\mathfrak{t}\rangle$.  It will be convenient to work with the elements $m_\mathfrak{uu}$ and $m_\mathfrak{ut}$; to this end we note that $m_\mathfrak{uu}=a_{(\lambda,l-1)}^{(k-1)}e_k$.  If $\lambda=\mu\cup\lbrace (j,\lambda_j)\rbrace$ and $a=l+\sum_{r=0}^{j}\lambda_r$, then  the relations $s_ie_{k-1}=e_{k-1}$ and $w_{a,i+1}w_{l,i+1}=w_{l,i+1}w_{a-1,i}$ and the branching coefficients
\begin{align*}
p_{(\mu,l)\to(\lambda,l)}^{(k+1)}&=w_{a,i+1}, &\text{and}& &p_{(\lambda,l-1)\to(\mu,l)}^{(k)}&=w_{l,i}e_{k-1}w_{i,a-1}\sum_{r=0}^{\mu_j}w_{a-1,a-1-r}
\end{align*}
allow us to write
\begin{align*}
m_\mathfrak{ut}
&=w_{i+1,l}a_{(\lambda,l)}^{(k+1)}w_{l,i+1}w_{a-1,i}e_{k-1}w_{i,a-1}\sum_{r=0}^{\mu_j}w_{a-1,a-1-r} \\ 
&= a_{(\lambda,l-1)}^{(k-1)}e_kw_{a-1,i}e_{k-1}w_{i,a-1}\sum_{r=0}^{\mu_j}w_{a-1,a-1-r}. 
\end{align*}
By definition, $\langle m_\mathfrak{u},f_\mathfrak{t}\rangle m_\mathfrak{uu}= m_\mathfrak{ut}F_\mathfrak{t} m_\mathfrak{uu}$. Thus we obtain 
\begin{align}
m_\mathfrak{ut}F_\mathfrak{t} m_\mathfrak{uu}&= \langle f_\mathfrak{t},f_\mathfrak{t}\rangle^{-1}a_{(\lambda,l-1)}^{(k-1)}e_k w_{a-1,i}e_{k-1}w_{i,a-1}\sum_{r=0}^{\mu_j}w_{a-1,a-1-r}F_\mathfrak{tt}e_ka_{(\lambda,l-1)}^{(k-1)}\notag\\
&\equiv \frac{\lambda_j}{\langle f_\mathfrak{t},f_\mathfrak{t}\rangle} a_{(\lambda,l-1)}^{(k-1)}e_k w_{a-1,i}e_{k-1}w_{i,a-1}F_\mathfrak{tt}e_ka_{(\lambda,l-1)}^{(k-1)}\notag\\
&\equiv \frac{\lambda_j}{\langle f_\mathfrak{t},f_\mathfrak{t}\rangle} a_{(\lambda,l-1)}^{(k-1)}e_k w_{a-1,i}e_{k-1}F_\mathfrak{vt}e_ka_{(\lambda,l-1)}^{(k-1)}\mod A_{k+1}^{\rhd(\lambda,l)},\label{ref}
\end{align}
where we have used Lemma~\ref{obs:5} and written $\mathfrak{v}= (\cdots ((\mathfrak{t}\sigma_{2a-1})\sigma_{2a+1})\cdots)\sigma_{2i-1}$. By Proposition~\ref{nonzero}, we have $f_\mathfrak{v}e_{k-1}e_k=e_{k-1}(\mathfrak{v,v})f_\mathfrak{v}e_k$. Thus, writing $\mathfrak{r}=\mathfrak{t}\downarrow_{k-1}$ and using Corollary~\ref{obs:7}, the congruence~\eqref{ref} becomes 
\begin{align*}
m_\mathfrak{ut}F_\mathfrak{t} m_\mathfrak{uu}&\equiv e_{k-1}(\mathfrak{v,v})\frac{\lambda_j}{\langle f_\mathfrak{t},f_\mathfrak{t}\rangle} a_{(\lambda,l-1)}^{(k-1)} w_{a-1,i}e_{k}F_\mathfrak{vt}e_ka_{(\lambda,l-1)}^{(k-1)}\\
&\equiv e_{k-1}(\mathfrak{v,v})\frac{\gamma_{\mu\to\lambda}}{\langle f_\mathfrak{t},f_\mathfrak{t}\rangle} a_{(\lambda,l-1)}^{(k-1)} e_{k}F_\mathfrak{tt}e_ka_{(\lambda,l-1)}^{(k-1)}\\
&\equiv\gamma_{\mu\to\lambda} e_{k-1}(\mathfrak{v,v})e_k(\mathfrak{t,t})\langle f_\mathfrak{r},f_\mathfrak{r}\rangle m_\mathfrak{uu}\mod A_{k+1}^{\rhd(\lambda,l)}. 
\end{align*}
Since $e_{k-1}(\mathfrak{v,v})\gamma_{\mu\to\lambda}=\gamma^{(k)}_{(\lambda,l-1)\to(\mu,l)}$, the formula for $e_k(\mathfrak{t,s})$ follows from Proposition~\ref{branching}. The formula for $e_k(\mathfrak{s,t})$ is obtained from the relation  $e_k(\mathfrak{s,t})e_k(\mathfrak{t,s})= e_k(\mathfrak{s,s})e_k(\mathfrak{t,t})$. 
\end{proof}
The proof of Theorem~\ref{sigmaodd}\eqref{sigmaodd.5} and Theorem~\ref{sigmaeven}\eqref{sigmaeven.5} will use the following observation.
\begin{lemma}\label{e-v-b}
Let $(\lambda,{l})\in\hat{A}_{2i+1}$. If  $\mathfrak{t}=((\lambda^{(0)},l_0),\ldots,(\lambda^{(2i+1)},l_{2i+1}))\in\hat{A}_{2i+1}^{(\lambda,{l})}$, where $\lambda^{(2i-2)}=\lambda^{(2i-1)}\subsetneq \lambda^{(2i)}=\lambda^{(2i+1)}$, then $m_\mathfrak{t}\sigma_{2i} =m_\mathfrak{t}$.
\end{lemma}
\begin{proof}
Assume that  $\lambda=\lambda^{(2i-1)}\cup{\lbrace}(j,\lambda_j){\rbrace}$. Then, $p_{\mathfrak{t}^{(2i)}\to\mathfrak{t}^{(2i+1)}} =p_{\mathfrak{t}^{(2i-2)}\to\mathfrak{t}^{(2i-1)}}=1$ and $p_{\mathfrak{t}^{(2i-1)}\to\mathfrak{t}^{(2i)}}=w_{a,i}$, where $a= l +\sum_{r=1}^j\lambda_r$.  
For $j=0,1,\ldots,$ define
\begin{align*}
\sigma_{2j}^{(0)}=\sigma_{2j}&&\text{and}&&
\sigma_{2j}^{(r)}=w_{r,r+j+1}\sigma_{2j}^{(r-1)}w_{r+j+1,r}&&\text{for $r=1,2,\ldots.$}
\end{align*}
Since ${\sigma}_{2}=1,$ using induction, we obtain the relations
\begin{align*}
e_{2r-1}\sigma_{2j}^{(r)}=
\begin{cases}
e_{2r-1},&\text{if $j=1$,}\\
e_{2r-1}\sigma_{2j-2}^{(r+1)},
&\text{if $j\ge2.$}
\end{cases}
\end{align*}
As $\sigma_{2i-2l}\mapsto 1$ under the map $A_{2i-2l+1}\to \mathbb{F}\mathfrak{S}_{i-l}$, we have
\begin{align*}
e_{2l-1}^{(l)}w_{a,i}\sigma_{2i}=w_{a,i}e_{2l-1}^{(l)}\sigma_{2i-2l}^{(l+1)}\equiv e_{2l-1} ^{(l)} w_{a,i} \mod A_{2i+1}^{\rhd(\lambda, l )},
\end{align*}
and $a_{(\lambda,l)}p_\mathfrak{t}\sigma_{2i}\equiv a_{(\lambda,l)}p_\mathfrak{t}\mod A_{2i+1}^{\rhd(\lambda,l)}$. 
\end{proof}

\begin{proof}[Proof of Theorem~\ref{contractioneven}]
\eqref{contractioneven.1} If $\mathfrak{s}\in\hat{A}_{k+1}^{(\lambda,l)}$, where $\mathfrak{s}^{(k-1)}\ne(\lambda,l-1)$, then $e_k(\mathfrak{s,t})=0$ by Proposition~\ref{nonzero}.

\eqref{contractioneven.2} We assume that $\mathfrak{s}^{(k-1)}=(\lambda,l)$ and prove the formulae for the off-diagonal structure constants. There are three cases to consider:

{\textsc{Case~1.}} If $\mathfrak{s}^{(k)}=(\rho,l-1)$ and $\mathfrak{t}^{(k)}=(\lambda,l)$, then $e_k(\mathfrak{s,t})=\gamma_{\lambda\to\rho}^{-1}$, by Proposition~\ref{offdiag:1}\eqref{offdiag:1.2}. 

{\textsc{Case~2.}} If $\mathfrak{s}^{(k)}=(\lambda,l)$ and $\mathfrak{t}^{(k)}=(\mu,l-1)$, then $e_k(\mathfrak{s,t})=e_k(\mathfrak{s,s}) e_k(\mathfrak{t,t})\gamma_{\lambda\to\mu}$, by Proposition~\ref{offdiag:1}\eqref{offdiag:1.2}. 

{\textsc{Case~3.}} Assume that $\mathfrak{s}^{(k)}=(\rho,l-1)$ and $\mathfrak{t}^{(k)}=(\mu,m-1)$. Let $\mathfrak{v}\in\hat{A}_{k+1}^{(\lambda,l)}$, where $\mathfrak{v}\stackrel{k}{\sim}\mathfrak{t}$ and $\mathfrak{v}^{(k)}=(\lambda,l)$. Then, by the two cases above, $e_k(\mathfrak{s,v})=\gamma_{\lambda\to\rho}^{-1}$ and $e_k(\mathfrak{v,t})=e_k(\mathfrak{v,v}) e_k(\mathfrak{t,t})\gamma_{\lambda\to\mu}$. The relation $e_k(\mathfrak{s,v})e_k(\mathfrak{v,t})= e_k(\mathfrak{v,v})e_k(\mathfrak{s,t})$ now yields the formula for $e_k(\mathfrak{s,t})$ given in the third case in~\eqref{contractioneven.2.1}. 

The diagonal structure constants $e_k(\mathfrak{t,t})$ may be deduced from Lemma~\ref{r-e-s-i} and the relation~\eqref{c-f-r-1}. Alternatively, Leduc and Ram~\cite[Theorem~3.12]{MR1427801} express the diagonal structure constants $e_k(\mathfrak{t,t})$ for the contraction $e_k$ in terms of the weights of the Markov trace on $A_{k+1}$. Applying the formulae for the weights of the Markov trace given by Halverson and Ram~\cite[Proposition~3.24]{MR2143201}, we obtain expressions for the diagonal terms $e_k(\mathfrak{t},\mathfrak{t})$.
\end{proof}
\begin{proof}[Proof of Theorem~\ref{contractionodd}] 
\eqref{contractionodd.1} If $\mathfrak{s}\in\hat{A}_{k+1}^{(\lambda,l)}$, where $\mathfrak{s}^{(k-1)}\ne(\lambda,l-1)$, then $e_k(\mathfrak{s,t})=0$ by Proposition~\ref{nonzero}. 

\eqref{contractionodd.2} Assume that $\mathfrak{s}\stackrel{k}{\sim}\mathfrak{t}$ and $\mathfrak{s}^{(k)}=(\lambda,l-1)$ and denote $(\nu,l)=\min\lbrace(\rho,r)\mid (\rho,r)\to(\lambda,l)\in\hat{A}\rbrace$. Let $\mathfrak{v}\in\widehat{A}_{k+1}^{(\lambda,l)}$, where $\mathfrak{v}\stackrel{k}{\sim}\mathfrak{t}$ and $\mathfrak{v}^{(k)}=(\nu,l)$. 

{\textsc{Case~1.}} Let $\mathfrak{s}^{(k)}=(\lambda,l-1)$ and $\mathfrak{t}^{(k)}=(\mu,l)$. By Proposition~\ref{offdiag:2}\eqref{offdiag:2.1} and Proposition~\ref{offdiag:2}\eqref{offdiag:2.2} respectively, we obtain $e_k(\mathfrak{s,v})=\gamma_{\nu\to\lambda}$ and $e_k(\mathfrak{v,t})=\gamma_{\mu\to\lambda}\gamma_{\nu\to \lambda}^{-1}e_k(\mathfrak{v,v})$. Thus $e_k(\mathfrak{s,t})=\gamma_{\mu\to\lambda}$. 

{\textsc{Case~2.}} Assume that $\mathfrak{s}^{(k)}=(\rho,l)$ and $\mathfrak{t}^{(k)}=(\lambda,l-1)$. From the previous case $e_k(\mathfrak{t,s})=\gamma_{\rho\to\lambda}$, whence  $e_k(\mathfrak{s,t})=e_k(\mathfrak{s,s}) e_k(\mathfrak{t,t})\gamma_{\rho\to\lambda}^{-1}$. 

{\textsc{Case~3.}} Assume that $\mathfrak{s}^{(k)}=(\rho,l)$ and $\mathfrak{t}^{(k)}=(\mu,l)$. By Proposition~\ref{offdiag:2}\eqref{offdiag:2.2},
\begin{align*}
e_k(\mathfrak{v,t})=\frac{\gamma_{\mu\to\lambda}}{ \gamma_{\nu\to\lambda}}e_k(\mathfrak{v,v})\qquad\text{and}\qquad e_k(\mathfrak{s,v})=\frac{\gamma_{\nu\to\lambda}}{\gamma_{\rho\to\lambda}}e_k(\mathfrak{s,s}).
\end{align*}
Hence the formula for $e_k(\mathfrak{s,t})$ given in the third case in~\eqref{contractionodd.2.1} follows. 

Finally, as in the proof of Theorem~\ref{contractioneven}, the diagonal structure constants for $e_k$, for $k$ odd, may be obtained using the results of~\cite{MR1427801,MR2143201}, or via Lemma~\ref{r-e-s-i} and the relation~\eqref{c-f-r-2}. 
\end{proof}
\begin{proof}[Proof of Theorem~\ref{sigmaeven}] 
\eqref{sigmaeven.1} Assume that  $\lambda^{(k+1)}=\lambda^{(k-1)}$ and $\lambda^{(k)}=\lambda^{(k-2)}$. Proposition~\ref{t-0}\eqref{t-0-2} and Proposition~\ref{nonzero} give
\begin{align*}
(c_\mathfrak{s}(k+1)-c_\mathfrak{t}(k-1))\sigma_{k}(\mathfrak{s,t})=\delta_\mathfrak{st}+(z-c_\mathfrak{t}(k-1)-c_\mathfrak{s}(k)-e_{k-1}(\mathfrak{t,t})) e_k(\mathfrak{s,t})
\end{align*}
Observe that $c_\mathfrak{s}(k+1)=c_\mathfrak{t}(k-1)$ only if $\mathfrak{s}=\mathfrak{t}$; hence we obtain the expression~\eqref{sigmaeven.1.1} for $\sigma_k(\mathfrak{s,t})$ when $\mathfrak{s}\ne\mathfrak{t}$. From Proposition~\ref{nonzero} and the relation $L_ke_k=\sigma_ke_{k-1}e_k$, we obtain
\begin{align*}
c_\mathfrak{t}(k)f_\mathfrak{t}e_k=f_\mathfrak{t}\sigma_ke_{k-1}e_k= \sum_{\mathfrak{u}\stackrel{k-1}{\sim}\mathfrak{s}\stackrel{k-1}{\sim}\mathfrak{t}}e_{k-1}(\mathfrak{u,s})\sigma_k(\mathfrak{s,t})f_\mathfrak{u}e_k=e_{k-1}(\mathfrak{t,t})\sigma_k(\mathfrak{t,t})f_\mathfrak{t}e_k
\end{align*}
which gives the expression for $\sigma_k(\mathfrak{t,t})$.

\eqref{sigmaeven.2} Assume that  $\lambda^{(k+1)}\ne \lambda^{(k-1)}$ and $\lambda^{(k)}=\lambda^{(k-2)}$. Together with Lemma~\ref{s-q} and Proposition~\ref{nonzero}, the relation $L_ke_k=\sigma_ke_{k-1}e_k$ gives
\begin{align*}
c_\mathfrak{t}(k)f_\mathfrak{t}e_{k}=c_\mathfrak{t}(k)e_k(\mathfrak{v,t})e_{k}(\mathfrak{v,v})^{-1}f_\mathfrak{v}e_k=\sigma_{k}(\mathfrak{v,t})e_{k-1}(\mathfrak{v,v})f_\mathfrak{v}e_k.
\end{align*}
Since $e_{k-1}(\mathfrak{v,v}) e_k(\mathfrak{v,v})=1$, the the expression~\eqref{sigmaeven.2.2} for $\sigma_k(\mathfrak{v,t})$  follows. Next, Proposition~\ref{t-0}\eqref{t-0-2} and Proposition~\ref{nonzero} give
\begin{align*}
(c_\mathfrak{s}(k+1)-c_\mathfrak{t}(k-1))\sigma_k(\mathfrak{s,t})=\delta_\mathfrak{st}-e_{k-1}(\mathfrak{v,t}) e_k(\mathfrak{s,v}).
\end{align*}
Using the fact $\lbrace \mathfrak{s}\in \hat{A}_{k+1}^{(\lambda,l)}\mid\mathfrak{s}\stackrel{k}{\approx}\mathfrak{t}\text{ and }c_\mathfrak{s}(k+1)=c_\mathfrak{t}(k-1)\rbrace=\lbrace\mathfrak{t}\rbrace$, which is readily verified from the assumptions on $\mathfrak{t}$, we obtain the expression~\eqref{sigmaeven.2.1} for $\sigma_k(\mathfrak{s,t})$. 

\eqref{sigmaeven.3} Assume that  $\lambda^{(k+1)}=\lambda^{(k-1)}$ and $\lambda^{(k)}\ne\lambda^{(k-2)}$. By Proposition~\ref{t-0}\eqref{t-0-2} and Proposition~\ref{nonzero},
\begin{align}\label{differentk}
(c_\mathfrak{s}(k+1)-c_\mathfrak{t}(k-1))\sigma_k(\mathfrak{s,t})=\delta_\mathfrak{st}+(z-c_\mathfrak{t}(k-1)-c_\mathfrak{s}(k))e_k(\mathfrak{s,t}).
\end{align}
Note that $c_\mathfrak{s}(k+1)=c_\mathfrak{t}(k-1)$ if and only if $\rho^{(k)}=\lambda^{(k-2)}$; hence, when $\rho^{(k)}\ne\lambda^{(k-2)}$, the relation~\eqref{differentk} yields the expression~\eqref{sigmaeven.3.1} for $\sigma_k(\mathfrak{s,t})$. If $\rho^{(k)}=\lambda^{(k-2)}$, we use the relation $\langle f_\mathfrak{t}\sigma_k,f_\mathfrak{s}\rangle =\langle f_\mathfrak{t},f_\mathfrak{s}\sigma_k\rangle$ to obtain the expression~\eqref{sigmaeven.3.2} for $\sigma_k(\mathfrak{s,t})$. 

\eqref{sigmaeven.4} Assume  that $\lambda^{(k+1)}\ne\lambda^{(k-1)}$ and $\lambda^{(k)}\ne\lambda^{(k-2)}$, where $\mathfrak{t}\sigma_k$ does not exist.  There are two cases to consider here.

{\textsc{Case~1.}} If $\lambda^{(k+1)}\ominus\lambda^{(k-2)}=\lbrace \alpha,\beta\rbrace$, where $\alpha,\beta$ are either in the same row or the same column, then $\lambda^{(k-2)}\supsetneq \lambda^{(k-1)}=\lambda^{(k)}\supsetneq\lambda^{(k+1)}$. Thus $c_\mathfrak{t}(k+1)\ne c_\mathfrak{t}(k-1)$ and the relation~\eqref{sigmaeven.4.1} follows from  Lemma~\ref{involutions} and Proposition~\ref{t-0}\eqref{t-0-1}. 

 {\textsc{Case~2.}}  If $\lambda^{(k+1)}\ominus\lambda^{(k-2)}$ consists of a single node then $\lambda^{(k-2)}=\lambda^{(k-1)}\subsetneq\lambda^{(k)}=\lambda^{(k+1)}$. Thus $c_\mathfrak{t}(k+1)\ne c_\mathfrak{t}(k-1)$ and the relation~\eqref{sigmaeven.4.1} follows from  Lemma~\ref{onebox} and Proposition~\ref{t-0}\eqref{t-0-1}. 

\eqref{sigmaeven.5} Suppose that $\lambda^{(k+1)}\ne\lambda^{(k-1)}$ and $\lambda^{(k)}\ne\lambda^{(k-2)}$, and that $\mathfrak{t}\sigma_{k}$ exists. By Lemma~\ref{involutions}, 
$\sigma_k(\mathfrak{s,t})=0$ unless $\mathfrak{s}\in\lbrace\mathfrak{t},\mathfrak{t}\sigma_{k}\rbrace$. Furthermore, the assumptions on $\mathfrak{t}$ imply that $f_\mathfrak{t} e_{k-1}=f_\mathfrak{t} e_{k}=0$.
Thus, Proposition~\ref{t-0}\eqref{t-0-2} yields $
(c_\mathfrak{t}(k+1)-c_\mathfrak{t}(k-1))\sigma_{k}(\mathfrak{t},\mathfrak{t})=1$, 
and the expression for the diagonal structure constant $\sigma_{k}(\mathfrak{t},\mathfrak{t})$ follows.

Next, we show that if $\mathfrak{s}=\mathfrak{t}\sigma_{k}$ and $\mathfrak{t}\succ\mathfrak{s}$, then $\sigma_{k}(\mathfrak{s},\mathfrak{t})=1$. There are two cases to consider.

{\textsc{Case 1.}} Assume that $\lambda^{(k-2)}\supsetneq\lambda^{(k-1)}$ and $\lambda^{(k-1)}\subsetneq\lambda^{(k)}=\lambda^{(k+1)}$. Since we may write
\begin{align*}
f_\mathfrak{t} \sigma_{k}&=
m_{\mathfrak{t}^{(\lambda,l)}} p_\mathfrak{t}\sigma_{k}+
\sum_{\mathfrak{u}\succ\mathfrak{t}}r_\mathfrak{u}f_\mathfrak{u} \sigma_{k},
\end{align*}
it is sufficient to demonstrate that
\begin{align}\label{b-r-1}
\sigma_{k}(\mathfrak{s,u})=0
\quad\text{whenever $\mathfrak{u}\in\hat{A}_{k+1}^{(\lambda, l )}$ and $\mathfrak{u}\succ\mathfrak{t}$,}
\end{align}
and 
\begin{align}\label{b-r-2}
m_{\mathfrak{t}^{(\lambda,l)}} p_\mathfrak{t}\sigma_{k}=
m_{\mathfrak{t}^{(\lambda,l)}} p_\mathfrak{s}.
\end{align}
To prove~\eqref{b-r-1}, suppose that $\mathfrak{u}\succ\mathfrak{t}$. If $\sigma_{k}(\mathfrak{s},\mathfrak{u})\ne0$, then $\mathfrak{u}\stackrel{k}{\approx}\mathfrak{s}$ and  $\mathfrak{u}\in{\lbrace}\mathfrak{t},\mathfrak{s}{\rbrace}$, in contradiction to the assumption that  $\mathfrak{u}\succ\mathfrak{t}\succ\mathfrak{s}$. Thus~\eqref{b-r-1} holds. We therefore verify~\eqref{b-r-2}. Let $k=2i$ and, to simplify notation, write 
\begin{align*} 
\big(\lambda^{(k-2)},\lambda^{(k-1)},\lambda^{(k)},\lambda^{(k+1)}\big)=(\nu,\mu,\lambda,\lambda),
\end{align*}
where $\lambda=\mu\cup{\lbrace}(j,\lambda_j){\rbrace}$ and $\nu=\mu\cup{\lbrace}({j'},\nu_{j'}){\rbrace}$. Let $\mathfrak{s}=((\rho^{(0)},r_0),\ldots,(\rho^{(k+1)},r_{k+1}))\in \hat{A}_{k+1}^{(\lambda,l)}$, where $\mathfrak{s}=\mathfrak{t}\sigma_k$ and $\mathfrak{t}\succ\mathfrak{s}$. We write 
\begin{align*}
\big(\rho^{(k-2)},\rho^{(k-1)},\rho^{(k)},\rho^{(k+1)}\big)
=(\nu,\nu,\upsilon,\lambda),
\end{align*}
where $\upsilon=\nu\cup{\lbrace}(j,\lambda_j){\rbrace}=\lambda\cup{\lbrace}({j'},\nu_{j'}){\rbrace}$. There is no loss of generality in assuming that $\mathfrak{t}\downarrow_{k-2}$ is maximal in $\hat{A}_{k-2}^{(\nu,l-1)}$. In this case, the branching coefficients for $\mathfrak{t}$ are 
\begin{align*}
p_{(\lambda,l)\to(\lambda,l)}^{(k+1)}=1\qquad\text{and}\qquad
p_{(\mu,l)\to(\lambda,l)}^{(k)}=w_{a_j,k},
\end{align*}
where $a_j= l +\sum_{r=1}^j\lambda_j$, and 
\begin{align*}
p_{(\nu,l-1)\to(\mu,l)}^{(k-1)}
=w_{l,i-1}p_{k-2}w_{i-1,a_{j'}}\sum_{r=0}^{\mu_{j'}}w_{a_{j'},a_{j'}-r},
\end{align*}
where $a_{j'}= l -1+\sum_{r=1}^{j'}\nu_r$. 
The branching coefficients for $\mathfrak{s}$ are 
\begin{align*}
p_{(\upsilon,l-1)\to(\lambda,l)}^{(k+1)}=
w_{l,i}e_{k}w_{i,\alpha_{j'}}\sum_{r=0}^{\lambda_{j'}}
w_{\alpha_{j'},\alpha_{j'}-r},
\end{align*}
where $\alpha_{j'}= l -1+\sum_{r=1}^{j'}\upsilon_r$, and 
\begin{align*}
p_{(\nu,l-1)\to(\upsilon,l-1)}^{(k)}=w_{\alpha_j,i},
\qquad\text{and}\qquad
p_{(\nu,l-1)\to(\nu,l-1)}^{(k-1)}=1,
\end{align*}
where $\alpha_j= l -1+\sum_{r=1}^j\upsilon_r$. Therefore,
\begin{align*}
m_{\mathfrak{t}^{(\lambda,l)}} p_\mathfrak{t}\sigma_{k}
&=m_{\mathfrak{t}^{(\lambda,l)}} w_{a_j,i}w_{ l ,i-1}e_{k-2}w_{i-1,a_{j'}}\bigg(\sum_{r=0}^{\mu_{j'}}w_{a_{j'},a_{j'}-r}\bigg)\sigma_{k}\\
&=m_{\mathfrak{t}^{(\lambda,l)}} w_{a_j,i}w_{ l ,i-1}\sigma_{k-1}e_{k-1}w_{i,a_{j'}}\sum_{r=0}^{\mu_{j'}}w_{a_{j'},a_{j'}-r}.
\end{align*}
Now, $m_{\mathfrak{t}^{(\lambda,l)}} w_{a_j,i}w_{ l ,i-1}=m_{\mathfrak{t}^{(\lambda,l)}} p_\mathfrak{u}$, where $\mathfrak{u}\in\hat{A}_{k+1}^{(\lambda,l)}$ is determined by the condition that $\mathfrak{u}\downarrow_{k-3}$ is maximal in $\hat{A}_{k-3}^{(\mu,l-1)}$ and 
\begin{align*}
\big(\mathfrak{u}^{(k-3)},\mathfrak{u}^{(k-2)},\mathfrak{u}^{(k-1)},\mathfrak{u}^{(k)},\mathfrak{u}^{(k+1)}\big)=((\mu,l-1),(\mu,l),(\mu,l),(\lambda,l),(\lambda,l)).
\end{align*}
By Proposition~\ref{e-v-b}, we have $m_{\mathfrak{t}^{(\lambda,l)}} p_\mathfrak{u}\sigma_{k}= m_{\mathfrak{t}^{(\lambda,l)}} p_\mathfrak{u}$ 
and
\begin{align}\label{twocases}
m_{\mathfrak{t}^{(\lambda,l)}} w_{a_j,i}w_{ l ,i-1}\sigma_{k-1}e_kw_{i,a_{j'}}\sum_{r=0}^{\mu_{j'}}w_{a_{j'},a_{j'}-r}
= m_{\mathfrak{t}^{(\lambda,l)}} w_{a_j,i}w_{ l ,i-1}s_{i-1}e_kw_{i,a_{j'}}\sum_{r=0}^{\mu_{j'}}w_{a_{j'},a_{j'}-r}.
\end{align}
There are two sub-cases to consider

{\textsc{Case~1.1.}} Assume that $j<j'$. Using the braid relation, the right hand side of the relation~\eqref{twocases} becomes
\begin{align*}
m_{\mathfrak{t}^{(\lambda,l)}}  w_{a_j,i}w_{l,i}e_k w_{i,a_{j'}}\sum_{r=0}^{\mu_{j'}}w_{a_{j'},a_{j'}-r}&=
m_{\mathfrak{t}^{(\lambda,l)}} w_{l,i}e_k w_{a_j-1,i-1}w_{i,a_{j'}}\sum_{r=0}^{\mu_{j'}}w_{a_{j'},a_{j'}-r}\\
&=m_{\mathfrak{t}^{(\lambda,l)}} w_{l,i} e_kw_{i,a_{j'}+1}w_{a_j-1,i} \sum_{r=0}^{\mu_{j'}}w_{a_{j'},a_{j'}-r}\\
&=m_{\mathfrak{t}^{(\lambda,l)}}  w_{l ,i} e_k w_{i,a_{j'}+1} \bigg(\sum_{r=0}^{\mu_{j'}}w_{a_{j'}+1,a_{j'}-r+1}\bigg)w_{a_j-1,i}\\
&=m_{\mathfrak{t}^{(\lambda,l)}}  w_{ l ,i} e_k w_{i,\alpha_{j'}}\bigg(\sum_{r=0}^{\lambda_{j'}}w_{\alpha_{j'},\alpha_{j'}-r}\bigg)w_{\alpha_j,i}\\
&=m_{\mathfrak{t}^{(\lambda,l)}} p_{\mathfrak{s}},
\end{align*}
since, in this case, $a_{j'}=\alpha_{j'}-1$ and $w_{a_j-1,i-1}w_{i,a_{j'}}=w_{i,a_{j'}+1}w_{a_j-1,i}$. 

{\textsc{Case~1.2.}} Assume that $j'<j$. In this case $a_j=\alpha_j$ and $a_{j'}=\alpha_{j'}$, and 
\begin{align*}
m_{\mathfrak{t}^{(\lambda,l)}}   w_{a_j,i}w_{l,i}e_k w_{i,a_{j'}}\sum_{r=0}^{\mu_{j'}}w_{a_{j'},a_{j'}-r}
&=m_{\mathfrak{t}^{(\lambda,l)}} w_{l,k}e_k w_{i,a_{j'}}w_{a_j,i}\sum_{r=0}^{\mu_{j'}}w_{a_{j'},a_{j'}-r}\\
&=m_{\mathfrak{t}^{(\lambda,l)}}w_{l,i}e_k w_{i,a_{j'}}\bigg(\sum_{r=0}^{\lambda_j}w_{a_{j'},a_{j'}-r}\bigg)w_{a_j,i}\\
&=m_{\mathfrak{t}^{(\lambda,l)}}   p_\mathfrak{s}.
\end{align*}
This completes the proof of the relation~\eqref{b-r-2}. 

{\textsc{Case 2.}} Assume that $\lambda^{(k-2)}\supsetneq \lambda^{(k-1)}=\lambda^{(k)}\supsetneq \lambda^{(k+1)}$. We regard $\hat{A}_{k+1}^{(\lambda,l)}$ as a subset of $\hat{A}_{k+2}^{(\lambda,l+1)}$ and embed $\Delta_{k+1,\mathbb{F}}^{(\lambda,l)}$  in $\Delta_{k+2,\mathbb{F}}^{(\lambda,l+1)}$. Let $\mathfrak{t}=((\lambda^{(0)},l_0),\ldots,(\lambda^{(k+2)},l_{k+2}))\in\hat{A}_{k+2}^{(\lambda,l+1)}$, where $\lambda^{(k-2)}\supsetneq \lambda^{(k-1)}=\lambda^{(k)}\supsetneq \lambda^{(k+1)}=\lambda^{(k+2)}$. To simplify notation, write 
\begin{align*}
(\lambda^{(k-2)},\lambda^{(k-1)},\lambda^{(k)},\lambda^{(k+1)},\lambda^{(k+2)})=
(\mu,\nu,\nu,\lambda,\lambda),
\end{align*}
where $\nu=\lambda\cup\lbrace(j,\nu_j)\rbrace$ and $\mu=\nu\cup \lbrace (j',\mu_{j'})$. We assume that  $(j,\nu_j)$ and $(j',\mu_{j'})$ are neither in the same row nor in the same column. Then observe that  $f_\mathfrak{t}\sigma_{k+1}=f_\mathfrak{t}$, by Theorem~\ref{sigmaodd}\eqref{sigmaodd.4}. Since $\sigma_{k+1}\sigma_{k}=\sigma_{2i+1}\sigma_{2i}=s_i$, we have the relation
\begin{align}\label{trans}
\sigma_{k}(\mathfrak{u},\mathfrak{v})=s_i(\mathfrak{u},\mathfrak{v}), \qquad\text{for all $\mathfrak{u,v}\in\hat{A}_{k+2}^{(\lambda, l+1 )}$.}
\end{align}
We therefore assume that $\mathfrak{s}=\mathfrak{t}\sigma_k$, where $\mathfrak{t}\succ\mathfrak{s}$, and show that 
\begin{align}\label{transposition}
s_i(\mathfrak{s,t})=1. 
\end{align}
As above, it is sufficient to demonstrate that 
\begin{align}\label{switch}
m_{\mathfrak{t}^{(\lambda,l+1)}}p_\mathfrak{t}s_i= m_{\mathfrak{t}^{(\lambda,l+1)}}p_\mathfrak{s}
\end{align}
and
\begin{align}\label{dominates}
s_i(\mathfrak{s,u})=0\quad\text{whenever $\mathfrak{u}\in\hat{A}^{(\lambda,l+1)}_{k+2}$ and $\mathfrak{u}\succ\mathfrak{t}$.}
\end{align}
If $\mathfrak{u}\succ\mathfrak{t}$ and $s_i(\mathfrak{s,u})\ne0$, then the relations~\eqref{trans} implies that $\mathfrak{u}\stackrel{k}{\approx}\mathfrak{t}$. Hence, by Lemma~\ref{involutions}, we have $\mathfrak{u}\in\lbrace \mathfrak{s,t}\rbrace$, in contradiction to the assumption $\mathfrak{u}\succ\mathfrak{t}\succ\mathfrak{s}$.  This verifies the condition~\eqref{dominates}. 

We now verify the relation~\eqref{switch}. Let $\mathfrak{s}=((\rho^{(0)},r_0),\ldots,(\rho^{(k+2)},r_{k+2}))\in \hat{A}_{k+2}^{(\lambda,l+1)}$, where $\mathfrak{s}=\mathfrak{t}\sigma_k$ and $\mathfrak{t}\succ\mathfrak{s}$. To simplify notation, write
\begin{align*}
(\rho^{(k-2)},\rho^{(k-1)},\rho^{(k)},\rho^{(k+1)},\rho^{(k+2)})=
(\mu,\upsilon,\upsilon,\lambda,\lambda),
\end{align*}
where $\nu=\lambda\cup{\lbrace}(j,\nu_j){\rbrace}$ and $\mu=\nu\cup{\lbrace}(j',\upsilon_{j'}){\rbrace}=\upsilon\cup{\lbrace}(j,\nu_j){\rbrace}$. There is no loss of generality in assuming that $\mathfrak{t}\downarrow_{k-2}$ is maximal in $\hat{A}_{k-2}^{(\mu,l-1)}$. Then the branching coefficients for $\mathfrak{t}$ are 
\begin{align*}
p_{(\lambda,l)\to(\lambda,l+1)}^{(k+2)}
=w_{l,i+1}\qquad\text{and}\qquad 
p_{(\nu,l-1)\to(\lambda,l)}^{(k+1)}=
w_{l-1,i}e_{k}w_{i,a_j}\sum_{r=0}^{\lambda_j}w_{a_j,a_j-r},
\end{align*}
where $a_j= l -2+\sum_{r=1}^j\nu_r$, and 
\begin{align*}
p_{(\nu,l-2)\to(\nu,l-1)}^{(k)}=w_{l-2,i} \qquad\text{and}\qquad 
p_{(\mu,l-3)\to(\nu,l-2)}^{(k-1)}=w_{l-3,i-1}e_{k-2}w_{i-1,a_{j'}}\sum_{r=0}^{\nu_{j'}}
w_{a_{j'},a_{j'}-r},
\end{align*}
where $a_{j'}= l-4+\sum_{r=1}^{j'}\mu_r$. The branching coefficients for $\mathfrak{s}$ are 
\begin{align*}
p_{(\lambda,l)\to(\lambda,l+1)}^{(k+2)}
=w_{l,i+1}\qquad\text{and}\qquad
p_{(\upsilon,l-1)\to(\lambda,l)}^{(k+1)}
=
w_{l-1,i}e_kw_{i,\alpha_{j'}}
\sum_{r=0}^{\lambda_{j'}}w_{\alpha_{j'},\alpha_{j'}-r},
\end{align*}
where $\alpha_{j'}= l -2+\sum_{r=1}^{j'}\upsilon_r=a_{j'}+1$, and 
\begin{align*}
p_{(\upsilon,l-2)\to(\upsilon,l-1)}^{(k)}=w_{l-2,i} \qquad\text{and}\qquad
p_{(\mu,l-3)\to(\upsilon,l-2)}^{(k-1)}= w_{ l -3,i-1}e_{k-2}w_{i-1,\alpha_{j}}
\sum_{r=0}^{\upsilon_{j}}w_{\alpha_{j},\alpha_{j}-r},
\end{align*}
where $\alpha_{j}= l -4+\sum_{r=1}^{j}\mu_r=a_{j}-2$. Using the braid relation, we obtain
\begin{equation*}
\begin{split}
m_{\mathfrak{t}^{(\lambda,l+1)}}p_\mathfrak{t}s_{k}
&=m_{\mathfrak{t}^{(\lambda,l+1)}} w_{l,i+1}w_{l-1,i}e_{k}w_{i,a_j}w_{l-2,i}w_{l-3,i-1}e_{k-2}w_{i-1,a_{j'}}s_i\\
&\qquad\times 
\bigg(\sum_{r=0}^{\upsilon_j}w_{\alpha_j,\alpha_j-r}\bigg)
\bigg(\sum_{r=0}^{\nu_{j'}}w_{a_{j'},a_{j'}-r}\bigg)
\end{split}
\end{equation*}
and
\begin{equation*}
\begin{split}
m_{\mathfrak{t}^{(\lambda,l+1)}}p_\mathfrak{s}
&=m_{\mathfrak{t}^{(\lambda,l+1)}} w_{l,i+1}w_{l -1,i}e_{k}w_{i,\alpha_{j'}}
w_{l-2,i}w_{l-3,i-1}e_{k-2}w_{i-1,\alpha_j}\\
&\qquad\times 
\bigg(\sum_{r=0}^{\nu_{j'}}w_{a_{j'},a_{j'}-r}\bigg)
\bigg(\sum_{r=0}^{\upsilon_j}w_{\alpha_j,\alpha_j-r}\bigg). 
\end{split}
\end{equation*}
The relation~\eqref{switch} will therefore follow from 
\begin{align}\label{c-l-3}
e_kw_{i,a_j}w_{l-2,i}w_{l-3,i-1}e_{k-2}w_{i-1,a_{j'}}s_i
=e_{k}w_{i,\alpha_{j'}}w_{l-2,i}w_{l-3,i-1}e_{k-2}w_{i-1,\alpha_j}.
\end{align}
Considering the left hand side of~\eqref{c-l-3}, the braid relation gives
\begin{align*}
&e_{k}w_{i,a_j}w_{l-2,i}w_{l-3,i-1}e_{k-2}w_{i-1,a_{j'}}s_i
=e_{k}w_{l-2,i}w_{l-3,i-1}e_{k-2}w_{i-2,a_j-2}w_{i-1,a_{j'}}s_i\\
&\qquad\qquad=w_{l-2,i-1}w_{l-3,i-2}e_{k}s_{i-1}s_{i-2}e_{k-2}w_{i-2,a_j-2}w_{i-1,a_{j'}}s_i\\
&\qquad\qquad=w_{l-2,i-1}w_{l-3,i-2}e_{k-4}e_{k}s_{i-1}s_{i-2}s_iw_{i-2,a_j-2}w_{i-1,a_{j'}}\\
&\qquad\qquad=w_{l-2,i-1}w_{l-3,i-2}e_{k-4}e_{k}s_{i-1}w_{i-2,a_j-2}w_{i-1,a_{j'}},
\end{align*}
since $e_{2j-2}e_{2j+2}=e_{2j-2}e_{2j+2}s_{j}s_{j-1}s_{j+1}s_{j}$. Considering the right hand side of~\eqref{c-l-3}, the braid relation gives
\begin{align*}
&e_{k}w_{i,\alpha_{j'}}w_{l-2,i}w_{l-3,i-1}e_{k-2}w_{i-1,\alpha_j}
=e_{k}w_{l-2,i}w_{l-3,i-1}e_{k-2}w_{i-2,\alpha_{j'}-2}w_{i-1,\alpha_j}\\
&\qquad\qquad=w_{l-2,i-1}w_{l-3,i-2}e_{k}s_{i-1}s_{i-2}e_{k-2}w_{i-1,\alpha_j}w_{i-1,\alpha_{j'}-1}\\
&\qquad\qquad=w_{l-2,i-1}w_{l-3,i-2}e_{k-4}e_{k}s_{i-1}w_{i-2,\alpha_j}w_{i-1,\alpha_{j'}-1},
\end{align*}
which completes the proof of~\eqref{c-l-3} and~\eqref{switch}. Thus $\sigma_{k}(\mathfrak{s},\mathfrak{t})=1$ whenever $\mathfrak{s}=\mathfrak{t}\sigma_{k}$ and $\mathfrak{t}\succ\mathfrak{s}$.  

If $\mathfrak{s}=\mathfrak{t}\sigma_k$ and $\mathfrak{s}\succ\mathfrak{t}$, the structure constant $\sigma_k(\mathfrak{s,t})$ is derived using the relation $\sigma_k^2=1$ and the observation that $c_\mathfrak{s}(k-1)=c_\mathfrak{t}(k+1)$ and $c_\mathfrak{s}(k+1)=c_\mathfrak{t}(k-1)$. 
\end{proof}
\begin{proof}[Proof of Theorem~\ref{sigmaodd}] 
\eqref{sigmaodd.1} Assume that  $\lambda^{(k+1)}=\lambda^{(k-1)}$ and $\lambda^{(k)}=\lambda^{(k-2)}$. Proposition~\ref{t-0}\eqref{t-0-1} and Proposition~\ref{nonzero} give
\begin{equation}\label{oddoffdiag.1}
\begin{split}
(c_\mathfrak{s}(k+1)-c_\mathfrak{t}(k-1))\sigma_k(\mathfrak{s,t})&= \delta_\mathfrak{st}-c_\mathfrak{t}(k-1)e_{k-1}(\mathfrak{s,t})-e_{k-1}(\mathfrak{t,t})e_k(\mathfrak{s,t})\\ &\qquad +c_\mathfrak{t}(k-1)e_{k-1}(\mathfrak{t,t}) e_k(\mathfrak{t,t})e_{k-1}(\mathfrak{s,t}).
\end{split}
\end{equation}
Observe that $e_{k-1}(\mathfrak{t,t}) e_k(\mathfrak{t,t})=1$ and $c_\mathfrak{s}(k+1)=c_\mathfrak{t}(k-1)$ if and only if $\mathfrak{s}=\mathfrak{t}$. Hence the expression~\eqref{sigmaodd.1.1} for the off-diagonal structure constant $\sigma_k(\mathfrak{s,t})$ follows from~\eqref{oddoffdiag.1}. Next, we use the relation $L_{k-1}e_{k-1}=\sigma_ke_ke_{k-1}$ in 
\begin{align*}
c_\mathfrak{t}(k-1)f_\mathfrak{t}e_{k-1}=f_\mathfrak{t}\sigma_ke_ke_{k-1}= \sum_{\mathfrak{u}\stackrel{k}{\sim}\mathfrak{s}\stackrel{k}{\sim}\mathfrak{t}}e_k(\mathfrak{u,s})\sigma_{k}(\mathfrak{s,t})f_\mathfrak{u}e_{k-1}=e_k(\mathfrak{t,t})\sigma_k(\mathfrak{t,t})f_\mathfrak{t}e_{k-1}
\end{align*}
to obtain the expression~\eqref{sigmaodd.1.2} for $\sigma_k(\mathfrak{t,t})$. 

\eqref{sigmaodd.2} Assume that  $\lambda^{(k+1)}\ne \lambda^{(k-1)}$ and $\lambda^{(k)}=\lambda^{(k-2)}$. Together with Proposition~\ref{nonzero}, the relation $L_{k-1}e_{k-1}=\sigma_ke_ke_{k-1}$ gives
\begin{align*}
c_\mathfrak{t}(k-1)f_\mathfrak{t}e_{k-1}=c_\mathfrak{t}(k-1)e_{k-1}(\mathfrak{v,t})e_{k-1}(\mathfrak{v,v})^{-1}f_\mathfrak{v}e_{k-1}=\sigma_k(\mathfrak{v,t})e_k(\mathfrak{v,v})f_\mathfrak{v}e_{k-1}. 
\end{align*}
Since $e_{k-1}(\mathfrak{v,v})e_k(\mathfrak{v,v})=1$, the expression~\eqref{sigmaodd.2.1} for $\sigma_k(\mathfrak{v,t})$ follows.  Next, Proposition~\ref{t-0}\eqref{t-0-1} and Proposition~\ref{nonzero} give
\begin{equation}\label{oddsigmas}
\begin{split}
(c_\mathfrak{s}(k+1)-c_\mathfrak{t}(k-1))\sigma_k(\mathfrak{s,t})&=\delta_\mathfrak{st}-c_\mathfrak{t}(k-1)e_{k-1}(\mathfrak{s,t})-e_{k-1}(\mathfrak{v,t}) e_k(\mathfrak{s,v})\\
&\qquad+e_{k-1}(\mathfrak{v,t})c_\mathfrak{v}(k-1)e_k(\mathfrak{v,v})e_{k-1}(\mathfrak{s,v}).
\end{split}
\end{equation}
Making substitutions for the relations 
\begin{align*}
e_{k-1}(\mathfrak{s,v}) e_{k-1}(\mathfrak{v,t})= e_{k-1}(\mathfrak{v,v})e_{k-1}(\mathfrak{s,t})\qquad\text{and}\qquad e_{k-1}(\mathfrak{v,v}) e_k(\mathfrak{v,v})=1
\end{align*}
in the last summand on the right hand side of the expression~\eqref{oddsigmas} gives 
\begin{equation}
\begin{split}\notag
(c_\mathfrak{s}(k+1)-c_\mathfrak{t}(k-1))\sigma_k(\mathfrak{s,t})&=\delta_\mathfrak{st}+(c_\mathfrak{v}(k-1)-c_\mathfrak{t}(k-1))e_{k-1}(\mathfrak{s,t})\\
&\qquad-e_{k-1}(\mathfrak{v,t}) e_k(\mathfrak{s,v}).
\end{split}
\end{equation}
Using the fact $\big\lbrace \mathfrak{s} \in\hat{A}_{k+1}^{(\lambda,l)}\mid  \mathfrak{s}\stackrel{k}{\approx}\mathfrak{t},\text{ and }c_\mathfrak{s}(k+1)=c_\mathfrak{t}(k-1)\big\rbrace=\big\lbrace \mathfrak{t}\big\rbrace$, which is readily verified from the assumptions on $\mathfrak{t}$, we obtain the expression~\eqref{sigmaodd.2.2} for $\sigma_k(\mathfrak{s,t})$.

\eqref{sigmaodd.3} Assume that  $\lambda^{(k+1)}=\lambda^{(k-1)}$ and $\lambda^{(k)}\ne\lambda^{(k-2)}$. Proposition~\ref{nonzero} implies that $f_\mathfrak{t}e_{k-1}=0$. Therefore, Proposition~\ref{t-0}\eqref{t-0-2} gives $(c_\mathfrak{s}(k+1)-c_\mathfrak{t}(k-1))\sigma_k(\mathfrak{s,t})=\delta_\mathfrak{st}$, and, when $\rho^{(k)}\ne\lambda^{(k-2)}$, we obtain the expression~\eqref{sigmaodd.3.1} for $\sigma_k(\mathfrak{s,t})$. If $\rho^{(k)}=\lambda^{(k-2)}$, we use the relation $\langle f_\mathfrak{t}\sigma_k,f_\mathfrak{s}\rangle = \langle f_\mathfrak{t},f_\mathfrak{s}\sigma_k\rangle$ to obtain the expression~\eqref{sigmaodd.3.2} for $\sigma_k(\mathfrak{s,t})$. 

\eqref{sigmaodd.4} Assuming that  $\lambda^{(k+1)}\ne \lambda^{(k-1)}$ and $\lambda^{(k)}\ne\lambda^{(k-2)}$ and that $\mathfrak{t}\sigma_k$ does not exist, we argue as in the proof of Theorem~\ref{sigmaeven}\eqref{sigmaeven.4} to obtain the expression~\eqref{sigmaodd.4.1} for $\sigma_{k}(\mathfrak{s,t})$.

\eqref{sigmaodd.5} Assume that  $\lambda^{(k+1)}\ne \lambda^{(k-1)}$ and $\lambda^{(k)}\ne\lambda^{(k-2)}$ and that $\mathfrak{t}\sigma_k$ exists. Lemma~\ref{similar} and Lemma~\ref{involutions} give $
f_\mathfrak{t}\sigma_k=\sigma_{k}(\mathfrak{t,t})f_\mathfrak{t}+\sigma_k(\mathfrak{s,t})f_\mathfrak{s}$, 
where $\mathfrak{s}=\mathfrak{t}\sigma_k$, while Proposition~\ref{nonzero} and the assumptions on $\mathfrak{t}$ imply that $f_\mathfrak{t}e_{k-1}=f_\mathfrak{t}e_k=0$. Using Proposition~\ref{t-0}\eqref{t-0-1}, we obtain the relation $(c_\mathfrak{t}(k+1)-c_\mathfrak{t}(k-1))\sigma_k(\mathfrak{t,t})=1$ and the expression for $\sigma_k(\mathfrak{t,t})$ in~\eqref{sigmaodd.5.1}. 

Next, we show that if $\mathfrak{s}=\mathfrak{t}\sigma_k$ and $\mathfrak{t}\succ\mathfrak{s}$, then $\sigma_k(\mathfrak{s,t})=1$. Expressing $f_\mathfrak{t}$ in terms of the Murphy type basis  from Definition~\ref{cellmodule}, we have 
 $f_\mathfrak{t}\sigma_k=m_{\mathfrak{t}^{(\lambda,l)}}p_\mathfrak{t}\sigma_k+\sum_{\mathfrak{u}\succ\mathfrak{t}}r_\mathfrak{u}f_\mathfrak{u}\sigma_k$. 
Thus it suffices to show that 
\begin{align}\label{switch:2}
m_{\mathfrak{t}^{(\lambda,l)}}p_\mathfrak{t}\sigma_k=m_{\mathfrak{t}^{(\lambda,l)}}p_\mathfrak{s} +\sum_{\mathfrak{v}\succ\mathfrak{t}}r'_\mathfrak{v}m_\mathfrak{v}.
\end{align}
and $\sigma_k(\mathfrak{s,u})=0$ whenever $\mathfrak{u}\in\hat{A}_{k+1}^{(\lambda,l)}$ and $\mathfrak{u}\succ\mathfrak{t}$. If $\mathfrak{u}\succ\mathfrak{t}$ and $\sigma_k(\mathfrak{s,u})\ne0$, then $\mathfrak{u}\in\lbrace\mathfrak{s,t}\rbrace$ by Lemma~\ref{similar} and Lemma~\ref{involutions}. However, since $\mathfrak{u}\succ\mathfrak{t}\succ\mathfrak{s}$, we have a contradiction. We now verify~\eqref{switch:2}. For this purpose, let $k=2i+1$ and write $\mathfrak{t}=((\lambda^{(0)},l_0),\ldots,(\lambda^{(k+1)},l_{k+1}))$ and $\mathfrak{s}=((\rho^{(0)},r_0),\ldots,(\rho^{(k+1)},r_{k+1}))$. There are two cases to consider. 

{\textsc{Case~1.}} Assume that $\lambda^{(k-2)}=\lambda^{(k-1)}\supsetneq\lambda^{(k)}$ and $\lambda^{(k)}\subsetneq \lambda^{(k+1)}$. To simplify notation, let
\begin{align*}
(\lambda^{(k-2)},\lambda^{(k-1)},\lambda^{(k)},\lambda^{(k+1)})=(\mu,\mu,\nu,\lambda),
\end{align*}
where $\lambda=\nu\cup{\lbrace}(j,\lambda_j){\rbrace}$ and $\mu=\nu\cup{\lbrace}(j',\mu_{j'}){\rbrace}$. Then 
\begin{align*}
(\rho^{(k-2)},\rho^{(k-1)},\rho^{(k)},\rho^{(k+1)})=(\mu,\upsilon,\lambda,\lambda), 
\end{align*}
where $\upsilon=\mu\cup{\lbrace}(j,\upsilon_j) {\rbrace}=\lambda\cup{\lbrace}(j',\upsilon_{j'})$. There is no loss of generality in assuming that $\mathfrak{t}\downarrow_{k-2}$ is maximal in $\hat{A}_{k-2}^{(\mu,l-2)}$. We may therefore take the branching coefficients for $\mathfrak{t}$ to be
\begin{align*}
p_{(\nu,l)\to(\lambda,l)}^{(k+1)}=w_{a_j,i+1} \qquad\text{and}\qquad p_{(\mu,l-1)\to(\nu,l)}^{(k)}=w_{l,i-1}e_{k-1}w_{i-1,a_{j'}}
\sum_{r=0}^{\nu_{j'}} w_{a_{j'},a_{j'}-r}
\end{align*}
where $a_j= l +\sum_{r=1}^j\lambda_r$ and $a_{j'}= l -1+\sum_{r=1}^{j'}\mu_r$, and
\begin{align*}
p_{(\mu,l-1)\to(\nu,l)}^{(k-1)}=w_{l-1,i}. 
\end{align*}
The relation $\sigma_{k}=s_{i}\sigma_{k-1}$ and the braid relation give
\begin{align*}
m_{\mathfrak{t}^{(\lambda,l)}}p_\mathfrak{t}\sigma_{k+1}
&=m_{\mathfrak{t}^{(\lambda,l)}} w_{a_j,i+1}w_{l,i-1}e_{k-1}w_{i-1,a_{j'}}
\bigg(\sum_{r=0}^{\upsilon_{j'}} w_{a_{j'},a_{j'}-r}\bigg)w_{l-1,i+1}\sigma_{k-1}\\
&=m_{\mathfrak{t}^{(\lambda,l)}} w_{l-1,i+1}w_{a_j-1,i}w_{l-1,i-2}e_{k-3}w_{i-2,a_{j'}-1}
\bigg(\sum_{r=0}^{\upsilon_{j'}} w_{a_{j'}-1,a_{j'}-r-1}\bigg)\sigma_{k-1}\\
&=m_{\mathfrak{t}^{(\lambda,l)}} w_{l,i+1}w_{a_j-1,i}w_{l-1,i-2}e_{k-3}w_{i-2,a_{j'}-1}
\bigg(\sum_{r=0}^{\upsilon_{j'}} w_{a_{j'}-1,a_{j'}-r-1}\bigg)\sigma_{k-1}\\
&=m_{\mathfrak{t}^{(\lambda,l)}} p_\mathfrak{u}\sigma_{k-1},
\end{align*}
where 
$\mathfrak{u}\in\hat{A}_{k+1}^{(\lambda,l)}$ is determined by the condition that $\mathfrak{u}\downarrow_{k-3}$ is maximal in $\hat{A}_{k-3}^{(\mu,l-2)}$ and 
\begin{align*}
(\mathfrak{u}^{(k-3)},\mathfrak{u}^{(k-2)},\mathfrak{u}^{(k-1)},\mathfrak{u}^{(k)},\mathfrak{u}^{(k+1)})=
((\mu,l-2),(\nu,l-1),(\lambda,l-1),(\lambda,l-1),(\lambda,l)). 
\end{align*}
Observing that $\mathfrak{s}=\mathfrak{u}\sigma_{k-1}$ and $\mathfrak{u}\succ\mathfrak{s}$ and using the relation~\eqref{b-r-2} and Proposition~\ref{filtration}, we obtain 
\begin{align*}
m_{\mathfrak{t}^{(\lambda,l)}} p_\mathfrak{t}\sigma_{k}=m_{\mathfrak{t}^{(\lambda,l)}} p_\mathfrak{u}\sigma_{k-1}=m_{\mathfrak{t}^{(\lambda,l)}} p_\mathfrak{s}+\sum_{\mathfrak{v}\succ\mathfrak{t}}r'_\mathfrak{v}m_\mathfrak{v},
\end{align*}
as required.

{\textsc{Case~2.}} Assume that $\lambda^{(k-2)}\subsetneq \lambda^{(k-1)}=\lambda^{(k)}\subsetneq \lambda^{(k+1)}$ and, to simplify notation, write
\begin{align*}
(\lambda^{(k-2)} \lambda^{(k-1)},\lambda^{(k)},\lambda^{(k+1)})=(\mu,\upsilon,\upsilon,\lambda), 
\end{align*}
where $\lambda=\upsilon\cup{\lbrace}(j',\lambda_{j'}){\rbrace}$ and $\upsilon=\mu\cup\lbrace(j,\lambda_{j})\rbrace$. Then 
\begin{align*}
(\rho^{(k-2)} \rho^{(k-1)},\rho^{(k)},\rho^{(k+1)})=(\mu,\nu,\nu,\lambda), 
\end{align*}
where $\lambda=\nu\cup\lbrace(j,\lambda_{j}){\rbrace}$ and $\nu=\mu\cup\lbrace(j',\lambda_{j'})\rbrace$. There is no loss of generality in assuming that $\mathfrak{t}\downarrow_{k-2}$ is maximal in $\hat{A}_{k-1}^{(\mu,l)}$. Then $
p_\mathfrak{t}= w_{a_{j'},i+1}w_{a_{j},i}$ and $p_\mathfrak{s}=w_{a_j, i+1}w_{a_{j'}-1,i}$, where  $a_j=l+\sum_{r=1}^{j}\lambda_r$ and $a_{j'}=l+\sum_{r=1}^{j'}\lambda_r$. Using $s_{i}\sigma_{k}=\sigma_{k-1}$ and the braid relation,
\begin{align*}
p_\mathfrak{t}\sigma_k= w_{a_{j'},i+1}w_{a_{j},i}\sigma_k=w_{a_{j'},i+1}w_{a_{j},i+1}\sigma_{k-1} = w_{a_{j},i+1}w_{a_{j'-1},i}\sigma_{k-1}. 
\end{align*}
Therefore, by Lemma~\ref{e-v-b} and Proposition~\ref{filtration},
\[
m_{\mathfrak{t}^{(\lambda,l)}}p_\mathfrak{t}\sigma_{k}=m_{\mathfrak{t}^{(\lambda,l)}}p_\mathfrak{s}\sigma_{k-1}=m_{\mathfrak{t}^{(\lambda,l)}}p_\mathfrak{s}+\sum_{\mathfrak{v}\succ\mathfrak{t}}r'_\mathfrak{v}m_\mathfrak{v}.
\]
This verifies~\eqref{switch:2} and the claim that $\sigma_{k}(\mathfrak{s},\mathfrak{t})=1$ if $\mathfrak{s}=\mathfrak{t}\sigma_{k}$ and $\mathfrak{t}\succ\mathfrak{s}$.  

If $\mathfrak{s}=\mathfrak{t}\sigma_k$ and $\mathfrak{s}\succ\mathfrak{t}$, the structure constant $\sigma_k(\mathfrak{s,t})$ is derived using the relation $\sigma_k^2=1$ and the observation that $c_\mathfrak{s}(k-1)=c_\mathfrak{t}(k+1)$ and $c_\mathfrak{s}(k+1)=c_\mathfrak{t}(k-1)$. 
\end{proof}

\section{Central Element Recursions}\label{c-e-r-s}
In this section we obtain partition algebra analogues of the central element recursions obtained by Nazarov~\cite{MR1398116} for the Brauer algebras and explain the relationship between these central element recursions and the seminormal representations of the partition algebras. 

Renormalise the family $\lbrace L_i\mid i\ge 0\rbrace$ by defining 
\begin{align*}
x_{i}=-\textstyle{\frac{z}{2}}+L_{i},\qquad\text{for $i=0,1,\ldots.$}
\end{align*}
For $i=1,2,\ldots,$ we have the relations
\begin{align}
&x_{2i+1}=-x_{2i}e_{2i}-e_{2i}x_{2i}-({\textstyle\frac{z}{2}}+x_{2i-1})e_{2i}+s_ix_{2i-1}s_i+\sigma_{2i},\label{x-1-def}\\
&x_{2i+2}=-\textstyle{\frac{z}{2}}e_{2i}-s_ix_{2i}e_{2i}-e_{2i}x_{2i}s_i+e_{2i}x_{2i}e_{2i+1}e_{2i}+s_ix_{2i}s_i+\sigma_{2i+1}.\label{x-2-def}
\end{align}
For $i=0,1,\ldots,$  define a sequence of central elements $\{x^{(j)}_{i+1} \mid j=0,1,\ldots\} \subseteq A_i$ by 
\begin{align}\label{r-e-d}
x_{i+1}^{(j)}e_{i+1}=e_{i+1}(x_{i+1})^je_{i+1}. 
\end{align}
Let $u$ denote a formal variable. For $i=0,1,\dots,$  define the central elements $W_{2i+1}(u)\in A_{2i}(z)[[u^{-1}]]$, and $W_{2i+2}(u)\in A_{2i+1}(z)[[u^{-1}]]$, by 
\begin{align*}
W_{2i+1}(u)=u^{-1}\sum_{j\ge0} x_{2i+1}^{(j)}u^{-j}\qquad\text{and}\qquad W_{2i+2}(u)=\sum_{j\ge0} x_{2i+2}^{(j)}u^{-j}. 
\end{align*}
We will work with the formal expressions 
\begin{align*}
W_{2i+1}(u)e_{2i+1}=e_{2i+1}\frac{1}{u-x_{2i+1}}e_{2i+1}\qquad\text{and}\qquad{W_{2i+2}(u)}e_{2i+2}=e_{2i+2}\frac{u}{u-x_{2i+2}}e_{2i+2}.
\end{align*}
The next statement is an analogue to~\cite[Proposition~4.2]{MR1398116}. \begin{proposition}\label{n-1}
If $i=1,2,\ldots,$ then 
\begin{align}
\frac{W_{2i+1}(u)+(\frac{z}{2}-u-1)}{W_{2i-1}(u)+(\frac{z}{2}-u-1)}
&=\frac{(u+x_{2i})^2-1}{(u+x_{2i})^2}\frac{(u-x_{2i-1})^2}{(u-x_{2i-1})^2-1},\label{n-1-a} 
\end{align}
and
\begin{align}
\frac{W_{2i+2}(u)}{W_{2i}(u)}
&=\frac{(u+x_{2i+1})^2-1}{(u+x_{2i+1})^2}\frac{(u-x_{2i})^2}{(u-x_{2i})^2-1}.\label{n-1-b}
\end{align}
\end{proposition}
\begin{proof}
We prove~\eqref{n-1-a}. From the relation~\eqref{x-1-def}, 
\begin{align*}
s_i(u-x_{2i+1})&={\textstyle{\frac{z}{2}}}e_{2i}+s_ix_{2i}e_{2i}+e_{2i}x_{2i}+x_{2i-1}e_{2i}+(u-x_{2i-1})s_i-\sigma_{2i+1}\\
&=\sigma_{2i+1}e_{2i-1}e_{2i}+e_{2i}x_{2i}+x_{2i-1}e_{2i}+(u-x_{2i-1})s_i-\sigma_{2i+1},
\end{align*}
and 
\begin{align*}
\frac{1}{u-x_{2i-1}}s_i
&=\frac{1}{u-x_{2i-1}}\sigma_{2i+1}e_{2i-1}e_{2i}\frac{1}{u-x_{2i+1}}
+\frac{1}{u-x_{2i-1}}e_{2i}x_{2i}\frac{1}{u-x_{2i+1}}\\
&\quad-e_{2i}\frac{1}{u-x_{2i+1}}
+\frac{u}{u-x_{2i-1}}e_{2i}\frac{1}{u-x_{2i+1}}+s_i\frac{1}{u-x_{2i+1}}\notag\\
&\quad-\sigma_{2i+1}\frac{1}{(u-x_{2i+1})(u-x_{2i-1})}.\notag
\end{align*}
Using the fact that $e_{2i}(u-x_{2i+1})=e_{2i}(u+x_{2i})$, the previous expression gives 
\begin{align*}
\frac{1}{u-x_{2i-1}}s_i&=\frac{1}{u-x_{2i-1}}\sigma_{2i+1}e_{2i-1}e_{2i}\frac{1}{u-x_{2i+1}}+\frac{1}{u-x_{2i-1}}e_{2i} - e_{2i}\frac{1}{u+x_{2i}}\\
&\quad+s_i\frac{1}{u-x_{2i+1}}-\sigma_{2i+1}\frac{1}{(u-x_{2i+1})(u-x_{2i-1})}.\notag
\end{align*}
Since $\sigma_{2i+1}$ commutes with $x_{2i-1}$, multiplying both sides of the last expression by $\sigma_{2i+1}$ on the left, we obtain 
\begin{equation}\label{s-1-a}
\begin{split}
\frac{1}{u-x_{2i-1}}\sigma_{2i}&=\frac{1}{u-x_{2i-1}}e_{2i-1}e_{2i}\frac{1}{u-x_{2i+1}}+\frac{1}{u-x_{2i-1}}e_{2i} - e_{2i}\frac{1}{u+x_{2i}}\\
&\quad+\sigma_{2i}\frac{1}{u-x_{2i+1}}-\frac{1}{(u-x_{2i+1})(u-x_{2i-1})}
\end{split}
\end{equation}
Multiplying each side of~\eqref{s-1-a} by $e_{2i+1}$ on the right and on the left,
\begin{equation}\label{s-1-b}
\begin{split}
e_{2i+1}\sigma_{2i}\frac{1}{u-x_{2i+1}}e_{2i+1}
&=\frac{1}{u-x_{2i-1}}e_{2i+1}\sigma_{2i}e_{2i+1}
-\frac{1}{u-x_{2i-1}}e_{2i-1}\frac{1}{u+x_{2i}}e_{2i+1}\\
&\quad+\frac{1}{u+x_{2i}}e_{2i+1}-\frac{1}{u-x_{2i-1}}e_{2i+1}
+\frac{1}{u-x_{2i-1}}{W_{2i+1}(u)}e_{2i+1}.
\end{split}
\end{equation}
Using Proposition~\ref{a-c-r}\eqref{a-c-r-15}, we obtain the relation $e_{2i+1}\sigma_{2i}e_{2i+1}=(\frac{z}{2}-x_{2i-1})e_{2i+1}$ which, substituted into~\eqref{s-1-b}, gives 
\begin{multline}
e_{2i+1}\sigma_{2i}\frac{1}{u-x_{2i+1}}e_{2i+1}
=({\textstyle\frac{z}{2}}-u-1)\frac{1}{u-x_{2i-1}}e_{2i+1}+e_{2i+1}-\frac{1}{u-x_{2i-1}}e_{2i+1}\\
-\frac{1}{u-x_{2i-1}}e_{2i-1}\frac{1}{u+x_{2i}}e_{2i+1}
+\frac{1}{u+x_{2i}}e_{2i+1}
+\frac{1}{u-x_{2i-1}}{W_{2i+1}(u)}e_{2i+1}.\label{s-1-c}
\end{multline}
Multiplying each side of~\eqref{s-1-a} by $e_{2i-1}$ on the left, and applying the algebra anti-involution $*$ to the result,
\begin{align*}
\sigma_{2i}\frac{1}{u-x_{2i-1}}e_{2i-1}
&=\frac{1}{u+x_{2i}}e_{2i}e_{2i-1}W_{2i-1}(u)
+\frac{1}{u-x_{2i-1}}e_{2i}e_{2i-1}-\frac{1}{u+x_{2i}}e_{2i}e_{2i-1}\\
&\quad+\frac{1}{u-x_{2i+1}}\sigma_{2i}e_{2i-1}
-\frac{1}{(u+x_{2i})(u-x_{2i+1})}e_{2i-1}.
\end{align*}
Making the substitution $\sigma_{2i}e_{2i-1}=(\frac{z}{2}+x_{2i})e_{2i}e_{2i-1}$ in the last expression gives
\begin{align*}
\sigma_{2i}\frac{1}{u-x_{2i-1}}e_{2i-1}
&=\frac{1}{u+x_{2i}}e_{2i}e_{2i-1}W_{2i-1}(u)
+\frac{1}{u-x_{2i-1}}e_{2i}e_{2i-1}+e_{2i}e_{2i-1}\\
&\quad+({\textstyle\frac{z}{2}}-u-1)\frac{1}{u+x_{2i}}e_{2i}e_{2i-1}
-\frac{1}{(u+x_{2i})(u-x_{2i+1})}e_{2i-1}.\notag
\end{align*}
Therefore, 
\begin{equation}
\begin{split}\label{s-1-d}
&e_{2i+1}\sigma_{2i}\frac{1}{u-x_{2i-1}}e_{2i-1}e_{2i}\frac{1}{u+x_{2i}}e_{2i+1}=\frac{1}{(u+x_{2i})^2}W_{2i-1}(u)e_{2i+1}
+\frac{1}{u+x_{2i}}e_{2i+1}\\
&\qquad+({\textstyle\frac{z}{2}}-u-1)\frac{1}{(u+x_{2i})^2}e_{2i+1}
+\frac{1}{(u-x_{2i-1})(u+x_{2i})}e_{2i+1}\\
&\qquad\qquad-\frac{1}{u-x_{2i-1}}e_{2i-1}\frac{1}{(u+x_{2i})^2}e_{2i+1}.
\end{split}
\end{equation}
Now, 
\begin{equation}\label{s-1-e}
\begin{split}
e_{2i+1}\sigma_{2i}\frac{1}{u-x_{2i-1}}\sigma_{2i}e_{2i+1}
&=e_{2i+1}\sigma_{2i}\frac{1}{u-x_{2i-1}}e_{2i-1}e_{2i}\frac{1}{u-x_{2i+1}}e_{2i+1}\\
&\quad +\frac{1}{u-x_{2i-1}}e_{2i+1}
-\frac{1}{u+x_{2i}}e_{2i+1}+W_{2i+1}(u)e_{2i+1}\\
&\quad-e_{2i+1}\sigma_{2i}\frac{1}{u-x_{2i+1}}e_{2i+1}\frac{1}{u-x_{2i-1}}.
\end{split}
\end{equation}
Using the relation
\begin{align*}
e_{2i+1}\sigma_{2i}\frac{1}{u-x_{2i-1}}\sigma_{2i}e_{2i+1}
=s_i e_{2i-1}\sigma_{2i+1}\frac{1}{u-x_{2i-1}}\sigma_{2i+1}e_{2i-1}s_i={W_{2i-1}(u)}e_{2i+1},
\end{align*}
and substituting~\eqref{s-1-c} and~\eqref{s-1-d} into~\eqref{s-1-e}, we obtain
\begin{align*}
W_{2i-1}(u)e_{2i+1}&=\frac{1}{(u+x_{2i})^2}W_{2i-1}(u)e_{2i+1}+({\textstyle\frac{z}{2}}-u-1)\frac{1}{(u+x_{2i})^2}e_{2i+1}+W_{2i+1}(u)e_{2i+1}\\
&\quad-\frac{1}{(u-x_{2i-1})^2}e_{2i+1}-({\textstyle\frac{z}{2}}-u-1)\frac{1}{(u-x_{2i-1})^2}e_{2i+1}
\end{align*}
or 
\begin{multline*}
\frac{(u+x_{2i})^2-1}{(u+x_{2i})^2}e_{2i+1}W_{2i-1}(u)-({\textstyle\frac{z}{2}}-u-1)\frac{1}{(u+x_{2i})^2}e_{2i+1}\\
=\frac{(u-x_{2i-1})^2-1}{(u-x_{2i-1})^2}e_{2i+1}W_{2i+1}(u)-({\textstyle\frac{z}{2}}-u-1)\frac{1}{(u-x_{2i-1})^2}e_{2i+1},
\end{multline*}
from which the relation~\eqref{n-1-a} follows.

For the proof of\eqref{n-1-b}, we use the the relation~\eqref{x-2-def} to write
\begin{align}\label{t-a-0}
(u-x_{2i+2})s_i={\textstyle\frac{z}{2}}e_{2i}+s_ix_{2i}e_{2i}+e_{2i}x_{2i}-e_{2i}x_{2i}e_{2i+1}e_{2i}+s_i(u-x_{2i})-\sigma_{2i}.
\end{align}
Since $(x_{2i+2})^je_{2i}=e_{2i}(x_{2i})^je_{2i+1}e_{2i}$ for $j=0,1,\ldots,$ we obtain
\begin{align*}
\frac{1}{u-x_{2i+2}}e_{2i}=e_{2i}\frac{1}{u-x_{2i}}e_{2i+1}e_{2i}, 
\end{align*}
which together with~\eqref{t-a-0}, gives 
\begin{align*}
s_i\frac{1}{u-x_{2i}}&=\frac{z}{2}\frac{1}{u-x_{2i+2}}e_{2i}\frac{1}{u-x_{2i}}+\frac{1}{u-x_{2i+2}}s_ix_{2i}e_{2i}\frac{1}{u-x_{2i}}-e_{2i}\frac{1}{u-x_{2i}}e_{2i+1}e_{2i}\\
&\quad+e_{2i}\frac{1}{u-x_{2i}}+\frac{1}{u-x_{2i+2}}s_i-\sigma_{2i}\frac{1}{(u-x_{2i})(u-x_{2i+2})}\\
&=\frac{1}{u-x_{2i+2}}\sigma_{2i}e_{2i+1}e_{2i}\frac{1}{u-x_{2i}}-e_{2i}\frac{1}{u-x_{2i}}e_{2i+1}e_{2i}
+e_{2i}\frac{1}{u-x_{2i}}\\
&\quad+\frac{1}{u-x_{2i+2}}s_i-\sigma_{2i}\frac{1}{(u-x_{2i})(u-x_{2i+2})}. 
\end{align*}
Multiplying both sides of the above expression by $\sigma_{2i}$ on the left, 
\begin{equation}\label{t-a-a}
\begin{split}
\sigma_{2i+1}\frac{1}{u-x_{2i}}&=\frac{1}{u-x_{2i+2}}e_{2i+1}e_{2i}\frac{1}{u-x_{2i}}-e_{2i}\frac{1}{u-x_{2i}}e_{2i+1}e_{2i}
+e_{2i}\frac{1}{u-x_{2i}}\\
&\quad+\frac{1}{u-x_{2i+2}}\sigma_{2i+1}-\frac{1}{(u-x_{2i})(u-x_{2i+2})}, 
\end{split}
\end{equation}
and multiplying both sides of the expression~\eqref{t-a-a} by $e_{2i}$ on the right, and then applying the anti-involution $*$ to the result,
\begin{equation}\label{t-a-b}
\begin{split}
e_{2i}\frac{1}{u-x_{2i}}\sigma_{2i+1}&=
\frac{W_{2i}(u)}{u}e_{2i}e_{2i+1}\frac{1}{u-x_{2i+2}}
+\frac{W_{2i}(u)}{u}e_{2i}\\
&\quad-e_{2i}\frac{1}{u-x_{2i}}e_{2i+1}e_{2i}\frac{1}{u-x_{2i}}.
\end{split}
\end{equation}
Multiplying both sides of the expression~\eqref{t-a-a} by $e_{2i+2}$ on the right and on the left, and using Proposition~\ref{a-c-r}~\eqref{a-c-r-17} together with the fact that $e_{2i+1}(u+x_{2i+1})=e_{2i+1}(u-x_{2i+2})$, we obtain
\begin{equation}\label{t-a-r}
\begin{split}
e_{2i+2}\frac{1}{u-x_{2i+2}}\sigma_{2i+1}e_{2i+2}=&
e_{2i}e_{2i+2}\frac{W_{2i}(u)}{u}-\frac{1}{u+x_{2i+1}}e_{2i}e_{2i+2}\frac{1}{u-x_{2i}}\\
&+\frac{1}{u-x_{2i}}e_{2i+2}\frac{W_{2i+2}(u)}{u}.
\end{split}
\end{equation}
Multiplying~\eqref{t-a-a} on the right by $\sigma_{2i+1}$ and substituting for the term appearing on the left hand side of~\eqref{t-a-b}, 
\begin{equation}\label{t-a-q}
\begin{split}
\sigma_{2i+1}\frac{1}{u-x_{2i}}\sigma_{2i+1}&= \frac{W_{2i}(u)}{u}\frac{1}{u+x_{2i+1}}e_{2i+1}\frac{1}{u+x_{2i+1}}
+\frac{W_{2i}(u)}{u}\frac{1}{u+x_{2i+1}}e_{2i+1}e_{2i} \\
&\quad-\frac{1}{(u-x_{2i})(u+x_{2i+1})}e_{2i+1}e_{2i}\frac{1}{u-x_{2i}}
- e_{2i}\frac{1}{u-x_{2i}}e_{2i+1}e_{2i}\\
&\quad+\frac{W_{2i}(u)}{u}e_{2i}e_{2i+1}\frac{1}{u+x_{2i+1}}-\frac{1}{(u-x_{2i})(u-x_{2i+2})}\sigma_{2i+1}\\ 
&\quad-e_{2i}\frac{1}{u-x_{2i}}e_{2i+1}e_{2i}\frac{1}{u-x_{2i}} +\frac{1}{u-x_{2i+2}}+\frac{W_{2i}(u)}{u}e_{2i}.
\end{split}
\end{equation}
Since
\begin{align*}
e_{2i+2}\sigma_{2i+1}\frac{1}{u-x_{2i}}\sigma_{2i+1}e_{2i+2}&
=e_{2i+2}\sigma_{2i+1}\sigma_{2i+2}\frac{1}{u-x_{2i}}\sigma_{2i+2}\sigma_{2i+1}e_{2i+2}\\
&=\sigma_{2i+1}\sigma_{2i+2}e_{2i}\frac{1}{u-x_{2i}}e_{2i}\sigma_{2i+2}\sigma_{2i+1}\\
&=\sigma_{2i+1}\sigma_{2i+2}e_{2i}\sigma_{2i+2}\sigma_{2i+1}\frac{W_{2i}(u)}{u}\\ 
&=e_{2i+2}\frac{W_{2i}(u)}{u},
\end{align*}
multiplying both sides of~\eqref{t-a-q} by $e_{2i+2}$ on the left and on the right, and substituting for the term appearing on the left hand side of~\eqref{t-a-r} gives
\begin{align*}
e_{2i+2}\frac{W_{2i}(u)}{u}&=
e_{2i+2}\bigg(\frac{W_{2i}(u)}{u}\frac{1}{(u+x_{2i+1})^2}
+\frac{W_{2i}(u)}{u}\frac{1}{u+x_{2i+1}}e_{2i}\\
&\qquad-\frac{1}{(u-x_{2i})(u+x_{2i+1})}e_{2i}\frac{1}{u-x_{2i}}-\frac{W_{2i}(u)}{u}e_{2i}
+\frac{W_{2i}(u)}{u}e_{2i}\frac{1}{u+x_{2i+1}}\\
&\qquad +\frac{W_{2i}(u)}{u}e_{2i}-\frac{W_{2i}(u)}{u}e_{2i}\frac{1}{u-x_{2i}}+\frac{W_{2i+2}(u)}{u}
-\frac{1}{u-x_{2i}}\frac{W_{2i}(u)}{u}e_{2i}\\
&\qquad +\frac{1}{(u-x_{2i})(u+x_{2i+1})}e_{2i}\frac{1}{u-x_{2i}}
-\frac{1}{(u-x_{2i})^2}\frac{W_{2i+2}(u)}{u}
\bigg).
\end{align*}
Since $(u-x_{2i})e_{2i}=(u+x_{2i+1})e_{2i}$, we now obtain
\begin{align*}
\frac{W_{2i}(u)}{u}
-\frac{W_{2i}(u)}{u}\frac{1}{(u+x_{2i+1})^2}=
\frac{W_{2i+2}(u)}{u}
-\frac{W_{2i+2}(u)}{u}\frac{1}{(u-x_{2i})^2},
\end{align*}
and the statement~\eqref{n-1-b} follows. 
\end{proof}
As an application of the recursions~\eqref{n-1-a} and~\eqref{n-1-b}, the seminormal matrix entries of the contractions $e_{2i-1}$ and $e_{2i}$ can be computed independently of any formula for the dimensions of the irreducible representations of the symmetric group.

Determine a series $Z_{2i+1}(u)\in A_{i}(z)[[u^{-1}]]$ by the recursion~\eqref{n-1-a} and 
\begin{align}\label{n-2-a}
Z_{2i+1}(u)=W_{2i+1}(u)+(\textstyle\frac{z}{2}-u-1)
&&\text{and}&&Z_1(u)=-\frac{(u+1+\textstyle\frac{z}{2})(u-\textstyle\frac{z}{2})}{(u+\textstyle\frac{z}{2})},
\end{align}
and a series $Z_{2i+2}(u)\in A_{i+\half}(z)[[u^{-1}]]$ by the recursion~\eqref{n-1-b}, and
\begin{align}\label{n-2-b}
Z_{2i+2}(u)=\frac{W_{2i+2}(u)}{u}
&&\text{and}&&
{Z_2(u)}=\frac{(u+1-\textstyle\frac{z}{2})}{(u+\textstyle\frac{z}{2})(u-\textstyle\frac{z}{2})}.
\end{align}
For $(\lambda,l)\in\hat{A}_{k}$, let $Z_{k+1}(u,\lambda)$ denote the scalar by which $Z_{k+1}(u)$ acts on the $A_{k}(z)[[u^{-1}]]$-module $\Delta^{(\lambda,l)}_{k,\mathbb{F}}$. 
If $\mathfrak{t}\in\hat{A}_k^{(\lambda,l)}$ and $i=1,\ldots,k$, let 
\begin{align*}
x_\mathfrak{t}(i)=-{\textstyle\frac{z}{2}}+c_\mathfrak{t}(i)
\end{align*}
denote the eigenvalue by which $x_i$ acts on $f_\mathfrak{t}$. The next statement is a counterpart to~\cite[(3.6)]{MR1398116} and~\cite[Lemma~7.4]{MR1866492}. 
\begin{lemma}\label{r-e-s-i}
Assume that $(\lambda, l )\in\hat{A}_{k+1}$ and $\mathfrak{t}=((\lambda^{(0)},l_0),\ldots,(\lambda^{(k+1)},l_{k+1}))$, where $\lambda^{(k-1)}=\lambda^{(k+1)}$. Then 
\begin{align}\label{r-e-s-i-1} 
e_{k}(\mathfrak{t},\mathfrak{t})=\Res_{u=x_\mathfrak{t}(k)}Z_{k}(u,\lambda).
\end{align}
\end{lemma}
\begin{proof}
Let $(\lambda, l )\in\hat{A}_{k+1}$ and $\mathfrak{t}\in\hat{A}_{k+1}^{(\lambda,l)}$ be as in the statement of the lemma and assume that $k$ is even. Then, using Lemma~\ref{s-q},
\begin{align*}
f_\mathfrak{t}\frac{W_{k}(u)}{u}e_{k}
&=f_\mathfrak{t}e_{k}\frac{1}{u-x_{k}}e_{k} 
=\sum_{\mathfrak{s}\stackrel{k}{\sim}\mathfrak{t}}
\frac{e_{k}(\mathfrak{s},\mathfrak{t})}{u-x_\mathfrak{s}(k)}f_\mathfrak{s}e_{k}
=\sum_{\mathfrak{s}\stackrel{k}{\sim}\mathfrak{t}}
\sum_{\mathfrak{v}\stackrel{k}{\sim}\mathfrak{s}}
\frac{e_{k}(\mathfrak{v},\mathfrak{s})e_{k}(\mathfrak{s},\mathfrak{t})}{u-x_\mathfrak{s}(k)}f_\mathfrak{v}\\
&=\sum_{\mathfrak{s}\stackrel{k}{\sim}\mathfrak{t}}\,
\sum_{\mathfrak{v}\stackrel{k}{\sim}\mathfrak{s}}
\frac{e_{k}(\mathfrak{s},\mathfrak{s})e_{k}(\mathfrak{v},\mathfrak{t})}{u-x_\mathfrak{s}(k)}f_\mathfrak{v}
=\sum_{\mathfrak{s}\stackrel{k}{\sim}\mathfrak{t}}
\frac{e_{k}(\mathfrak{s},\mathfrak{s})}{u-x_\mathfrak{s}(k)}f_\mathfrak{t}e_{k}. 
\end{align*}
Since $e_{k}(\mathfrak{t},\mathfrak{t})\ne0$ the relation~\eqref{r-e-s-i-1} now follows. The proof when $k$ is odd is similar.
\end{proof}
The formulae given in the next proposition are proved by induction. 
\begin{proposition}\label{c-f-r}
(1) If $(\lambda,l)\in\hat{A}_{2k}$, then 
\begin{align}
Z_{2k+1}(u,\lambda)&=-
{\displaystyle\frac{\big(u-|\lambda|+1+\frac{z}{2}\big)}{\big(u-|\lambda|+\frac{z}{2}\big)}
\frac{\prod_{\beta\in A(\lambda)}\big(u+c(\beta)-\frac{z}{2}\big)}{\prod_{\beta\in R(\lambda)}\big(u+c(\beta)-\frac{z}{2}\big)}}.\label{c-f-r-2}
\end{align}
\item[(2)]If $(\lambda,l)\in\hat{A}_{2k+1}$, then
\begin{align}
Z_{2k+2}(u,\lambda)&=
{\displaystyle
\frac{\big(u+|\lambda|+1-\frac{z}{2}\big)}{\big(u+|\lambda|-\frac{z}{2}\big)}
\frac{\prod_{\beta\in R(\lambda)}\big(u-c(\beta)+\frac{z}{2}\big)}{\prod_{\beta\in A(\lambda)}\big(u-c(\beta)+\frac{z}{2}\big)}.}\label{c-f-r-1}
\end{align}
\end{proposition}
\section{Gram Determinants}\label{det-proof}
In this section we verify the formulae for the branching factors of the inner product on cell modules given in Proposition~\ref{branching}. We will use the following notation from~\cite[Theorem~4.2]{EG:2012}.

If $(\lambda,l)\in\hat{A}_{2i-1}$ and $(\mu,m)\in\hat{A}_{2i}$, where $(\lambda,l)\to(\mu,m)$ in $\hat{A}$, define 
\begin{align*}
u_{(\lambda,l)\to(\mu,m)}^{(2i)}=
\begin{cases}
\bar{u}^{\,(2i-2l)}_{\mu\to\lambda}e_{2i-1}^{(m)},&\text{if $\lambda\subsetneq\mu$ and $l=m$,}\medskip\\
e_{2i-1}^{(m)},&\text{if $\lambda=\mu$ and $l=m-1$.}
\end{cases}
\end{align*}
Similarly, if $(\lambda,l)\in\hat{A}_{2i}$ and $(\mu,m)\in\hat{A}_{2i+1}$, where $(\lambda,l)\to(\mu,m)$ in $\hat{A}$, define
\begin{align*}
u_{(\lambda,l)\to(\mu,m)}^{(2i+1)}=
\begin{cases}
e_{2i}^{(m)},&\text{if $\lambda=\mu$ and $l=m$, }\medskip\\
\bar{d}_{\mu\to\lambda}^{\,(2i-2l)}e_{2i}^{(m)},&\text{if $\lambda\subsetneq\mu$ and $l=m-1$.}
\end{cases}
\end{align*}
\begin{lemma}\label{observation:0}
Assume that $(\mu,m)\in\hat{A}_{k+1}$ and $\mathfrak{s},\mathfrak{u}\in\hat{A}_{k+1}^{(\mu,m)}$, where $\Shape(\mathfrak{s}\downarrow_{k})= \Shape(\mathfrak{u}\downarrow_{k})$. Let $\mathfrak{t}=\mathfrak{s}\downarrow_k$ and $\mathfrak{v}=\mathfrak{u}\downarrow_k$. Then
\begin{align*}
\frac{\langle f_\mathfrak{s},f_\mathfrak{s}\rangle}{ \langle f_\mathfrak{t},f_\mathfrak{t}\rangle}=
\frac{\langle f_\mathfrak{u},f_\mathfrak{u}\rangle}{ \langle f_\mathfrak{v},f_\mathfrak{v}\rangle}.
\end{align*}
\end{lemma}
\begin{proof}
Let $\Shape(\mathfrak{s}\downarrow_k)=(\lambda,l)$, where $(\lambda,l)\in\hat{A}_k$ and $(\lambda,l)\to(\mu,m)$ in $\hat{A}$.
Then, 
\begin{align*}
c_{(\mu,m)}^{(k+1)}d_{(\lambda,l)\to(\mu,m)}^{(k+1)} =\left(u^{(k+1)}_{(\lambda,l)\to(\mu,m)}\right)^{\hspace{-1.5pt}*}c_{(\lambda,l)}^{(k)}. 
\end{align*}
Moreover, if  $a\in A_k$, where $c_{(\lambda,l)}^{(k)}a\equiv 0\mod A_{k}^{\rhd(\lambda,l)}$, then Proposition~\ref{filtration} gives the relation 
\begin{align}\label{modstuff}
\left(u^{(k+1)}_{(\lambda,l)\to(\mu,m)}\right)^{\hspace{-1.5pt}*}c_{(\lambda,l)}^{(k)}a \equiv0\mod A_{k+1}^{\rhd(\mu,m)}. 
\end{align}
Therefore,
\begin{align*}
c_{(\mu,m)}^{(k+1)}d_{(\lambda,l)\to(\mu,m)}^{(k+1)} d_\mathfrak{t} F_\mathfrak{s} d_\mathfrak{t}^* (d_{(\lambda,l)\to(\mu,m)}^{(k+1)})^*c_{(\mu,m)}^{(k+1)} = 
(u_{(\lambda,l)\to(\mu,m)}^{(k+1)} )^* c_{(\lambda,l)}^{(k)} d_\mathfrak{t} F_\mathfrak{t} d_\mathfrak{t}^* c_{(\lambda,l)}^{(k)} \Psi_{\mathfrak{s},k}u_{(\lambda,l)\to(\mu,m)}^{(k+1)},
\end{align*}
where $\Psi_{\mathfrak{s},k}=\Psi_{\mathfrak{s},k}(L_{k+1})$ is a polynomial in $L_{k+1}$ over $\mathbb{F}$ such that $F_\mathfrak{s}=F_\mathfrak{t}\Psi_{\mathfrak{s},k}$. Since,
\begin{align*}
c_{(\lambda,l)}^{(k)} d_\mathfrak{t} F_\mathfrak{t} d_\mathfrak{t}^* c_{(\lambda,l)}^{(k)}\equiv \langle f_\mathfrak{t},f_\mathfrak{t}\rangle c_{(\lambda,l)}^{(k)}  \mod A_{k}^{\rhd(\lambda,l)},
\end{align*}
using the relation~\eqref{modstuff}, we have 
\begin{align*}
c_{(\mu,m)}^{(k+1)}d_{(\lambda,l)\to(\mu,m)}^{(k+1)} d_\mathfrak{t} F_\mathfrak{s} d_\mathfrak{t}^* (d_{(\lambda,l)\to(\mu,m)}^{(k+1)})^*c_{(\mu,m)}^{(k+1)} \equiv \langle f_\mathfrak{t},f_\mathfrak{t}\rangle  c_{(\mu,m)}^{(k+1)}d_{(\lambda,l)\to(\mu,m)}^{(k+1)} \Psi_{\mathfrak{s},k} u_{(\lambda,l)\to(\mu,m)}^{(k+1)}. 
\end{align*}
The same argument shows that
\begin{align*}
c_{(\mu,m)}^{(k+1)}d_{(\lambda,l)\to(\mu,m)}^{(k+1)} d_\mathfrak{v} F_\mathfrak{t} d_\mathfrak{v}^* (d_{(\lambda,l)\to(\mu,m)}^{(k+1)})^*c_{(\mu,m)}^{(k+1)} \equiv \langle f_\mathfrak{v},f_\mathfrak{v}\rangle  c_{(\mu,m)}^{(k+1)}d_{(\lambda,l)\to(\mu,m)}^{(k+1)} \Psi_{\mathfrak{u},k} u_{(\lambda,l)\to(\mu,m)}^{(k+1)}. 
\end{align*}
Since the polynomials $\Psi_{\mathfrak{s},k},\Psi_{\mathfrak{u},k}$ depend only on the $(k+1)^{\text{st}}$ edges in the paths $\mathfrak{s}$ and $\mathfrak{u}$ respectively, we have $\Psi_{\mathfrak{s},k}=\Psi_{\mathfrak{u},k}$, and the statement follows.  
\end{proof}

\begin{lemma}\label{observation:2}
Let $(\lambda,l)\in\hat{A}_{k+1}$ and  $\mathfrak{s}\in\hat{A}_{k+1}^{(\lambda,l)}$. If $\mathfrak{t}=\mathfrak{s}\sigma_{k}$ exists and $\mathfrak{t}\succ\mathfrak{s}$, then
\begin{align*}
\langle f_\mathfrak{s},f_\mathfrak{s}\rangle = \frac{(c_\mathfrak{s}(k+1)-c_\mathfrak{s}(k-1))^2-1}{(c_\mathfrak{s}(k+1)-c_\mathfrak{s}(k-1))^2}\langle f_\mathfrak{t},f_\mathfrak{t}\rangle. 
\end{align*}
\end{lemma}
\begin{proof}
We use the relation $\langle f_\mathfrak{s},f_\mathfrak{s}\rangle =\langle f_\mathfrak{s}\sigma_k,f_\mathfrak{s}\sigma_k\rangle$ and the formula~\eqref{sigmaeven.5.1} or~\eqref{sigmaodd.5.1}, depending on whether $k$ is odd or even. 
\end{proof}
Both our proof of the branching rule~\eqref{quotient} on inner products and our derivation of the off-diagonal structure constants for the contraction $e_k$, for $k$ even, will depend on the following combinatorial lemma. 
\begin{lemma}\label{obs:5}
Let $k$ be even and $(\lambda,l)\in\hat{A}_{k}$. Assume that $\mathfrak{s}=\mathfrak{t}^{(\lambda,l)}\in\hat{A}_{k}^{(\lambda,l)}$ and  let $\alpha=(j,\lambda_j)\in R(\lambda)$, where $\lambda_{j+1}>0$. Denote $\upsilon=\lambda\setminus\lbrace \alpha\rbrace$,  $a=l+\sum_{r=0}^j\lambda_r$ and $t=\sum_{r>j}\lambda_r$. Then the following statements hold:
\begin{enumerate}
[label=(\arabic{*}), ref=\arabic{*},leftmargin=0pt,itemindent=1.5em]
\item\label{obs:5.1} $\lambda^{(2a)}=\lambda^{(2a-1)}\cup\lbrace \alpha\rbrace$.
\item\label{obs:5.2} If $0\le r\le t,$ then $\mathfrak{s}_r=(\cdots ((\mathfrak{s}\sigma_{2a+1})\sigma_{2a+3})\cdots)\sigma_{2a+2r-1}$ exists in $\hat{A}_k^{(\lambda,l)}$.
\item\label{obs:5.3} If $\mathfrak{s}_r=((\nu^{(0)},u_0),\ldots,(\nu^{(k)},u_k))$, then $\nu^{(2a+2r)}=\nu^{(2a+2r-1)}\cup\lbrace \alpha\rbrace$ and $\nu^{(i)}=\lambda^{(i)}$ for $i=2a+2r,\ldots ,k$.
\item\label{obs:5.4} The sequence $\lbrace \mathfrak{s}_i\mid i=0,\ldots,t\rbrace $ satisfies $\mathfrak{s}_0\succ\mathfrak{s}_1\succ\cdots\succ\mathfrak{s}_{t}$. 
\item\label{obs:5.5} If $r=0,1,\ldots,t$, then $f_{\mathfrak{s}_r}\sigma_{2a+2r-2}=f_{\mathfrak{s}_r}$. 
\item\label{obs:5.6} If $\mathfrak{t}=\mathfrak{s}_{t}$,  then $\Shape(\mathfrak{t}\downarrow_{k-1})=(\upsilon,l)$ and 
\begin{align}\label{expansion}
f_\mathfrak{s} \sigma_{2a+1}\sigma_{2a+3}\cdots \sigma_{2a+2t-1}=f_\mathfrak{t}+\sum_{\mathfrak{u}\succ\mathfrak{t}}r_\mathfrak{u}f_\mathfrak{u},
\end{align}
where the sum is over $\mathfrak{u}\in\hat{A}_{k}^{(\lambda)}$ such that $\Shape(\mathfrak{u}\downarrow_{k-1})\ne (\upsilon,l)$. 
\end{enumerate}
\end{lemma}
\begin{proof}
\eqref{obs:5.1} Let $\mathfrak{s}=((\lambda^{(0)},l_0),\ldots,(\lambda^{(k)},l_k))$. Observe that $(\lambda^{(r)},l_{r})=(\emptyset,r)$ for $0\le r\le 2l+1$ and $\lambda^{(i-1)}\subsetneq\lambda^{(i)}$ exactly when $i=2l+2,2l+4,\ldots,2l+2a+2t.$ By the maximality property of $\mathfrak{s}$, we have $\lambda^{(r+1)}=\lambda^{(r)}\cup \lbrace\alpha\rbrace$ at the $a^\text{th}$ instance where $\lambda^{(r)}\subsetneq\lambda^{(r+1)}$. Thus the item  follows.

\eqref{obs:5.2} Let $\lbrace\beta_1,\ldots,\beta_{t}\rbrace\subseteq\lambda$ denote the set of nodes of $\lambda$ with row index greater that $j$, ordered so that $\lambda^{(2a+2r)}=\lambda^{(2a+2r-1)}\cup\lbrace \beta_{r}\rbrace$ for $r=1,2,\ldots,t$. The assumption that $\alpha\in R(\lambda)$ implies that if  $1\le i\le t$, then  $\alpha$ and $\beta_r$ are neither in the same row nor the same column. In particular, since $\lambda^{(2a+2)}\ominus\lambda^{(2a-1)}=\lbrace\alpha,\beta_1\rbrace$, the path $\mathfrak{s}_1=\mathfrak{s}\sigma_{2a+1}$ exists. If $\mathfrak{s}_1=((\mu^{(0)},l_0),\ldots,(\mu^{(k)},l_k))$, then $(\mu^{(i)},l_i)=(\lambda^{(i)},l_i)$ except at the positions $i=2a$ and $i=2a+1$,  where $\mu^{(2a)}=\mu^{(2a+1)}=\mu^{(2a-1)}\cup\lbrace\beta_1\rbrace$. Note that $\mu^{(2a+2)}=\mu^{(2a+1)}\cup\lbrace\alpha\rbrace$ and $\mu^{(2a+4)}=\mu^{(2a+3)}\cup\lbrace\beta_2\rbrace$, where the nodes $\alpha$ and $\beta_2$ are neither in the same row nor the same column. Thus $\mathfrak{s}_2=\mathfrak{s}_1\sigma_{2a+3}$ exists and, continuing by induction, we construct the sequence $\mathfrak{s}_0,\mathfrak{s}_1,\ldots,\mathfrak{s}_t$ in $\hat{A}_k^{(\lambda,l)}$. 

\eqref{obs:5.3} The observation follows from the construction of the path $\mathfrak{s}_r$ above. 

\eqref{obs:5.4} 
If $\mathfrak{s}_{r-1}=((\mu^{(0)},l_0),\ldots,(\mu^{(k)},l_k))$ and $\mathfrak{s}_{r}=((\nu^{(0)},l_0),\ldots,(\nu^{(k)},l_k))$, then 
\begin{align*}
\mu^{(2a+2r)}=\mu^{(2a+2r-1)}\cup\lbrace\beta_r\rbrace\qquad\text{and}\qquad \nu^{(2a+2r)}=\nu^{(2a+2r-1)}\cup\lbrace\alpha\rbrace.
\end{align*}
Since $\mu^{(i)}=\nu^{(i)}$ for $i=2a+2r+2,\ldots,k$ we conclude that $\mu^{(2a+2r)}\rhd\mu^{(2a+2r-1)}$ and thus that the relation $\mathfrak{s}_{r-1}\succ\mathfrak{s}_r$ holds.

\eqref{obs:5.5} Assuming that $0\le r\le t$ and $\mathfrak{s}_r=((\nu^{(0)},l_0),\ldots,(\nu^{(k)},l_k))$, we have the inclusions $\nu^{(2a+2r-4)}=\nu^{(2a+2r-3)}\subsetneq \nu^{(2a+2r-2)}=\nu^{(2a+2r-1)}$. Thus $f_{\mathfrak{s}_r}\sigma_{2a+2r-2}=f_{\mathfrak{s}_r}$, by  Theorem~\ref{sigmaeven}\eqref{sigmaeven.4}. 

\eqref{obs:5.6} If $t=1$, the statement is a consequence of Theorem~\ref{sigmaodd}\eqref{sigmaodd.5}. Otherwise, by induction,  
\begin{align*}
f_\mathfrak{s}\sigma_{2a+1} \sigma_{2a+3}\cdots \sigma_{2a+2t-1} &=f_\mathfrak{t}+\sum_{\mathfrak{u}\succ\mathfrak{t}} r'_\mathfrak{u}f_\mathfrak{u}+ \sigma_{2a+1}(\mathfrak{s,s})f_\mathfrak{s}\sigma_{2a+3}\cdots \sigma_{2a+2t-1}.
\end{align*}
We focus our attention on the term $f_\mathfrak{s}\sigma_{2a+3}\cdots \sigma_{2a+2t-1}$. There two cases to consider. 

{\textsc{Case~1.}} If $\lbrace \beta_2,\beta_3,\cdots, \beta_t\rbrace$ all lie in the same row or the same column, Theorem~\ref{sigmaodd}\eqref{sigmaodd.4} implies that either $f_\mathfrak{s}\sigma_{2a+3}\cdots \sigma_{2a+2t-1} = f_\mathfrak{s}$ or $f_\mathfrak{s}\sigma_{2a+3}\cdots \sigma_{2a+2t-1} = -f_\mathfrak{s}$ and we are done. 

{\textsc{Case~2.}} If $\lambda_{j+1}=s$ and $\lambda_{j+2}>0$, then Theorem~\ref{sigmaodd}\eqref{sigmaodd.4} gives 
\begin{align*}
f_\mathfrak{s}\sigma_{2a+3}\sigma_{2a+5}\cdots \sigma_{2a+2t-1}&= f_\mathfrak{s}\sigma_{2a+2s+1}\sigma_{2a+2s+3}\cdots \sigma_{2a+2t-1},
\end{align*} 
where
\begin{align}\label{expand}
f_\mathfrak{s}\sigma_{2a+2s+1}\sigma_{2a+2s+3}\cdots \sigma_{2a+2t-1}=\sum_{\mathfrak{v}} r''_\mathfrak{v}f_\mathfrak{v}.
\end{align}
Firstly, we  show that if $\mathfrak{v}=((\rho^{(0)},r_0),\ldots,(\rho^{(k)},r_k))$ and $r''_\mathfrak{v}\ne0$ in the expression~\eqref{expand}, then $\lambda\ne \rho^{(k-1)}\cup\lbrace a\rbrace$. Using the maximality of $\mathfrak{s}$ and  parts~\eqref{sigmaodd.4} and~\eqref{sigmaodd.5} of Theorem~\ref{sigmaodd}, we observe that $\rho^{(2i-3)}\subsetneq\rho^{(2i-2)}=\rho^{(2i-1)}\subsetneq\rho^{(2i)}$ for $i=l+2,l+3,\ldots,a+t$. Hence
\begin{align*}
\big\lbrace i\mid 1\le i\le k\text{ and } \rho^{(i)}\subsetneq\rho^{(i-1)}\big\rbrace=\emptyset.
\end{align*}
Now note that $L_{i}$ and  $\sigma_{2a+2s+1}\sigma_{2a+2s+3}\cdots \sigma_{2a+2t-1}$ commute for $i=1,\ldots,2a$. Hence the separating property of the elements $\lbrace L_1,\ldots,L_{2a}\rbrace$ implies that $\mathfrak{v}\downarrow_{2a}=\mathfrak{s}\downarrow_{2a}$. It follows that $\rho^{(2a)}=\rho^{(2a-1)}\cup\lbrace\alpha\rbrace$ and $\lambda\ne\rho^{(k-1)}\cup\lbrace \alpha\rbrace$. Since $\lambda=\rho^{(k-1)}\cup \lbrace\beta_i\rbrace$, where $s\le i\le t$, we have $\rho^{(k-1)}\rhd\upsilon$ and $\mathfrak{v}\succ\mathfrak{t}$. This completes the proof of the lemma. 
\end{proof}
If $\lambda=(\lambda_1,\lambda_2,\ldots,\lambda_t)$ is a partition, let $\lambda!=\prod_{r=1}^t(\lambda_r!)$. 
\begin{lemma}\label{factorial}
If $(\lambda,l)\in\hat{A}_k$, then $\langle f_{\mathfrak{t}^{(\lambda,l)}},f_{\mathfrak{t}^{(\lambda,l)}}\rangle= \lambda!z^l$. 
\end{lemma}
\begin{proof}
The statement follows from the definition of $a_{(\lambda,l)}$ and the corresponding result~\cite[Lemma~3.41]{MR1711316} for the group algebra of the symmetric group. 
\end{proof}
\begin{proof}[Proof of Proposition~\ref{branching}] Let $\mathfrak{s}\in\hat{A}_{k+1}^{(\mu,m)}$ and  $\mathfrak{s}\downarrow_{k}=\mathfrak{s}'\in\hat{A}_k^{(\lambda,l)}$. We compute the quotient on the left hand side of the expression~\eqref{quotient} in four separate cases below.

{\textsc{Case 1.}} Assume that $k$ is even and $(\mu,l)=(\lambda,l)$.  By Lemma~\ref{observation:0}, there is no loss of generality in assuming that $\mathfrak{s}$ is maximal in $\hat{A}_{k+1}^{(\mu,m)}$ and  $\mathfrak{s'}$ is maximal in $\in\hat{A}_{k}^{(\mu,m)}$. Using Lemma~\ref{factorial}, we have
$\langle f_\mathfrak{s},f_\mathfrak{s}\rangle = \langle f_\mathfrak{s'},f_\mathfrak{s'}\rangle$, which verifies the relation~\eqref{quotient} in this case. 

{\textsc{Case 2.}} Assume that $k$ is odd and let $\mathfrak{s}=((\mu^{(0)},m_0),\ldots,(\mu^{(k+1)},m_{k+1}))\in\hat{A}_{k+1}^{(\mu,m)}$, where $(\mu^{(k)},m_k)=(\lambda,m)$,  $\mu=\lambda\cup\lbrace \alpha\rbrace$ and  $\alpha=(j,\mu_j)$. If $\mathfrak{s}$ is maximal in $\hat{A}_{k+1}^{(\lambda,l)}$, then $\mu_{j+1}=0$; hence  $\langle f_\mathfrak{s},f_\mathfrak{s}\rangle =\mu_j\langle f_\mathfrak{s'},f_\mathfrak{s'}\rangle$ and there is nothing further to show. We therefore proceed by induction. There is no loss of generality in assuming that $\mu_{j+1}>0$ and $\mathfrak{s'}$ is maximal in $\hat{A}_k^{(\lambda,m)}$. Since $\mu^{(k-2)}\subsetneq \mu^{(k-1)}=\mu^{(k)}\subsetneq \mu^{(k+1)}$, we may assume that $\mu^{(k+1)}\ominus \mu^{(k-2)}=\lbrace\alpha,\beta\rbrace$, where $\beta=(r,\mu_r)$ and $\mu_{r+1}=0$. By the hypotheses on $\mathfrak{s'}$, the nodes $\alpha$ and $\beta$ are neither in the same row nor in the same column. Hence there exists $\mathfrak{t}=\mathfrak{s}\sigma_k\in\hat{A}_{k+1}^{(\mu,m)}$ given by $\mathfrak{t}\downarrow_{k-2}=\mathfrak{s}\downarrow_{k-2}$ and $(\mathfrak{t}^{(k-1)},\mathfrak{t}^{(k)},\mathfrak{t}^{(k+1)})=((\upsilon,m),(\upsilon,m),(\mu,m))$, where $\upsilon=\mu^{(k-2)}\cup\lbrace \alpha \rbrace $ and $\mu=\upsilon\cup\lbrace\beta\rbrace$. Since $j<r$, we conclude that $\mathfrak{t}\succ\mathfrak{s}$. Let $\mathfrak{t}'=\mathfrak{t}\downarrow_{k}$.  By Lemma~\ref{observation:2} and induction on $\succcurlyeq$, we have 
\begin{align}\label{subst}
\frac{\langle f_\mathfrak{s},f_\mathfrak{s}\rangle}{\langle f_\mathfrak{t'},f_\mathfrak{t'}\rangle}= {\mu_r}\frac{(c(\alpha)-c(\beta))^2-1}{(c(\alpha)-c(\beta))^2}.
\end{align}
Let $\mathfrak{u}=\mathfrak{s}\downarrow_{k-2}$. By the first case and induction on $k$, we obtain
\begin{align*}
\frac{\langle f_\mathfrak{s'},f_\mathfrak{s'}\rangle}{\langle f_\mathfrak{u},f_\mathfrak{u}\rangle}=\mu_r\qquad\text{and}\qquad \frac{\langle f_\mathfrak{t'},f_\mathfrak{t'}\rangle}{\langle f_\mathfrak{u},f_\mathfrak{u}\rangle}=\frac{\prod_{\delta\in A(\upsilon)^{<\alpha}}(c(\alpha)-c(\delta))}{\prod_{\delta\in R(\upsilon)^{<\alpha}}(c(\alpha)-c(\delta))},
\end{align*}
which, together with the relation~\eqref{subst}, gives the required result.

{\textsc{Case 3.}} Assume that $k$ is even,  $m=l+1$ and $\lambda=\mu\cup\lbrace(j,\lambda_j)\rbrace$. By Lemma~\ref{observation:0}, there is no loss of generality in assuming that $\mathfrak{s'}=\mathfrak{t}^{(\lambda,l)}\in\hat{A}_k^{(\lambda,l)}$.  Let $(\rho,l)\in\hat{A}_k$ be minimal with respect to the property that $(\rho,l)\to(\mu,m)$ is an edge in $\hat{A}$ and define $\mathfrak{u}\in\hat{A}_{k+1}^{(\mu,m)}$ by the condition that $\mathfrak{u}'=\mathfrak{u}\downarrow_{k}=\mathfrak{t}^{(\rho,l)}\in\hat{A}_{k}^{(\rho,l)}$. Writing $a=l+\sum_{r=0}^{j}\mu_r$, we have 
\begin{align*}
p_\mathfrak{u}=w_{m,i}e_k\qquad\text{and}\qquad p_\mathfrak{s}=w_{m,i}e_kw_{i,a}\sum_{r=0}^{\mu_j}w_{a,a-r}.
\end{align*}
Thus
\begin{align*}
m_\mathfrak{us}=a_{(\mu,m-1)}^{(k-1)}e_kw_{i,a}\sum_{r=0}^{\mu_j}w_{a,a-r}=e_kw_{i,a}a_{(\lambda,l)}^{(k)}.
\end{align*}
Furthermore, if $p\in A_k$ and $a_{(\lambda,l)}^{(k)}p\in A_k^{\rhd(\lambda,l)}$, then  Proposition~\ref{filtration} implies the relation
\[
e_kw_{i,a}a_{(\lambda,l)}^{(k)}p\equiv0\mod A_{k+1}^{\rhd(\mu,m)}.
\] 
Let $\Psi_{\mathfrak{s},k}=\Psi_{\mathfrak{s},k}(L_{k+1})$ be a polynomial in $L_{k+1}$, such that $F_\mathfrak{s}=F_\mathfrak{s'}\Psi_{\mathfrak{s},k}$. Since 
\[
m_\mathfrak{us}F_\mathfrak{s}m_\mathfrak{su}\equiv \langle f_\mathfrak{s},f_\mathfrak{s}\rangle m_\mathfrak{uu}\mod A_{k+1}^{\rhd(\mu,m)}, 
\]
we use the relation $F_\mathfrak{s}=\langle f_\mathfrak{s},f_\mathfrak{s}\rangle^{-1}F_\mathfrak{ss}$ to obtain  
\begin{align}
m_\mathfrak{us}F_\mathfrak{s}m_\mathfrak{su}&\equiv e_kw_{i,a} a_{(\lambda,l)}^{(k)}F_{s'}a_{(\lambda,l)}\Psi_{\mathfrak{s},k}w_{a,i}e_k\notag\\
&\equiv \langle f_\mathfrak{s'},f_\mathfrak{s'}\rangle e_kw_{i,a}a_{(\lambda,l)}^{(k)}F_\mathfrak{s'} \Psi_{\mathfrak{s},k}w_{a,i}e_k \notag \\ 
&\equiv \langle f_\mathfrak{s'},f_\mathfrak{s'}\rangle e_kw_{i,a}F_\mathfrak{s's'}\Psi_{\mathfrak{s},k}  w_{a,i}e_k \notag \\
&\equiv {\langle f_\mathfrak{s'},f_\mathfrak{s'}\rangle^2} e_kw_{i,a}F_\mathfrak{s} w_{a,i}e_k \notag \\
&\equiv {\langle f_\mathfrak{s'},f_\mathfrak{s'}\rangle^2}{\langle f_\mathfrak{s},f_\mathfrak{s}\rangle}^{-1} e_kw_{i,a}F_\mathfrak{ss} w_{a,i}e_k \mod A_{k+1}^{\rhd(\mu,m)}.\label{cong:1}
\end{align}
By Lemma~\ref{obs:5}, the path $\mathfrak{t}=\mathfrak{s}\sigma_{2a+1}\sigma_{2a+3}\ldots\sigma_{2i-1}$ exists in $\hat{A}_{k+1}^{(\mu,m)}$. Moreover, 
\begin{align*}
f_\mathfrak{s}w_{a,i}&=f_\mathfrak{s}s_{a}s_{a+1}\cdots s_{i-1}=f_\mathfrak{s}\sigma_{2a+1}\sigma_{2a+3}\ldots\sigma_{2i-1}= f_\mathfrak{t}+\sum_\mathfrak{a} r_\mathfrak{a}f_\mathfrak{a},
\end{align*}
where the sum is taken over $\mathfrak{a}\in\hat{A}_{k+1}^{(\mu,m)}$ such that  $\Shape(\mathfrak{a}\downarrow_{k-1})\ne (\mu,l)$. Thus the congruence~\eqref{cong:1} becomes
\begin{align*}
\langle f_\mathfrak{s},f_\mathfrak{s}\rangle m_\mathfrak{uu}
&\equiv\langle f_\mathfrak{s'},f_\mathfrak{s'}\rangle^2\langle f_\mathfrak{s},f_\mathfrak{s}\rangle^{-1} e_kF_\mathfrak{tt} e_k + \sum_\mathfrak{a,b} r_\mathfrak{a,b} e_k F_\mathfrak{ab}e_k \mod A_{k+1}^{\rhd(\mu,m)}, 
\end{align*}
where the sum is over $\mathfrak{a,b}\in\hat{A}_{k+1}^{(\mu,m)}$ such that $\Shape(\mathfrak{a}\downarrow_{k-1})\ne(\mu,l)$ and $\Shape(\mathfrak{b}\downarrow_{k-1})\ne(\mu,l)$. Let $\mathfrak{v}=\mathfrak{u}\downarrow_{k-1}$. We may now use Proposition~\ref{nonzero} and the observation that $\mathfrak{t}\downarrow_{k-1}=\mathfrak{u}\downarrow_{k-1}$ to obtain  
\begin{align}
\langle f_\mathfrak{s},f_\mathfrak{s}\rangle m_\mathfrak{uu}
&\equiv\langle f_\mathfrak{s'},f_\mathfrak{s'}\rangle^2\langle f_\mathfrak{s},f_\mathfrak{s}\rangle^{-1}\langle f_\mathfrak{t},f_\mathfrak{t}\rangle   e_kF_\mathfrak{t}e_k\notag \\
&\equiv\langle f_\mathfrak{s'},f_\mathfrak{s'}\rangle^2\langle f_\mathfrak{s},f_\mathfrak{s}\rangle^{-1}\langle f_\mathfrak{t},f_\mathfrak{t}\rangle e_k(\mathfrak{t,t})F_\mathfrak{v}e_k\notag \\
&\equiv \langle f_\mathfrak{s'},f_\mathfrak{s'}\rangle^2\langle f_\mathfrak{s},f_\mathfrak{s}\rangle^{-1}\langle f_\mathfrak{t},f_\mathfrak{t}\rangle \langle f_\mathfrak{v},f_\mathfrak{v}\rangle^{-1}e_k(\mathfrak{t,t})F_\mathfrak{vv}e_k \notag\\
&\equiv \langle f_\mathfrak{s'},f_\mathfrak{s'}\rangle^2\langle f_\mathfrak{s},f_\mathfrak{s}\rangle^{-1}\langle f_\mathfrak{t},f_\mathfrak{t}\rangle \langle f_\mathfrak{v},f_\mathfrak{v}\rangle^{-1} e_k(\mathfrak{t,t})m_\mathfrak{uu}\label{caseiii} \mod A_{k+1}^{\rhd(\mu,m)}.
\end{align}
Let $\mathfrak{t}'=\mathfrak{t}\downarrow_{k}$. Using Lemma~\ref{observation:0} and the previous case, the relation~\eqref{caseiii} gives 
\begin{align*}
\frac{\langle f_\mathfrak{s},f_\mathfrak{s}\rangle}{\langle f_\mathfrak{s'},f_\mathfrak{s'}\rangle}
= \frac{\langle f_\mathfrak{s'},f_\mathfrak{s'}\rangle}{\langle f_\mathfrak{s},f_\mathfrak{s}\rangle}
\frac{\langle f_\mathfrak{t},f_\mathfrak{t}\rangle}{\langle f_\mathfrak{t'},f_\mathfrak{t'}\rangle} \frac{\langle f_\mathfrak{t'},f_\mathfrak{t'}\rangle}{\langle f_\mathfrak{v},f_\mathfrak{v}\rangle}e_k(\mathfrak{t,t})=\frac{\langle f_\mathfrak{t'},f_\mathfrak{t'}\rangle}{\langle f_\mathfrak{v},f_\mathfrak{v}\rangle}e_k(\mathfrak{t,t})= e_k(\mathfrak{t,t})\gamma_{\mu\to\lambda},
\end{align*}
as required.

{\textsc{Case 4.}} Let $k=2i+1$ and assume that $\Shape(\mathfrak{s'})=(\mu,m-1)$. There is no loss of generality in assuming that $\mathfrak{s'}=\mathfrak{t}^{(\mu,m-1)}\in\hat{A}_k^{(\mu,m-1)}$. Then $p_\mathfrak{s}=w_{m,i}$ and 
\begin{align*}
m_\mathfrak{ss}=a_{(\mu,m-1)}^{(k-1)}e_k=e_k a_{(\mu,m-1)}^{(k)}. 
\end{align*}
Let $\Psi_{\mathfrak{s},k}$ be a polynomial in $L_{k+1}$, where $F_\mathfrak{s}=F_\mathfrak{s'}\Psi_{\mathfrak{s},k}$. Then 
\begin{align*}
\langle f_\mathfrak{s},f_\mathfrak{s}\rangle m_\mathfrak{ss} &\equiv e_k a_{(\mu,m-1)}^{(k)} F_\mathfrak{s'} a_{(\mu,m-1)}^{(k)}\Psi_{\mathfrak{s},k}e_k \\
&\equiv \langle f_\mathfrak{s'},f_\mathfrak{s'}\rangle a_{(\mu,m-1)}^{(k-1)} e_kF_\mathfrak{s}e_k \\
&\equiv \langle f_\mathfrak{s'},f_\mathfrak{s'}\rangle e_k(\mathfrak{s,s})m_\mathfrak{ss} \mod A_{k+1}^{\rhd(\mu,m)}
\end{align*}
gives the required result.
\end{proof}
Our proof of Lemma~\ref{offdiag:2} requires the following observation.
\begin{corollary}\label{obs:7}
Let $k=2i+1$ and  $(\mu,m)\in\hat{A}_{k}$. Assume that $\mathfrak{s}=\mathfrak{t}^{(\mu,m)}\in\hat{A}_{k}^{(\mu,m)}$ and $(\lambda,m)\in\hat{A}_{k-1}$, where $\mu=\lambda\cup\lbrace(j,\mu_j)\rbrace$ and $\mu_{j+1}>0$. Let $a=m+\sum_{r=0}^j\mu_r$. Then $\mathfrak{v}=(\cdots ((\mathfrak{s}\sigma_{2a+1})\sigma_{2a+3})\cdots)\sigma_{2i-1}$ exists in $\hat{A}_{k}^{(\mu,m)}$ and 
\begin{align*}
f_\mathfrak{v}w_{a,i}=\frac{\gamma_{\lambda\to\mu}}{\mu_j}f_\mathfrak{s}+\sum_{\mathfrak{s}\succ\mathfrak{u}}r_\mathfrak{u}f_\mathfrak{u}.
\end{align*}
\end{corollary}
\begin{proof}
By Lemma~\ref{obs:5}, we have $\langle f_\mathfrak{v}w_{i,a},f_\mathfrak{s}\rangle = \langle f_\mathfrak{v},f_\mathfrak{s}w_{a,i}\rangle= \langle f_\mathfrak{v},f_\mathfrak{v}\rangle$. Therefore the coefficient of $f_\mathfrak{s}$ in the expansion of $f_\mathfrak{v}w_{i,a}$ in the seminormal basis is $\langle f_\mathfrak{v},f_\mathfrak{v}\rangle \langle f_\mathfrak{s},f_\mathfrak{s}\rangle^{-1}$. If we denote $\mathfrak{v}'=\mathfrak{v}\downarrow_{k-1}$, then the branching factors for inner products given in Proposition~\ref{branching} together with the expressions for $\langle f_\mathfrak{v'},f_\mathfrak{v'}\rangle$ and $\langle f_\mathfrak{s},f_\mathfrak{s}\rangle$ given by Lemma~\ref{factorial} yield the required statement. 
\end{proof}
For $(\lambda,l)\in\hat{A}_{k}$, let $G^{(\lambda,l)}_k$ denote the Gram matrix of the bilinear form~\eqref{f-d-1} on $\Delta^{(\lambda,l)}_{k}$ and $\det\big(G^{(\lambda,l)}_k\big)$ denote the determinant of $G^{(\lambda,l)}_k$.  We recall the definition~\eqref{gammadef} of branching factors for the inner product on cell modules.
\begin{theorem}\label{detrecursion}
Let $(\mu,m)\in\hat{A}_{k+1}$. Then 
\begin{align}\label{dets}
\det \big(G_{k+1}^{(\mu,m)}\big)=\prod_{(\lambda,l)\to(\mu,m)}\det \big(G_{k}^{(\lambda,l)}\big)\big(\gamma_{(\lambda,l)\to(\mu,m)}^{(k+1)}\big)^{\dim(\Delta_{k}^{(\lambda,l)})},
\end{align}
where the product is taken over $(\lambda,l)\in\hat{A}_{k}$ such that $(\lambda,l)\to(\mu,m)$ in $\hat{A}$.
\end{theorem}
Theorem~\ref{detrecursion} provides a partition algebra counterpart to the branching rule for discriminants associated with cell modules of Brauer algebras given by Rui and Si~\cite[Theorem~4.11]{MR2369064}. For the representations of the partition algebras which factor through the group algebra of the symmetric group, the recursion~\eqref{dets} coincides with the well-known branching formula for the determinant of a Specht module given by James and Murphy~\cite[Sect.~2]{MR541676}.

\section{Tables of Representing Matrices}\label{stables}
The tables in this section give the representing matrices for ${e}_i,\sigma_{i}$, $i=1,\ldots,k-1,$ relative to seminormal bases for selected cell modules of small rank. In each example, the set $\hat{A}_k^{(\lambda,l)}$ is linearly ordered by $\succcurlyeq$ and the bases $\lbrace f_{\mathfrak{t}_i}\mid  \mathfrak{t}_i\in\hat{A}_k^{(\lambda,l)} \rbrace$  for $\Delta_{k,\mathbb{F}}^{(\lambda,l)}$, and $(\lambda,l)\in\hat{A}_k$, are indexed so that $i>j$ whenever $\mathfrak{t}_j\succ \mathfrak{t}_i$. The representing matrices have been verified in the computer algebra package Maple. 
\vfill

\begin{table}[hhh]
\centering
\begin{tabular}{ccc}
\begin{tabular}{cc}
 \toprule
 $A_2(z)$& $\lambda=\emptyset$\\
 \midrule
 ${e}_{1}\mapsto z$ &  \\
 \bottomrule
 \end{tabular} 
 & \hspace{2em}&
 \begin{tabular}{ll}
 \toprule
 $A_{2}(z)$ & $\lambda=(1)$\\
 \midrule
 ${e}_{1}\mapsto 0 $ & \\
 \bottomrule
 \end{tabular} 
\end{tabular}
\end{table}

\begin{table}[hhh]
\centering
 \begin{tabular}{ll}
 \toprule
 $A_{3}(z)$ & $\lambda=\emptyset$\\
 \midrule
 ${e}_{1}\mapsto\begin{bmatrix}z&0\\0&0\end{bmatrix}$ & ${e}_{2}\mapsto\begin{bmatrix}\frac{1}{z}&{\textstyle\frac{z-1}{z^2}} \\1&{\textstyle\frac{z-1}{z}}\end{bmatrix}$\\ 
 \bottomrule
 \end{tabular}
\end{table}

\begin{table}[hhh]
\centering
\begin{tabular}{cc}
 \toprule
 $A_3(z)$& $\lambda=(1)$ \\
 \midrule
 ${e}_{1}\mapsto 0$ & $e_2\mapsto 0$  \\
 \bottomrule
 \end{tabular}
\end{table}

\vfill 
\newpage 

\begin{landscape}
\
\vfill 
\begin{table}[hhh]
\centering
\begin{tabular}{c}
 \begin{tabular}{cccc}
 \toprule
 $A_{4}(z)$ & $\lambda=\emptyset$\\
 \midrule\vspace{-01.5em}\\
 ${e}_{1}\mapsto\begin{bmatrix}z&0\\0&0\end{bmatrix}$ & ${e}_{2}\mapsto\begin{bmatrix}\frac{1}{z}&{\textstyle\frac{z-1}{z^2}} \\1&{\textstyle\frac{z-1}{z}}\end{bmatrix}$ &
 ${e}_{3}\mapsto\begin{bmatrix}z&0\\0&0\end{bmatrix}$ & ${\sigma}_{3}\mapsto\begin{bmatrix}1&0\\0&1\end{bmatrix}$\medskip\\
 \bottomrule
 \end{tabular}
\\
\\
 \begin{tabular}{cccc}
 \toprule
 $A_{4}(z)$& $\lambda=(1)$\\
 \midrule\vspace{-01.5em}\\ 
 ${e}_{1}\mapsto\begin{bmatrix}z&0&0\\0&0&0\\0&0&0\end{bmatrix}$ & ${e}_{2}\mapsto\begin{bmatrix}\frac{1}{z}&\frac{z-1}{z^2}&0 \\1&\frac{z-1}{z}&0\\ 0&0&0\end{bmatrix}$ & 
 ${e}_{3}\mapsto\begin{bmatrix}0&0&0\\0&\frac{z}{z-1}&\frac{z^2(z-2)}{(z-1)^2}\\ 0& 1& \frac{z(z-2)}{z-1}\end{bmatrix}$ 
 & ${\sigma}_{3}\mapsto\begin{bmatrix}0&\frac{1}{z}&\frac{z-2}{z-1}\\ 
 \frac{z}{z-1}& \frac{z-2}{z-1}&-\frac{z(z-2)}{(z-1)^2}\\
 1&-\frac{1}{z}&\frac{1}{z-1} \end{bmatrix}$\medskip\\
 \bottomrule
\end{tabular}

\\

\\
\begin{tabular}{cccc}
 \begin{tabular}{ccccc}
 \toprule
 $A_{4}(z)$ & $\lambda=(1,1)$\\
 \midrule
 ${e}_{1}\mapsto 0 $ &  ${e}_{2}\mapsto 0 $&  ${e}_{3}\mapsto 0 $& ${\sigma}_{3}\mapsto -1$\\
 \bottomrule
 \end{tabular}
 & &
\begin{tabular}{cccc}
 \toprule
 $A_{4}(z)$ & $\lambda=(2)$\\
 \midrule
 ${e}_{1}\mapsto 0 $ & 
 ${e}_{2}\mapsto 0 $&  
 ${e}_{3}\mapsto 0 $& $\sigma_{3}\mapsto 1 $ \\
 \bottomrule
 \end{tabular}
\end{tabular}

\end{tabular}
\end{table}
\vfill 
\end{landscape}
\newpage
\begin{landscape}
\ 
\vfill 
\begin{table}[hhh]
\centering
 \begin{tabular}{ccc}
 \toprule
 $A_{5}(z)$ & $\lambda=\emptyset$\\
 \midrule
 ${e}_{1}\mapsto \begin{bmatrix}z&0&0&0&0\\0&0&0&0&0\\ 0&0&z&0&0\\ 0&0&0&0&0\\0&0&0&0&0\\ \end{bmatrix}$ &  
 ${e}_{2}\mapsto{\displaystyle\begin{bmatrix}{\frac{1}{z}}&{\textstyle\frac{z-1}{z^2}}   &0&0&0\\1&{\textstyle\frac{z-1}{z}}&0&0&0\\ 0&0&\frac{1}{z}&\frac{z-1}{z^2}&0\\0&0&1&\frac{z-1}{z}&0\\0&0&0&0&0\end{bmatrix}}$ & 
 ${e}_{3}\mapsto
 \begin{bmatrix}z&0&0&0&0
 \\0&0&0&0&0
 \\0&0&0&0&0
 \\0&0&0&\frac{z}{z-1}&\frac{z^2(z-2)}{(z-1)^2}
 \\0&0&0&1&\frac{z(z-2)}{z-1}
 \end{bmatrix}$ \medskip \\ 
 ${e}_{4}\mapsto\begin{bmatrix}\frac{1}{z}&0&\frac{z-1}{z^2}&0&0\\0&\frac{1}{z}&0&\frac{z-1}{z^2}&0\\1&0&\frac{z-1}{z}&0&0\\0&1&0&\frac{z-1}{z}&0\\0&0&0&0&0\end{bmatrix}$ & 
 ${\sigma}_{3}\mapsto
 \begin{bmatrix}1&0&0&0&0
 \\0&1&0&0&0\\
 0&0&0&\frac{1}{z}&\frac{z-2}{z-1}\\
 0&0&\frac{z}{z-1}&\frac{z-2}{z-1}&-\frac{z(z-2)}{(z-1)^2}\\
 0&0&1&-\frac{1}{z}&\frac{1}{z-1}\end{bmatrix}$ &
 $\sigma_{4}\mapsto
 \begin{bmatrix}1&0&0&0&0\\0&0&0&\frac{1}{z}&\frac{z-2}{z-1}\\0&0&1&0&0\\
 0&\frac{z}{z-1}&0&\frac{z-2}{z-1}&-\frac{z(z-2)}{(z-1)^2}\\0&1&0&-\frac{1}{z}&\frac{1}{z-1}\end{bmatrix}$\medskip\\
 \bottomrule
\end{tabular}
\end{table}
\vfill

\end{landscape}

\newpage
\begin{landscape}
\ 
\vfill 
\begin{table}[hhh]
\centering
\begin{tabular}{ccc}
\toprule
$A_{5}(z)$& $\lambda=(1)$\\
\midrule
  ${e}_{1}\mapsto\begin{bmatrix}z&0&0&0&0\\0&0&0&0&0\\ 0&0&0&0&0\\ 0&0&0&0&0\\0&0&0&0&0\\ \end{bmatrix}$
  & ${e}_{2}\mapsto{\displaystyle\begin{bmatrix}{\displaystyle\frac{1}{z}}&{\textstyle\frac{z-1}{z^2}} &0&0&0\\1&{\textstyle\frac{z-1}{z}}&0&0&0\\ 0&0&0&0&0\\0&0&0&0&0\\0&0&0&0&0\end{bmatrix}}$ 
  & ${e}_{3}\mapsto
 \begin{bmatrix}
 0&0&0&0&0\\
 0&\frac{z}{z-1}&\frac{z^2(z-2)}{(z-1)^2}&0&0\\
 0&1&\frac{z(z-2)}{z-1}&0&0\\
 0&0&0&0&0\\
 0&0&0&0&0
 \end{bmatrix}$  \medskip
 \\
 ${e}_{4}\mapsto
 \begin{bmatrix}
 0&0&0&0&0\\
 0&0&0&0&0\\
 0&0&\frac{z-1}{z(z-2)}&\frac{(z-3)(z-1)}{z(z-2)^2}&\frac{1}{2}\frac{(z-1)^2}{z^2(z-2)}\\
 0&0&\frac{1}{2}&\frac{1}{2}\frac{z-3}{(z-2)}&\frac{1}{4}\frac{z-1}{z}\\
 0&0&1&\frac{z-3}{z-2}&\frac{1}{2}\frac{z-1}{z}
 \end{bmatrix}$ & 
 ${\sigma}_{3}\mapsto
 \begin{bmatrix}
 0&\frac{1}{z}&\frac{z-2}{z-1}&0&0\\
 \frac{z}{z-1}&\frac{z-2}{z-1}&-\frac{z(z-2)}{(z-1)^2}&0&0\\
 1&-\frac{1}{z}&\frac{1}{z-1}&0&0\\
 0&0&0&1&0\\
 0&0&0&0&-1
 \end{bmatrix}$& 
 $\sigma_{4}\mapsto
 \begin{bmatrix}
 1&0&0&0&0\\
 0&-\frac{1}{z-1}&\frac{z}{(z-1)^2}&\frac{z(z-3)}{(z-1)(z-2)}&-\frac{1}{2}\\
 0&\frac{z}{z(z-2)}&\frac{z^3-3z^2+2z-1}{z(z-2)(z-1)}&-\frac{z-3}{z(z-2)^2}&\frac{1}{2}\frac{z-1}{z^2(z-2)}\\
 0&\frac{1}{2}&-\frac{1}{2}\frac{1}{(z-1)}&\frac{1}{2}\frac{z-1}{(z-2)}&\frac{1}{4}\frac{z-1}{z}\\
 0&-1&\frac{1}{z-1}&\frac{z-3}{z-2}&\frac{1}{2}\frac{z+1}{z}
 \end{bmatrix}$ \medskip
 \\
 \bottomrule
\end{tabular}
\end{table}
\vfill 
\end{landscape}

\begin{landscape}
\ 
\vfill 
\begin{table}
\centering
\begin{tabular}{cccc}
\begin{tabular}{cccccc}
 \toprule
 $A_{5}(z)$ & $\lambda=(2)$\\
 \midrule
 ${e}_{1}\mapsto 0$ & ${e}_{2}\mapsto 0$ &
 ${e}_{3}\mapsto 0$& ${e}_{4}\mapsto 0$ &
 ${\sigma}_{3}\mapsto 1$ & ${\sigma}_{4}\mapsto 1 $\\
 \bottomrule
 \end{tabular}
 & & &
\begin{tabular}{cccccc}
 \toprule
 $A_{5}(z)$ & $\lambda=(1,1)$\\
 \midrule
 ${e}_{1}\mapsto0$ & ${e}_{2}\mapsto 0 $ &
 ${e}_{3}\mapsto 0$& ${e}_{4}\mapsto 0 $  &
  ${\sigma}_{3}\mapsto -1 $& ${\sigma}_{4}\mapsto 1 $ \\
 \bottomrule
 \end{tabular}
\end{tabular}
\end{table}
\begin{table}[hhh]
 \begin{tabular}{cccc}
 \toprule
 $A_{6}(z)$ & $\lambda=\emptyset$\\
 \midrule
 ${e}_{1}\mapsto\begin{bmatrix}z&0&0&0&0\\0&0&0&0&0\\ 0&0&z&0&0\\ 0&0&0&0&0\\0&0&0&0&0\\ \end{bmatrix}$& 
 ${e}_{2}\mapsto{\displaystyle\begin{bmatrix}{\displaystyle\frac{1}{z}}&{\textstyle\frac{z-1}{z^2}}   &0&0&0\\1&{\textstyle\frac{z-1}{z}}&0&0&0\\ 0&0&\frac{1}{z}&\frac{z-1}{z^2}&0\\0&0&1&\frac{z-1}{z}&0\\0&0&0&0&0\end{bmatrix}}$ & 
 ${e}_{3}\mapsto
 \begin{bmatrix}z&0&0&0&0
 \\0&0&0&0&0
 \\0&0&0&0&0
 \\0&0&0&\frac{z}{z-1}&\frac{z^2(z-2)}{(z-1)^2}
 \\0&0&0&1&\frac{z(z-2)}{z-1}
 \end{bmatrix}$ &  ${e}_{4}\mapsto\begin{bmatrix}\frac{1}{z}&0&\frac{z-1}{z^2}&0&0\\0&\frac{1}{z}&0&\frac{z-1}{z^2}&0\\1&0&\frac{z-1}{z}&0&0\\0&1&0&\frac{z-1}{z}&0\\0&0&0&0&0\end{bmatrix}$
 \medskip 
 \\
 ${e}_{5}\mapsto\begin{bmatrix}z&0&0&0&0\\0&z&0&0&0\\ 0&0&0&0&0\\0&0&0&0&0\\0&0&0&0&0\end{bmatrix}$ &
 ${\sigma}_{3}\mapsto
 \begin{bmatrix}1&0&0&0&0
 \\0&1&0&0&0\\
 0&0&0&\frac{1}{z}&\frac{z-2}{z-1}\\
 0&0&\frac{z}{z-1}&\frac{z-2}{z-1}&-\frac{z(z-2)}{(z-1)^2}\\
 0&0&1&-\frac{1}{z}&\frac{1}{z-1}\end{bmatrix}$ &
 ${\sigma}_{4}\mapsto
 \begin{bmatrix}1&0&0&0&0\\0&0&0&\frac{1}{z}&\frac{z-2}{z-1}\\0&0&1&0&0\\
 0&\frac{z}{z-1}&0&\frac{z-2}{z-1}&-\frac{z(z-2)}{(z-1)^2}\\0&1&0&-\frac{1}{z}&\frac{1}{z-1}\end{bmatrix}$ 
 & 
 ${\sigma}_{5}\mapsto\begin{bmatrix}1&0&0&0&0\\0&1&0&0&0\\0&0&1&0&0\\0&0&0&1&0\\0&0&0&0&1\end{bmatrix}$\medskip\\ 
 \bottomrule
\end{tabular}
\end{table}
\vfill 
\end{landscape}

\begin{landscape}
\
\vfill 
\begin{tabular}{ll}
 \toprule
 $A_{6}(z)$ & $\lambda=(1)$\\
 \midrule
$e_1 \mapsto  
\begin{bmatrix} z&0&0&0&0&0&0&0&0&0\\
0&0&0&0&0&0&0&0&0&0\\
0&0&z&0&0&0&0&0&0&0\\
0&0&0&0&0&0&0&0&0&0\\
0&0&0&0&0&0&0&0&0&0\\
0&0&0&0&0&z&0&0&0&0\\
0&0&0&0&0&0&0&0&0&0\\
0&0&0&0&0&0&0&0&0&0\\
0&0&0&0&0&0&0&0&0&0\\
0&0&0&0&0&0&0&0&0&0
\end{bmatrix}$
&
$e_2\mapsto \begin{bmatrix}
\frac{1}{z}&{\frac {z-1}{{z}^{2}}}&0&0
&0&0&0&0&0&0\\
1&{\frac {z-1}{z}}&0&0&0&0&0&0&0&0
\\\noalign{\medskip}0&0&\frac{1}{z} &{\frac {z-1}{{z}^{2}}}&0&0&0&0&0&0
\\
0&0&1&{\frac {z-1}{z}}&0&0&0&0&0&0
\\ 
0&0&0&0&0&0&0&0&0&0\\
0&0&0&0&0& \frac{1}{z}&{\frac {z-1}{{z}^{2}}}&0&0&0\\
0&0&0&0&0&1&{
\frac {z-1}{z}}&0&0&0\\
0&0&0&0&0&0&0&0&0&0
\\
0&0&0&0&0&0&0&0&0&0\\
0&0&0&0&0&0&0&0&0&0 
\end{bmatrix} $
\\
$e_3\mapsto \begin{bmatrix} z&0&0&0&0&0&0&0&0&0
\\
0&0&0&0&0&0&0&0&0&0\\
0&0&0&0&0&0&0&0&0&0\\
0&0&0&{\frac {z}{z-1}}&{\frac {{z}^{2}
 \left( z-2 \right) }{ \left( z-1 \right) ^{2}}}&0&0&0&0&0\\
0&0&0&1&{\frac {z \left( z-2 \right) }{z-1}}&0&0&0 &0&0\\
0&0&0&0&0&0&0&0&0&0\\
0&0&0&0&0&0&{\frac {z}{z-1}}&{\frac {{z}^{2} \left( z-2 \right) }{ \left( z-1 \right) ^{2}}}&0&0\\ 0&0&0&0&0&0&1&{\frac {z \left( 
z-2 \right) }{z-1}}&0&0\\\noalign{\medskip}0&0&0&0&0&0&0&0&0&0
\\0&0&0&0&0&0&0&0&0&0
\end{bmatrix}
$
&
$e_4\mapsto \begin{bmatrix}
\frac{1}{z}&0&{\frac {z-1}{{z}^{2}}}&0&0&0&0&0&0&0\\
0&\frac{1}{z}&0&{\frac {z-1}{{z}^{2}}}&0&0&0&0&0&0\\
1&0&{\frac {z-1}{z}}&0&0&0&0&0&0&0
\\
0&1&0&{\frac {z-1}{z}}&0&0&0&0&0&0
\\
0&0&0&0&0&0&0&0&0&0\\
0&0&0&0&0&0
&0&0&0&0\\
0&0&0&0&0&0&0&0&0&0\\
0&0
&0&0&0&0&0&{\frac {z-1}{z \left( z-2 \right) }}&{\frac { \left( z-3
 \right)  \left( z-1 \right) }{ \left( z-2 \right) ^{2}z}}& \frac{1}{2}{
\frac { \left( z-1 \right) ^{2}}{{z}^{2} \left( z-2 \right) }}\\
0&0&0&0&0&0&0&\frac{1}{2}&{\frac{1}{2}\frac {z-3}{z-2}}&{\frac{1}{4}
\frac {z-1}{z}}\\
0&0&0&0&0&0&0&1&{\frac {z-3}{z-2}}
&{\frac{1}{2}\frac {z-1}{z}}
\end{bmatrix}
$

\\
\bottomrule
\end{tabular}
\end{landscape}
\begin{landscape}
\
\vfill
\begin{table}[hhh]
\centering
\begin{tabular}{ll}
 \toprule
 $A_{6}(z)$ & $\lambda=(1)$\\
 \midrule

$e_5\mapsto 
\begin{bmatrix}
0&0&0&0&0&0&0&0&0&0
\\
0&0&0&0&0&0&0&0&0&0\\
0&0&{\frac 
{z}{z-1}}&0&0&{\frac {{z}^{2} \left( z-2 \right) }{ \left( z-1
 \right) ^{2}}}&0&0&0&0\\
 0&0&0&{\frac {z}{z-1}}&0&0&
{\frac {{z}^{2} \left( z-2 \right) }{ \left( z-1 \right) ^{2}}}&0&0&0
\\
0&0&0&0&{\frac {z}{z-1}}&0&0&{\frac {{z}^{2}
 \left( z-2 \right) }{ \left( z-1 \right) ^{2}}}&0&0
\\
0&0&1&0&0&{\frac {z \left( z-2 \right) }{z-1}}&0&0
&0&0\\
0&0&0&1&0&0&{\frac {z \left( z-2 \right) }{z-1
}}&0&0&0\\
0&0&0&0&1&0&0&{\frac {z \left( z-2
 \right) }{z-1}}&0&0\\
0&0&0&0&0&0&0&0&0&0
\\
0&0&0&0&0&0&0&0&0&0 
\end{bmatrix}
$ 
&

$\sigma_3\mapsto \begin{bmatrix} 1&0&0&0&0&0&0&0&0&0
\\
0&1&0&0&0&0&0&0&0&0\\
0&0&0&\frac{1}{z}&{\frac {z-2}{z-1}}&0&0&0&0&0\\
0&0&{\frac {z}{z-1
}}&{\frac {z-2}{z-1}}&-{\frac {z \left( z-2 \right) }{ \left( z-1
 \right) ^{2}}}&0&0&0&0&0\\
 0&0&1&-\frac{1}{z}& \frac{1}{z
-1}&0&0&0&0&0\\
0&0&0&0&0&0&\frac{1}{z}&{
\frac {z-2}{z-1}}&0&0\\
0&0&0&0&0&{\frac {z}{z-1}}&{
\frac {z-2}{z-1}}&-{\frac {z \left( z-2 \right) }{ \left( z-1 \right) 
^{2}}}&0&0\\
0&0&0&0&0&1&-\frac{1}{z}& \frac{1}{z-1}&0&0\\
 0&0&0&0&0&0&0&0&1&0
\\
0&0&0&0&0&0&0&0&0&-1
\end{bmatrix}
$
\\
\bottomrule
\end{tabular}
\end{table}
\vfill 
\end{landscape}
\begin{landscape}
\begin{table}[hhh]
\centering
\begin{tabular}{l}
 \toprule
 $A_{6}(z)$ \qquad $\lambda=(1)$\\
 \midrule
 $\sigma_4\mapsto
\begin{bmatrix} 1&0&0&0&0&0&0&0&0&0
\\
0&0&0&\frac{1}{z}&{\frac {z-2}{z-1}}&0&0&0&0&0
\\
0&0&1&0&0&0&0&0&0&0\\
0&{\frac {z
}{z-1}}&0&{\frac{z-2}{z-1}}&-{\frac {z \left( z-2 \right) }{ \left( z
-1 \right) ^{2}}}&0&0&0&0&0\\
0&1&0&-\frac{1}{z}&
\frac{1}{z-1} &0&0&0&0&0\\
 0&0&0&0&0&1&0&0
&0&0\\
0&0&0&0&0&0&-\frac{1}{z-1}&{\frac 
{z}{ \left( z-1 \right) ^{2}}}&{\frac {z \left( z-3 \right) }{
 \left( z-2 \right)  \left( z-1 \right) }}&-\frac{1}{2}\\
 0&0
&0&0&0&0&{\frac {1}{z \left( z-2 \right) }}&{\frac {{z}^{3}-3{z}^{2}
+2z-1}{z \left( z-2 \right)  \left( z-1 \right) }}&-{\frac {z-3}{
 z\left( z-2 \right) ^{2}}}&\frac{1}{2}{\frac {z-1}{{z}^{2} \left( z-2
 \right) }}\\
 0&0&0&0&0&0&\frac{1}{2}&-\frac{1}{2} \frac{1}{z-2}&\frac{1}{2}{\frac {z-1}{z-2}}&{\frac{1}{4}\frac {z-1}{z}}
\\
0&0&0&0&0&0&-1& \frac{1}{z-1} &{\frac {z-3}{z-2}}&{\frac{1}{2}\frac {z+1}{z}}
\end{bmatrix}
$\medskip\\
$\sigma_5\mapsto\begin{bmatrix}
 0&0&\frac{1}{z}&0&0&{\frac {z-2}{z-1}}
&0&0&0&0\\\noalign{\medskip}0&0&0&\frac{1}{z}&0&0&{\frac {z-2}{z-1}}&0&0&0
\\\noalign{\medskip}{\frac {z}{z-1}}&0&{\frac {z-2}{z-1}}&0&0&-{\frac 
{z \left( z-2 \right) }{ \left( z-1 \right) ^{2}}}&0&0&0&0
\\\noalign{\medskip}0&{\frac {z}{z-1}}&0&{\frac {z-2}{z-1}}&0&0&-{
\frac {z \left( z-2 \right) }{ \left( z-1 \right) ^{2}}}&0&0&0
\\\noalign{\medskip}0&0&0&0&- \frac{1}{z-1}&0&0&{\frac {z}{
 \left( z-1 \right) ^{2}}}&{\frac {z \left( z-3 \right) }{ \left( z-2
 \right)  \left( z-1 \right) }}&-\frac{1}{2}\\\noalign{\medskip}1&0&-\frac{1}{z}&0
&0&\frac{1}{z-1}&0&0&0&0\\\noalign{\medskip}0&1&0&-\frac{1}{z}&0&0& \frac{1}{z-1}&0&0&0\\\noalign{\medskip}0&0&0&0&{
\frac {1}{z \left( z-2 \right) }}&0&0&{\frac {{z}^{3}-3{z}^{2}+2z-
1}{z \left( z-2 \right)  \left( z-1 \right) }}&-{\frac {z-3}{ z\left( 
z-2 \right) ^{2}}}&\frac{1}{2}{\frac {z-1}{{z}^{2} \left( z-2 \right) }}
\\\noalign{\medskip}0&0&0&0&\frac{1}{2}&0&0&-\frac{1}{2}\frac{1}{z-1} &\frac{1}{2}
\,{\frac {z-1}{z-2}}&\frac{1}{4}{\frac {z-1}{z}}\\\noalign{\medskip}0&0&0&0&
-1&0&0& \frac{1}{z-1}&{\frac {z-3}{z-2}}&\frac{1}{2}{\frac {z+1}
{z}}
\end{bmatrix}
$
\medskip\\
\bottomrule
\end{tabular}
\end{table} 
\vfill 
\end{landscape}

\begin{landscape}

\begin{table}[hhh]
\centering
\begin{tabular}{lll}
 \toprule
 $A_{6}(z)$ \qquad $\lambda=(2)$\\
 \midrule
$
e_1\mapsto
\begin{bmatrix}
z&0&0&0&0&0\\ 
0&0&0&0&0&0\\ 
0&0&0&0&0&0\\ 
0&0&0&0&0&0\\ 
0&0&0&0&0&0\\ 
0&0&0&0&0&0
\end{bmatrix}
$ &
$e_2\mapsto 
\begin{bmatrix}
\frac{1}{z}&{\frac {z-1}{{z}^{2}}}&0&0&0&0
\\ 
1&{\frac {z-1}{z}}&0&0&0&0\\ 
0&0&0&0&0&0\\ 
0&0&0&0&0&0\\ 
0&0&0&0&0&0\\ 
0&0&0&0&0&0
\end{bmatrix}
$ &
$e_3\mapsto
\begin{bmatrix}
 0&0&0&0&0&0\\ 
 0&{
\frac {z}{z-1}}&{\frac {{z}^{2}(z-2) }{(z-1) ^{2}}}&0&0&0\\ 
0&1&{\frac {z (z-2)}{z-1}}&0&0&0\\ 
0&0&0&0&0&0\\ 
0&0&0&0&0&0\\ 
0&0&0&0&0&0
\end{bmatrix}
$\\

$
e_4\mapsto
\begin{bmatrix}
0&0&0&0&0&0\\ 
0&0&0&0&0&0\\ 
0&0&{\frac {z-1}{z (z-2) }}&{
\frac {(z-1)( z-3) }{z(z-2)^{2}}}&\frac{1}{2}\,{\frac {( z-1)^{2}}{{z}^{2}(z-2) }}&0\\ 
 0&0&\frac{1}{2}&\frac{1}{2}\,{\frac {z-3}{z-2}}&\frac{1}{4}
\,{\frac {z-1}{z}}&0\\ 
0&0&1&{\frac {z-3}{z-2}}&\frac{1}{2}
\,{\frac {z-1}{z}}&0\\ 
0&0&0&0&0&0
\end{bmatrix}
$
& 
$e_5\mapsto
\begin{bmatrix}
0&0&0&0&0&0\\ 
0&0&0&0&0&0\\ 
0&0&0&\frac{2(z-2)}{z-3}&0&{\frac {( z-1)( z-2)(z-4)}{(z-3)^{2}}}\\ 
0&0&0&0&0&0\\ 
0&0&0&2&0&{\frac {( z-1)(z-4) }{z-3}}
\end{bmatrix}
$
&
 
 $
\sigma_3\mapsto
\begin{bmatrix}
0&\frac{1}{z}&{\frac {z-2}{z-1}}&0&0&0
\\ {\frac {z}{z-1}}&{\frac {z-2}{z-1}}&-{\frac{z(z-2)}{( z-1)^{2}}}&0&0&0\\ 
1&-\frac{1}{z}& \frac{1}{z-1}&0&0&0\\ 
0&0&0&1&0&0\\ 
0&0&0&0&-1&0\\ 
0&0&0&0&0&1
\end{bmatrix}
$
\\

\bottomrule
\end{tabular}
\end{table} 

\vfill

\begin{table}[hhh]
\centering
\begin{tabular}{ll}
 \toprule
 $A_{6}(z)$ \qquad $\lambda=(2)$\\
 \midrule
$\sigma_4\mapsto
\begin{bmatrix}
1&0&0&0&0&0\\ 
0&-
 \frac{1}{z-1}&{\frac {z}{ \left( z-1 \right) ^{2}}}&{
\frac { z\left( z-3 \right) }{ \left( z-1 \right)  \left( -2+z
 \right) }}&-\frac{1}{2}&0\\ 
 0&{\frac {1}{z \left(z-2
 \right) }}&{\frac {{z}^{3}-3{z}^{2}+2z-1}{ \left( z-1 \right) z
 \left( z-2 \right) }}&-{\frac {z-3}{ \left(z-2 \right) ^{2}z}}&\frac{1}{2}
\,{\frac {z-1}{{z}^{2} \left(z-2\right) }}&0\\
0&
\frac{1}{2}&-\frac{1}{2}\, \frac{1}{z-1}&\frac{1}{2}\,{\frac {z-1}{z-2}}&\frac{1}{4}\,{
\frac {z-1}{z}}&0\\ 
0&-1& \frac{1}{z-1}&{
\frac {z-3}{z-2}}&\frac{1}{2}\,{\frac {1+z}{z}}&0\\ 
0&0&0&0&0&1
\end{bmatrix}
\medskip
$ &
$\sigma_5\mapsto
\begin{bmatrix}
1&0&0&0&0&0\\0&1&0&0&0
&0\\ 0&0&-{\frac {1}{ \left(z -2 \right) z}}&{\frac 
{{z}^{2}-2\,z+1}{ \left( z-2 \right) ^{2}z}}&\frac{1}{2}\,{\frac {{z}^{2}-2\,
z+1}{{z}^{2} \left( z-2 \right) }}&{\frac {-4+9\,z-6\,{z}^{2}+{z}^{3}
}{ \left( z-3 \right)  \left( z-2 \right) z}}\\0&0&
\frac{1}{2}\,{\frac {z-1}{z-3}}&\frac{1}{2}\,{\frac {13-11\,z+2\,{z}^{2}}{ \left( z-3
 \right)  \left( z-2 \right) }}&-\frac{1}{4}\,{\frac {z-1}{ \left( z-3
 \right) z}}&-\frac{1}{2}\,{\frac {{z}^{2}-5\,z+4}{ \left( z-3 \right) ^{2}}}
\\ 0&0&1&- \frac{1}{ z-2 }&\frac{1}{2}\,{\frac {2\,
z-1}{z}}&-{\frac {z-4}{z-3}} \\ 0&0&1&- \frac{1}{ z-2 }&-\frac{1}{2z}& \frac{1}{ z-3 }
\end{bmatrix}
$
\\

\bottomrule
\end{tabular}
\end{table} 
\end{landscape}

\begin{landscape}
\begin{table}[hhh]
\centering
\begin{tabular}{lll}
 \toprule
 $A_{6}(z)$ \qquad $\lambda=(1,1)$\\
 \midrule
$e_1\mapsto
\begin{bmatrix}
z&0&0&0&0&0\\ 
0&0&0&0&0&0\\ 
0&0&0&0&0&0\\ 
0&0&0&0&0&0\\ 
0&0&0&0&0&0\\ 
0&0&0&0&0&0
\end{bmatrix}
$
&
$e_2\mapsto
\begin{bmatrix}
\frac{1}{z}&{\frac {z-1}{{z}^{2}}}&0&0&0&0\\ 
1&{\frac {z-1}{z}}&0&0&0&0\\ 
0&0&0&0&0&0\\ 
0&0&0&0&0&0\\ 
0&0&0&0&0&0\\ 
0&0&0&0&0&0
\end{bmatrix}
$
&
$e_3\mapsto
\begin{bmatrix}
0&0&0&0&0&0\\ 
0&{
\frac {z}{z-1}}&{\frac {{z}^{2} \left(z-3 \right) }{ \left( z-1
 \right) ^{2}}}&0&0&0\\ 
 0&1&{\frac { \left(z-3 \right) z}{z-1}}&0&0&0\\ 
 0&0&0&0&0&0
\\ \noalign{\medskip}0&0&0&0&0&0\\ 
0&0&0&0&0&0
\end{bmatrix}
$
\\
$e_4\mapsto
\begin{bmatrix}
0&0&0&0&0&0\\ 
0&0&0&0&0&0\\ 
0&0&{\frac {z-1}{ \left( z-3 \right) z}}&{
\frac { \left( z-3 \right)  \left( z-1 \right) }{ \left(z-2
 \right) ^{2}z}}&\frac{1}{2}\,{\frac { \left( z-1 \right) ^{2}}{{z}^{2}
 \left( z-2 \right) }}&0\\ 
 0&0&\frac{1}{2}&\frac{1}{2}\,{\frac {z-3}{z-3}}&\frac{1}{4}\,{\frac {z-1}{z}}&0\\ 
0&0&1&{\frac {z-3}{z-2}}&\frac{1}{2}\,{\frac {z-1}{z}}&0\\ 
0&0&0&0&0&0
\end{bmatrix}
$
&
$e_5\mapsto
\begin{bmatrix}
0&0&0&0&0&0\\ 
0&0&0&0&0&0\\ 
0&0&0&0&0&0\\ 
0&0&0&0&{\frac {2z}{z-1}}&{\frac {2{z}^{2}
 \left( z-3 \right) }{ \left( z-1 \right) ^{2}}}\\ 
0&0&0&0&1&{\frac { \left(z-3 \right) z}{z-1}}
\end{bmatrix}
\medskip
$ &
$\sigma_3\mapsto
\begin{bmatrix}
0&\frac{1}{z}&{\frac {z-2}{z-1}}&0&0&0
\\ 
{\frac {z}{z-1}}&{\frac {z-2}{z-1}}&-{\frac {
 \left( z-2 \right) z}{ \left( z-1 \right) ^{2}}}&0&0&0
\\ 
1&-\frac{1}{z}& \frac{1}{ z-1} &0&0&0
\\ 
0&0&0&1&0&0\\ 
0&0&0&0&-1&0 \\ 
0&0&0&0&0&-1
\end{bmatrix}
$\\

\bottomrule
\end{tabular}
\end{table}

\vfill 
\begin{table}[hhh]
\centering
\begin{tabular}{ll}
 \toprule
 $A_{6}(z)$ \qquad $\lambda=(1,1)$\\
 \midrule
 
$\sigma_4\mapsto
\begin{bmatrix}
1&0&0&0&0&0\\ 
0&- \frac{1}{z-1}&{\frac {z}{ \left( z-1 \right) ^{2}}}&{
\frac { \left( z-3 \right) z}{ \left(z-2 \right)  \left( z-1
 \right) }}&-\frac{1}{2}&0\\ 
 0&{\frac {1}{ \left(z-2 \right) z}}&{\frac {{z}^{3}-3{z}^{2}+2z-1}{ \left(z-2 \right) 
 \left( z-1 \right) z}}&-{\frac {z-3}{ \left( z-2 \right) ^{2}z}}&\frac{1}{2}{\frac {z-1}{{z}^{2} \left( z-2 \right) }}&0\\ 
0&\frac{1}{2}&-\frac{1}{2}\, \frac{1}{z-1}&\frac{1}{2}\,{\frac {z-1}{z-2}}&\frac{1}{4}\,{
\frac {z-1}{z}}&0\\ 
0&-1& \frac{1}{z-1} &{
\frac {z-3}{z-2}}&\frac{1}{2}\,{\frac {z+1}{z}}&0\\ 
0&0&0&0&0&1
\end{bmatrix}
$ &
$\sigma_5\mapsto
\begin{bmatrix}
 -1&0&0&0&0&0\\ 
 0&-1&0&0&0&0\\ 
 0&0&{\frac {1}{ \left( z-2 \right) z}}&{
\frac {{z}^{2}-4z+3}{ \left(z-2 \right)^{2}z}}&\frac{1}{2}\,{\frac {{z}^{
2}-1}{{z}^{2} \left(z-2\right) }}&{\frac {z-3}{z-2}}
\\ 
0&0&\frac{1}{2}&-\frac{1}{2}\, \frac{1}{z-2}&\frac{1}{4}\,{
\frac {2\,z-1}{z}}&-\frac{1}{2}\,{\frac {z}{z-1}}\\ 
0&0&{
\frac {z+1}{z-1}}&{\frac {2{z}^{2}-7z+3}{ \left(z-2 \right) 
 \left( z-1 \right) }}&-\frac{1}{2}\,{\frac {1+z}{ \left( z-1 \right) z}}&-{
\frac { z\left(z-3\right) }{ \left( z-1 \right) ^{2}}}
\\ 
0&0&1&- \frac{1}{z-2}&-\frac{1}{2z}&
\frac{1}{ z-1}
\end{bmatrix}
$

\medskip\\
\bottomrule
\end{tabular}
\end{table}

\end{landscape}
 
\newpage

\vfill

\begin{table}[hhh]
\centering
\begin{tabular}{cccc}
 \toprule
 $A_{6}(z)$ & $\lambda=(3)$\\
 \midrule
 $e_1\mapsto 0$ & $e_2\mapsto 0$ & $e_3\mapsto 0$ & $e_4\mapsto 0$\\
$e_5\mapsto 0$ & $\sigma_3\mapsto 1$ &$\sigma_4\mapsto 1$ &$\sigma_5\mapsto 1$
\medskip\\
\bottomrule
\end{tabular}
\end{table}

\begin{table}[hhh]
\centering
\begin{tabular}{cccc}
 \toprule
 $A_{6}(z)$ & $\lambda=(2,1)$ \\
 \midrule
 $e_1\mapsto 
 \begin{bmatrix}
 0 & 0 \\ 0 & 0
 \end{bmatrix}$ & 
 $e_2\mapsto  
 \begin{bmatrix}
 0 & 0 \\ 0 & 0
 \end{bmatrix}$ & 
 $e_3\mapsto  
 \begin{bmatrix}
 0 & 0 \\ 0 & 0
 \end{bmatrix}$ & 
 $e_4\mapsto  
 \begin{bmatrix}
 0 & 0 \\ 0 & 0
 \end{bmatrix}$\medskip \\
 $e_5\mapsto  
 \begin{bmatrix}
 0 & 0 \\ 0 & 0
 \end{bmatrix}$ & 
 $\sigma_3\mapsto  
 \begin{bmatrix}
 1 & 0 \\ 0 & -1
 \end{bmatrix}$ &
 $\sigma_4\mapsto  
 \begin{bmatrix}
 1 & 0 \\ 0 & 1
 \end{bmatrix}$ &
 $\sigma_5\mapsto 
  \begin{bmatrix}
 -\frac{1}{2} & \frac{3}{2} \\ 1 & \frac{1}{2}
 \end{bmatrix}$
\medskip\\
\bottomrule
\end{tabular}
\end{table}

\begin{table}[hhh]
\centering
\begin{tabular}{cccc}
 \toprule
 $A_{6}(z)$ & $\lambda=(1,1,1)$\\
 \midrule
 $e_1\mapsto 0$ & $e_2\mapsto 0$ & $e_3\mapsto 0$ & $e_4\mapsto 0$\\
$e_5\mapsto 0$ & $\sigma_3\mapsto -1$ &$\sigma_4\mapsto 1$ &$\sigma_5\mapsto -1$
\medskip\\
\bottomrule
\end{tabular}
\end{table}


\bibliographystyle{amsplain}

\bibliography{seminorm}


\end{document}